\documentclass[12pt]{article}

\usepackage{amsmath,amsthm,amsfonts,amssymb, color,xcolor,subcaption,graphicx} 

\newtheorem{theorem}{Theorem}[section]
\newtheorem{corollary}[theorem]{Corollary}
\newtheorem{example}[theorem]{Example}

\newtheorem{proposition}[theorem]{Proposition}
\newtheorem{lemma}[theorem]{Lemma}
\newtheorem{definition}[theorem]{Definition}
\newtheorem{remark}[theorem]{Remark}

\def\cA{\mathcal{A}}
\def\cB{\mathcal{B}}
\def\cC{\mathcal{C}}
\def\cD{\mathcal{D}}
\def\cE{\mathcal{E}}
\def\cF{\mathcal{F}}

\def\cH{\mathcal{H}}
\def\cI{\mathcal{I}}

\def\cL{\mathcal{L}}
\def\cM{\mathcal{M}}
\def\cN{\mathcal{N}}
\def\cP{\mathcal{P}}
\def\cR{\mathcal{R}}

\def\cS{\mathcal{S}}

\def\cX{\mathcal{X}}
\def\cZ{\mathcal{Z}}

\def\bC{\mathbb{C}}

\def\bE{\mathbb{E}}

\def\bN{\mathbb{N}}

\def\bR{\mathbb{R}}

\newcommand{\fH}{\mathfrak{H}}

\def\k{\kappa}

\def\s{\sigma}

\newcommand{\fm}{\mathfrak{m}}

\begin{document}

\title{Moment estimates for solutions of SPDEs \\
with L\'evy colored noise}

\author{Raluca M. Balan\footnote{Corresponding author. University of Ottawa, Department of Mathematics and Statistics, 150 Louis Pasteur Private, Ottawa, Ontario, K1N 6N5, Canada. E-mail address: rbalan@uottawa.ca.} \footnote{Research supported by a grant from the Natural Sciences and Engineering Research Council of Canada.}
\and
Juan J. Jim\'enez\footnote{Auburn University, Department of Mathematics and Statistics, 221 Parker Hall, Auburn, Alabama 36849, USA. E-mail address: juj0003@auburn.edu}
}

\date{January 9, 2026}
\maketitle

\begin{abstract}
\noindent In this article, we continue the investigations initiated by the first author in \cite{B15} related to the study of stochastic partial differential equations (SPDEs) with L\'evy colored noise on $\bR_{+} \times \bR^d$. This noise is constructed from a L\'evy white noise on $\bR_{+} \times \bR^d$ (which is in turn built from a Poisson random measure on $\bR_{+} \times \bR^d \times \bR_0$ with intensity $dtdx \nu(dz)$), using the convolution with a suitable spatial kernel $\k$. We assume that the L\'evy measure $\nu$ has finite variance. Therefore, the stochastic integral with respect to this noise is constructed similarly to the integral with respect to the spatially-homogeneous Gaussian case considered in \cite{dalang99}. Using Rosenthal's inequality, we provide an upper bound for the $p$-th moment of the stochastic integral with respect to the L\'evy colored noise, which allows us to identify sufficient conditions for the solution of an SPDE driven by this noise to have higher order moments. We first analyze this question for the linear SPDE (in which the noise enters in an additive way), considering as examples the stochastic heat and wave equations in any dimension $d$, for three examples of kernels $\k$: the heat kernel, the Riesz kernel, and the Bessel kernel. Then, we present a general theory for a non-linear SPDE with Lipschitz coefficients, and perform a detailed analysis in the case of the heat equation (in dimension $d\geq 1$), and wave equation (in dimension $d\leq 3$), for the same kernels $\k$. We show that the solution of each of these equations has a finite upper Lyapounov exponent of order $p\geq 2$, and in some cases, is weakly intermittent (in the sense of \cite{FK13}). In the case of the parabolic/hyperbolic Anderson model with L\'evy colored noise, we provide the Poisson chaos expansion of the solution and the explicit form of the second-order Lyapounov exponent. 
\end{abstract}


\vspace{1mm}


\pagebreak

\tableofcontents

\section{Introduction}

The study of stochastic partial differential equations (SPDEs) is an active area of research in probability theory, which has been growing tremendously in the last four decades, since the publication of Walsh' lecture notes \cite{walsh86}. Many investigations focus on equations driven by Gaussian noise, which may be white or colored in space or time.

A significant portion of the literature is dedicated to equations driven by a pure-jump {\em space-time L\'evy white noise}, which does not have a Gaussian component. This noise is given by a collection $\{\Lambda(A); A \in \cB_b(\bR_{+}\times \bR^d)\}$ of random variables indexed by the bounded Borel subsets of $\bR_{+} \times \bR^d$, and is defined by: 
\[
\Lambda(A)=b|A|+\int_{A \times \{|z|\leq 1\}}z \widehat{N}(dt,dx,dz)+\int_{A \times \{|z|> 1\}}z N(dt,dx,dz),
\]
where $b \in \bR$, $N$ is a Poisson random measure (PRM) on $\bR_{+}\times \bR^d \times \bR_0$ of intensity $\mu(dt,dx,dz)=dtdx \nu(dz)$, $\widehat{N}(F)=N(F)-\mu(F)$ is the compensated version of $N$, for Borel sets $F$ in $\bR_{+}\times \bR^d \times \bR_0$, and $\nu$ is a {\em L\'evy measure} on $\bR$ i.e. $\nu(\{0\})=0$ and 
\begin{equation}
\label{nu-Levy}
\quad \int_{\bR_0}(|z|^2 \wedge 1)\nu(dz)<\infty.
\end{equation}

In this theory, $\bR_{0}:=\bR \verb2\2 \{0\}$ is equipped with the distance $d(x,y)=|\frac{1}{x}-\frac{1}{y}|$, so that ``bounded sets'' in $\bR_0$ are in fact those bounded away from $0$. One motivation for this comes from the symmetric $\alpha$-stable L\'evy motion $\{Z(t)\}_{t \in [0,1]}$, which has L\'evy measure:
\[
\nu_{\alpha}(dz)=\frac{1}{2}\alpha|z|^{-\alpha-1} 1_{\{|z|>0\}} dz \quad \mbox{with} \quad \alpha \in (0,2),
\]
and for which the underlying PRM has points of the form $(T_i,Z_i=\varepsilon_i \Gamma_i^{-1/\alpha})$ where $(T_i)_{i\geq 1}$ are i.i.d. uniform on $[0,1]$, $(\varepsilon_i)_{i\geq 1}$ are Rademacher variables, and $(\Gamma_i)_{i\geq 1}$ are the points of a Poisson process on $[0,\infty)$ of intensity 1. Since $|Z_i| \to 0$ as $i \to \infty$, and a PRM has a finite number of points in any compact set, we must exclude 0 from the ``jumps space''.

\medskip

The moments of $\Lambda(A)$ are closely related to integrability properties of $\nu$: for any $p>0$, $\bE|\Lambda(A)|^p<\infty$ if and only if $\int_{|z|>1}|z|^p \nu(dz)<\infty$. Therefore, if we assume that
\begin{equation}
\label{finite-m2}
m_2:=\int_{\bR_0}|z|^2 \nu(dz)<\infty,
\end{equation}
then $\Lambda(A)$ has finite variance.
If in addition, we choose $b=-\int_{|z|>1}z \nu(dz)$, we obtain a centered process:
\begin{equation}
\label{def-L}
L(A)=\int_{A \times \bR_{0}} z \widehat{N}(dt,dx,dz), \quad  A \in \cB_b(\bR_{+}\times \bR^d).
\end{equation}
Since
\[
\bE\big[L(A)L(B)\big]=m_2 |A\cap B|,
\]
the theory of stochastic integration with respect to $L$ is the same as for the space-time Gaussian white noise $W$. Moreover, SPDEs with noise $L$ share many similarities with the corresponding equations driven by $W$, with some notable exceptions, such as: the existence of higher-order moments, and the regularity of the sample paths. {\em Throughout this work, we will assume that \eqref{finite-m2} holds, and we let $L$ be defined by \eqref{def-L}.}

\medskip

We start by recalling the major contributions for SPDEs with general L\'evy white noise, for which the techniques are very different compared to the Gaussian case. Then we return to the finite variance case, which is the topic of the present article. Note that for our literature review, we mention only the references based on Walsh's random field approach. For a different approach to SPDEs with L\'evy noise, we refer the reader to the monograph \cite{PZ07} by Peszat and Zabczyk, and the references therein.

\medskip

The {\em stochastic heat equation} with general L\'evy noise: 
\begin{equation}
\label{SHE}
\frac{\partial u}{\partial t}(t,x)=\frac{1}{2}\Delta u(t,x)+\sigma\big(u(t,x)\big)\dot{\Lambda}(t,x), \quad t>0,x\in\bR^d
\end{equation}
was studied for the first time in \cite{bie98}, where it was shown that if $\sigma$ is Lipschitz, \eqref{SHE} has a unique solution, and this solution has a finite $p$-th moment, provided that $p<1+\frac{2}{d}$ and 
\begin{equation}
\label{finite-z}
m_p:=\int_{\bR_0}|z|^p \nu(dz)<\infty.
\end{equation}
Condition \eqref{finite-z} is not satisfied by the $\alpha$-stable L\'evy noise, with $\alpha \in (0,2)$. In \cite{mueller98}, Mueller proved the existence of a solution of \eqref{SHE} driven by an $\alpha$-stable L\'evy noise with $\alpha \in (0, 1)$, and a non-Lipschitz function $\sigma(u) = u^{\gamma}$ with $\gamma > 0$. The same problem for the case $\alpha \in (1, 2)$ was treated by Mytnik in \cite{mytnik02}. The compact support property for solutions to stochastic
heat equations with $\alpha$-stable L\'evy noise has been recently established in \cite{hughes25}.

The problem of existence of a solution of a general SPDE with $\alpha$-stable L\'evy noise and
Lipschitz function $\sigma$ remained open until 2014. In \cite{B14}, the first author of this paper established the existence of a solution of a large class of SPDEs on bounded domains, driven by the $\alpha$-stable L\'evy noise. This was achieved by solving first the equation with truncated noise, obtained by removing the ``large jumps'' that exceed a flat level $N$.

In 2017, Chong \cite{chong17-SPA}, introduced a space-dependent truncation technique, and improved significantly the results of \cite{B14,bie98} by showing the existence of a solution of \eqref{SHE}. 
His condition: 
\begin{equation}
\label{pq-con}
\int_{\{|z| \leq 1\}} |z|^p \nu(dz)<\infty  \quad \mbox{and} \quad
\int_{\{|z|>1\}}|z|^q \nu(dz) < \infty \quad \mbox{for some $0 < q \leq p$},
\end{equation}
uses different exponents $p$ and $q$ for the ``small jumps'' region $\{|z|\leq 1\}$ versus the ``large jumps'' region $\{|z|> 1\}$, and is satisfied by the $\alpha$-stable L\'evy noise.
Subsequently, the authors of \cite{CDH19} proved the regularity of the
sample paths of the solution of \eqref{SHE} as a process with values in a Sobolev space. 
The uniqueness of the solution of \eqref{SHE} was proved in \cite{BCL}, in the case when $\sigma(u) = \lambda u$ for some $\lambda > 0$, and $\nu(-\infty, 0) = 0$. As far as we know, for a general Lipschitz function $\sigma$, the uniqueness of the solution remains an open problem.

\medskip

In the hyperbolic case, the first author of this paper established  in \cite{B23} the existence of a solution of the {\em stochastic wave equation}:
\begin{equation}
\label{SWE}
\frac{\partial^2 u}{\partial t^2}(t,x)=\Delta u(t,x)+\sigma\big(u(t,x)\big)\dot{\Lambda}(t,x), \quad t>0,x\in\bR^d \quad (d\leq 2)
\end{equation}
using the same truncation and mixed-exponent techniques as in \cite{chong17-SPA}.
In the recent article \cite{J24}, the second author of this paper proved the uniqueness of the solution of \eqref{SWE}.

\medskip


We return now to the case of the finite variance L\'evy white noise. The existence and uniqueness of the solution of an SPDE with noise $\dot{L}$ was established in \cite{BN17}, using the same techniques as for the space-time Gaussian white noise. The higher order moments of the solution were studied in \cite{BN16} for the wave equation in dimension $d=1$. The linear stochastic heat and wave equations with this type of noise have random field solutions only in spatial dimension $d=1$, because
the fundamental solutions of these equations are not square-integrable on $[0,T] \times \bR^d$, if $d \geq 2$.  
We recall that a {\em space-time Gaussian white noise} is a zero-mean Gaussian process $\{W(A);A \in \cB_b(\bR_{+}\times \bR^d)\}$ with covariance:
\[
\bE[W(A)W(B)]=|A \cap B|,
\]
where $|A|$ is the Lebesque measure of $A$. This can be extended to an isonormal Gaussian process indexed by $L^2(\bR_{+}\times \bR^d)$, which in turn induces a zero-mean Gaussian process $W=\{W_t(\varphi);t\geq 0,\varphi \in L^2( \bR^d)\}$ with covariance:
\[
\bE[W_t(\varphi)W_s(\psi)]=(t\wedge s)\int_{\bR^d}\varphi(x) \psi(x)dx.
\]

In the case of Gaussian noise, a well-established procedure for obtaining random field solutions to SPDEs in higher dimensions is to introduce some ``color'' in the spatial structure of the noise. This idea appeared for the first time in \cite{dalang-frangos98} and \cite{millet-sanz99} for the stochastic wave equation in dimension $d=2$, and was extended to general SPDEs in \cite{dalang99}. More precisely, in these references, the noise is given by a zero-mean Gaussian process $F=\{F_t(\varphi);t\geq 0,\varphi \in \cS( \bR^d)\}$ with covariance:
\begin{align}
\label{cov-F}
\bE[F_t(\varphi)F_s(\psi)]&=(t \wedge s)\int_{\bR^d}\int_{\bR^d}\varphi(x)\psi(y) f(x-y)dxdy,
\end{align}
where $f:\bR^d \to [0,\infty]$ is the Fourier transform of a tempered measure $\mu$ on $\bR^d$. In \cite{dalang99}, it was proved that if the measure $\mu$ satisfies the (now-called) {\em Dalang's condition}:
\begin{equation}
\label{Dalang-cond}
\int_{\bR^d}\frac{1}{1+|\xi|^2}\mu(d\xi)<\infty,
\end{equation}
then the stochastic heat and wave equations with noise $\dot{W}$ have random field solutions for any $d\geq 1$, respectively for any $d\leq 3$. The extension to case $d\geq 4$ was studied in \cite{conus-dalang08}.  

\medskip

The noise $\dot{F}$  induces a {\em stationary random distribution} on $\bR^d$, an object which was introduced by It\^o in \cite{ito54} in the case $d=1$. Its defining property is that the covariance  is invariant under translations. In higher dimensions, such a process was studied by Yaglom in \cite{yaglom57}, who called it a {\em homogeneous generalized random field}. For this reason, $F$ is  usually referred to as a ``spatially-homogenous'' Gaussian noise.

An important observation is that $F$ can be constructed from $W$, by letting
\begin{equation}
\label{def-F}
F_t(\varphi)=W_t(\varphi * \k), \quad \varphi \in \cS(\bR^d)
\end{equation}
for a suitable kernel $\k:\bR^d \to [0,\infty]$ for which $\varphi * \k \in L^2(\bR^d)$ for all $\varphi \in \cS(\bR^d)$. This is because the covariance of $F$:
\[
\bE[F_t(\varphi)F_s(\psi)]=(t\wedge s)\int_{\bR^d}(\varphi * \k)(x)(\psi * \k)(x)dx,
\]
matches the right hand-side of \eqref{cov-F}, provided that  $\k * \widetilde{\k}=f$, where $\widetilde{\k}(x)=\k(-x)$. 

On the other hand, we have the following expression which involves only $\mu$:
\[
\bE[F_t(\varphi)F_s(\psi)]=
(t \wedge s)\int_{\bR^d} \cF \varphi(\xi)\overline{\cF \psi(\xi)} \mu(d\xi),
\]
where $\cF \varphi(\xi)=\int_{\bR^d}e^{-i \xi \cdot x}\varphi(x) dx$ is the Fourier transform of $\varphi$. This leads to the following alternative definition: 
\[
F_t(\varphi)=\widehat{W}_t(\cF \varphi \cdot \cF \k), \quad \varphi \in \cS(\bR^d),
\]
provided that the kernel $\k$ satisfies the condition: $\cF \varphi \cdot \cF \k \in L_{\bC}^2(\bR^d)$ for all $\varphi \in \cS(\bR^d)$.
Here $\widehat{W}$ is the ``Fourier transform'' of $W$, defined by $W_t(\varphi)=\widehat{W}_t(\cF \varphi)$ for all $\varphi \in L^2_{\bC}(\bR^d)$, and $L_{\bC}^2(\bR^d)$ is the space of $\bC$-valued square-integrable functions on $\bR^d$.

The existence of random field solutions of linear SPDEs with noise $\dot{F}$, of the form 
\[
\cL u(t,x)=\dot{F}(t,x), \quad t>0,x\in \bR^d
\]
for a second-order partial differential operator $\cL$,
can be done using only Fourier transforms, without requiring the existence of the kernel $f$. But the study of non-linear equations with Lipschitz coefficients $\sigma$ and $b$:
\[
\cL u(t,x)=\sigma\big(u(t,x)\big)\dot{F}(t,x)+b\big(u(t,x)\big), \quad t>0,x\in \bR^d,
\]
is greatly simplified, {\em if} we assume the existence of the kernel $f$. This foundational study 
was performed by Dalang in the seminal article \cite{dalang99}, and is now the backbone for a huge area of research.

\medskip

In 2015, at Dalang's suggestion, the first author of this paper introduced in \cite{B15} a similar ``coloring'' procedure in the L\'evy case, and defined the {\em L\'evy colored noise}:
\[
X_t(\varphi)=\widehat{L}_t(\cF \varphi \cdot \cF \k), \quad \varphi \in \cS(\bR^d),
\]
based on the ``Fourier transform'' $\widehat{L}$ of $L$. 
Article \cite{B15} developed an integration theory with respect to $X$, which is then applied for the study of linear SPDEs with noise $X$.

\medskip

After 10 years, in the present article, we continue this line of investigation by considering the next step: the study of non-linear equations with noise $X$. As in the Gaussian case considered in \cite{dalang99}, we assume for simplicity the existence of the kernel $f$. More precisely, by analogy with \eqref{def-F}, we define the L\'evy colored noise $X$ by:
\[
X_t(\varphi)=L_t(\varphi * \k),  \quad \varphi \in \cS(\bR^d).
\]
We postpone the general case (which does not require the existence of $f$) for future work.

\medskip

Throughout this article, we assume that $\nu$ satisfies condition \eqref{finite-m2}.
Therefore, modulo the constant $m_2$, the covariance of $L$ is  the same as that of $W$:
\[
\bE[L_t(\varphi)L_s(\psi)]=m_2(t\wedge s)\int_{\bR^d}\varphi(x) \psi(x)dx,
\]
and hence, the covariance of $X$ will be (modulo the constant $m_2$) the same as that of $F$:
\[
\bE[X_t(\varphi)X_s(\psi)]=m_2(t\wedge s)\int_{\bR^d}(\varphi * \k)(x)(\psi * \k)(x)dx.
\]

This means that most problems that were studied in \cite{dalang99} and are related to second-order moments can be extended immediately to the L\'evy case. The difficulties arise when we start to consider higher order moments. These moments will be the main focus and constitute the novelty of the present paper. 

\medskip

More precisely, the goal of the present paper is to provide some sufficient conditions under which the solution of the non-linear SPDE:
\begin{equation}
\label{nonlin-eq}
\cL u(t,x)=\sigma(u(t,x))\dot{X}(t,x)+b(u(t,x)) \quad t>0,x \in \bR^d
\end{equation}
(with Lipschitz functions $\sigma$ and $b$), has moments of order $p\geq 2$, and analyze the asymptotic behaviour of these moments when $t \to \infty$. 
We present the general theory for the case when $\cL$ is a second order partial differential operator (with coefficients that do not depend on the space variable), and then
perform a detailed analysis in the case when $\cL$ is the heat or wave operator, for three examples of kernels $\k$: the heat kernel, the Riesz kernel, and the Bessel kernel.

We denote by $G_t$ the fundamental solution of $\cL u=0$, with the convention:
\begin{equation}
\label{G-convention}
G_t(x)=0 \quad \mbox{if $t<0$}.
\end{equation}

We suppose that the initial condition is given by the constant $\eta \in \bR$. 
More precisely, we assume that the solution $w$ of the equation $\cL u=0$ with the same initial conditions is the constant function $w(t,x)=\eta$.  
For instance, \\
(i) if $\cL=\frac{\partial}{\partial t} -\frac{1}{2}\Delta $ is the heat operator and the initial condition is $u(0,x)=u_0(x)=\eta$, then $w(t,x)=(G_t*u_0)(x)=\eta$;\\
(ii) if $\cL=\partial_{tt} -\Delta $ is the wave operator and the initial conditions are $u(0,x)=u_0(x)=\eta$ and $\frac{\partial}{\partial t} u(0,x)=v_0=0$, then $w(t,x)=\frac{\partial}{\partial t}(G_t*u_0)(x)+(G_t*v_0)(x)=\eta$. 

\medskip

Following \cite{walsh86,dalang99}, we introduce the following definition.

\begin{definition}
\label{def-11}
{\rm 
A predictable process $u=\{u(t,x);t \geq 0,x\in \bR^d\}$ is a {\em (mild) solution} of equation \eqref{nonlin-eq} with constant initial condition $\eta \in \bR$, if it satisfies the integral
equation:
\[
u(t,x)=\eta+\int_0^t \int_{\bR^d}G_{t-s}(x-y)\sigma\big(u(s,y)\big) X(ds,dy)+\int_0^t \int_{\bR^d}b\big(u(t-s,x-y)\big)G_{s}(dy)ds.
\]
}
\end{definition}

To explain our strategy for handling moments, we recall that in the case of stochastic integrals with respect to discontinuous martingales, the Burkholder-Davis-Gundy inequality does not provide useful bounds, and one has to rely on Rosenthal's inequality instead. This inequality (which we recall in Proposition \ref{ros-prop}), is applied in Theorem \ref{dalang-th5} to obtain a moment estimate for the $p$-th moment of the stochastic integral 
\begin{equation}
\label{stoc-int-XZ}
\int_0^t\int_{\bR^d} S(s,x)Z(s,x)X(ds,dx)
\end{equation}
with respect to the L\'evy colored noise $X$, when $S$ is deterministic and $Z$ is random.
Compared to the bound given by Theorem 5 of \cite{dalang99} for the spatially-homogeneous Gaussian noise $F$, our estimate contains an extra term which involves $\|S(s,\cdot) * \k\|_{L^p(\bR^d)}$.
To see that this term is natural, assume for simplicity that $Z=1$. Then the integral in \eqref{stoc-int-XZ} coincides with the integral of $S(s,\cdot)*\k$ with respect to $L$ (see Remark \ref{remark-lin}). This integral has a finite $p$-th moment if and only if the integrand is in $L^p([0,t] \times \bR^d)$ and $m_p<\infty$, according to Corollary \ref{p-mom-int}. 
By exploiting this idea, we provide in Section \ref{section-lin} some sufficient conditions which guarantee that the solutions of the linear heat and wave equations (with $\sigma=1$ and $b=0$) have finite $p$-th moments, in the case the three examples of kernel mentioned above. Then, in Section \ref{section-with-drift}, we extend this result to the non-linear equations.

In Theorem \ref{dalang-th5-new}, we provide an alternative bound for the $p$-th moment of integral \eqref{stoc-int-XZ}, in which the constant does not depend on $t$.
This result is used in Section \ref{section-no-drift} to obtain more precise bounds for the $p$-th moments of the solution $u$ of equation \eqref{nonlin-eq} without drift (i.e. with $b=0$), and conclude that the solution $u$ has a finite {\em upper Lyapunov exponent}:
\begin{equation}
\label{def-gamma}
\overline{\gamma}(p):=\limsup_{t \to \infty}\frac{1}{t}\log \bE|u(t,x)|^p.
\end{equation}
Note that $\overline{\gamma}(p)$ does not depend on $x$ since $u$ is stationary in space, a property which is due to the space homogeneity of the noise and the fact that the initial condition is constant.

This leads us immediately to the question of intermittency, which has been studied extensively in the Gaussian case. Following \cite{FK13}, we say that $u$ is {\em weakly intermittent} if 
\[
\overline{\gamma}(p)<\infty \ \mbox{for all $p\geq2$}, \quad \mbox{and} \quad \overline{\gamma}(2)>0.
\]
In Section \ref{section-interm}, using a similar argument as in \cite{FK13} for the Gaussian case, we give some sufficient conditions on the kernel $\k$ under which $\overline{\gamma}(2)>0$, in the case of the heat equation in dimension $d\geq 1$ and the wave equation in dimension $d\leq 2$. Then we identify some cases for which the solutions of these equations are weakly intermittent.

Finally, in Section \ref{section-pham}, we consider the particular case of the 
the {\em parabolic/hyperbolic Anderson model}. This is the name given to
equation \eqref{nonlin-eq} when $\cL$ is the heat/wave operator, $b=0$ and
$\sigma(u)=\lambda u$. In the case of the wave operator, we assume that 
$d\leq 2$. Using tools from Malliavin calculus on the Poisson space, we identify the chaos expansion of the solution and compute explicitly $\overline{\gamma}(2)$.



\section{Preliminaries}

In this section, we introduce the L\'evy white noise and the L\'evy colored noise and discuss some of their properties. 

\subsection{L\'evy white noise}
\label{section-Levy-white}

In this section, we introduce the stochastic integral with respect to the L\'evy white noise $L$, and we present a new result which states that this integral has a finite $p$-th moment if and only if the (deterministic) integrand is $p$-integrable and $m_p<\infty$.




\medskip

Let $N$ be the PRM specified in the introduction and $L$ be the L\'evy white noise given by \eqref{def-L}. 
We recall the following properties:

\begin{itemize}

\item $L$ is a {\em ($\sigma$-additive) random measure}, i.e. if $(A_n)_{n \geq 1} \subset \cB_b(\bR_{+} \times \bR^d)$ are disjoint and $\bigcup_{n\geq 1}A_n$ is bounded, then
    $    L\Big(\bigcup_{n\geq 1} A_n\Big)=\sum_{n\geq 1}L(A_n)$ a.s.

\item $L$ is {\em independently scattered}, i.e. $L(A_1),\ldots, L(A_n)$ are independent, for any disjoint sets $A_1,\ldots, A_n \in \cB_b(\bR_{+} \times \bR^d)$.
    
    \item $L$ is {\em infinitely divisible (ID)}, i.e. for any $A \in \cB_b(\bR_{+} \times \bR^d)$, $L(A)$ has an ID distribution. More precisely, the characteristic function of $L(A)$ is given by:
\[
\bE(e^{iu L(A)})=\exp\left\{|A| \int_{\bR_0} (e^{i u z}-1-i u z)\nu(dz)  \right\}, \quad u \in \bR.
\]

\end{itemize}

We include below two examples of L\'evy white noise with finite moments of all orders.

\begin{example}[Gamma white noise]
\label{ex-gamma}
{\rm 
A {\em Gamma white noise} with parameters $\alpha>0,\beta>0$ is an independently scattered random measure $\gamma=\{\gamma(A);A \in \cB_b(\bR_{+} \times \bR^d)\}$ such that 
\[
\gamma(A)\sim {\rm Gamma}(\alpha|A|,\beta) \quad \mbox{for any $A \in \cB_{b}(\bR_{+}\times \bR^d)$}.
\]
Here the Gamma$(\alpha,\beta)$ distribution has
density $f(x)=\frac{\beta^{\alpha}}{\Gamma(\alpha)} x^{\alpha -1} e^{-\beta x}1_{\{x>0\}}$. (The word ``white'' does not mean ``zero-mean''.)
Then,
\[
\bE[e^{iu \gamma(A)}]=\left(1-\frac{iu}{\beta} \right)^{-\alpha|A|}=\exp\left\{\alpha|A|\int_0^{\infty}(e^{iuz}-1) \frac{1}{z}e^{-\beta z}dz\right\}, \quad u\in \bR,
\]
and the L\'evy measure $\nu(dz)=\alpha z^{-1}e^{-\beta z}1_{\{z>0\}}dz$ satisfies:
\[
m_p=\int_{\bR_0}|z|^p\nu(dz)=\alpha \int_0^{\infty}z^{p-1}e^{-\beta z}dz=\frac{\alpha \Gamma(p)}{\beta^{p}} \quad \mbox{for any $p>0$}.
\]
The centered Gamma white noise $L$ defined by
$L(A):=\gamma(A)-\frac{\alpha}{\beta} |A|$ is a L\'evy white noise of form \eqref{def-L},
for which $\bE|L(A)|^p<\infty$ for all $p>0$.

}
\end{example}

\begin{example}[Variance Gamma white noise] 
\label{ex-VG}
{\rm In this example, we give an extension of the variance gamma (VG) process defined in \cite{MCC98}. 
A Gamma white noise with ``mean rate'' $\mu\in \bR$ and ``variance rate'' $\nu>0$ is a Gamma white noise with parameters $\alpha=\mu^2/\nu$ and $\beta=\mu/\nu$, as defined in Example \ref{ex-gamma}. Let $\gamma_p$ and $\gamma_n$ be independent Gamma white noises with mean and variance rates $(\mu_p,\nu_p)$ and $(\mu_n,\nu_n)$ given by:
\[
\mu_p=\frac{1}{2}\sqrt{\theta^2+\frac{2\sigma^2}{\nu}}+\frac{\theta}{2}, \quad
\nu_p=\mu_p^2 \nu \quad \mbox{and} \quad
\mu_n=\frac{1}{2}\sqrt{\theta^2+\frac{2\sigma^2}{\nu}}-\frac{\theta}{2}, \quad \nu_n=\mu_n^2 \nu,
\]
for some $\theta \in \bR$, $\sigma>0$ and $\nu>0$. The process $\{X(A); A \in \cB_b(\bR_{+} \times \bR^d)\}$ defined by
\[
X(A)=\gamma_p(A)-\gamma_n(A) \quad 
\]
is called a {\em variance gamma (VG) white noise} with parameters $(\theta,\sigma,\nu)$. 

The values $(\mu_p,\nu_p)$ and $(\mu_n,\nu_n)$ satisfy:
\[
\frac{\mu_p^2}{\nu_p}=\frac{\mu_n^2}{\nu_n}=\frac{1}{\nu}, \quad \quad \quad \frac{\nu_p}{\mu_p} \cdot \frac{\nu_n}{\mu_n}=\frac{\sigma^2 \nu}{2}, \quad \quad \quad  \frac{\nu_p}{\mu_p}-\frac{\nu_n}{\mu_n}=\theta \nu.
\]
Therefore, the characteristic function of $X(A)$ can be computed as follows:
\begin{align*}
\bE[e^{iu X(A)}]&=\bE\big[e^{i u\gamma_p(A)}\big]\cdot \bE\big[e^{-i u \gamma_n(A)}\big]=
\left(1-\frac{iu \nu_p}{\mu_p}\right)^{-\frac{\mu_p^2 |A|}{\nu_p}} \left(1+\frac{iu \nu_n}{\mu_n}\right)^{-\frac{\mu_n^2 |A|}{\nu_n}}\\
&=
\left(1-iu \theta \nu +\frac{u^2 \sigma^2 \nu}{2}\right)^{-\frac{|A|}{\nu}}.
\end{align*}

This means that $X(A)\stackrel{d}{=}Y(A)$ for any $A \in \cB_b(\bR_{+} \times \bR^d)$, where 
\[
Y(A)=B_{\gamma(A)}
\]
is defined by ``subordinating'' a classical Brownian motion $B_t=\theta +\sigma W_t$ (with $(W_t)_{t\geq 0}$ a standard Brownian motion) to an independent Gamma white noise $\{\gamma(A);A \in \cB_b(\bR_{+} \times \bR^d)\}$ with mean rate $1$ and variance rate $\nu$.
Unlike the classical case (see e.g. Theorem 30.1 of \cite{sato99}), this subordination procedure does not yield an independently scattered process:  
for disjoint sets $A_1,\ldots,A_n \in \cB_{b}(\bR_{+} \times \bR^d)$, and bounded continuous functions $f_1,\ldots,f_n$, 
\[
\bE\Big[\prod_{j=1}^n f_j(Y(A_j))\Big]=\bE\big[G(\gamma(A_1),\ldots,\gamma(A_n))\big]
\]
with function $G(s_1,\ldots,s_n):=\bE\Big[\prod_{j=1}^n f_j(B_{s_j})\Big]$ that cannot be transformed into a product. This means that the processes $\{X(A)\}$ and $\{Y(A)\}$ do not have the same distribution, although they have the same marginal distributions.

Clearly, $\bE|X(A)|^p<\infty$ for any $p>0$. In fact, the mean and variance of $X(A)$ are:
\[
\bE\big[X(A)\big]=\theta |A| \quad \mbox{and} \quad {\rm Var}\big(X(A)\big)=\left(\frac{5}{8}\theta^2 \nu+\sigma^2\right)|A|.
\]
The characteristic function of $X(A)$ can be written also as:
\[
\bE[e^{iu X(A)}]=\exp\left\{|A|\int_0^{\infty}(e^{iuz}-1) \nu^{*}(dz) \right\}, \quad u\in \bR,
\]
with L\'evy measure $\nu^*$ given by:
\[
\nu^*(dz) =\left[\frac{\mu_p^2}{\nu_p}\cdot \frac{1}{z}e^{-\frac{\mu_p}{\nu_p}z}1_{\{z>0\}}+
\frac{\mu_n^2}{\nu_n}\cdot \frac{1}{|z|}e^{-\frac{\mu_n}{\nu_n}|z|}1_{\{z<0\}}\right]dz.
\]

The {\em centered VG white noise} $L$ defined by
$L(A):=X(A)-\theta |A|$ is a L\'evy white noise of form \eqref{def-L} with L\'evy measure $\nu^*$,
for which $\bE|L(A)|^p<\infty$ for all $p>0$.
}
\end{example}

We continue now with a discussion about the stochastic integral with respect to $L$.
We let $L(1_{A})=L(A)$. The map $1_A \mapsto L(1_{A})\in L^2(\Omega)$ is an isometry which can be extended to $L_{\bC}^2(\bR_{+} \times \bR^d)$, by linearity and a density argument. 
We denote this map by:
\[
L(\varphi)=\int_0^{\infty}\int_{\bR^d}\varphi(t,x)L(dt,dx) \quad \mbox{for all} \quad \varphi \in L_{\bC}^2(\bR_{+} \times \bR^d).
\]
More precisely, the following isometry property holds: for all $\varphi,\psi \in L_{\bC}^2(\bR_{+}\times \bR^d)$,
\begin{equation}
\label{isom-L1}
\bE[L(\varphi)L(\psi)]=m_2 \int_{\bR_{+} \times \bR^d}\varphi(t,x) \overline{\psi(t,x)}dtdx.
\end{equation}

For any $\varphi \in L^2(\bR_{+} \times \bR^d)$, $L(\varphi)$ has the Poisson representation:
\begin{equation}
\label{Poisson-rep}
L(\varphi)=\int_{\bR_{0} \times \bR^d \times \bR_0}  \varphi(t,x)z \widehat{N}(dt,dx,dz),
\end{equation}
and characteristic function
\[
\bE(e^{i\lambda L(\varphi)})=\exp\left\{\int_{\bR_{+}\times \bR^d \times \bR_0} \big(e^{i \lambda z \varphi(t,x)}-1-i\lambda z \varphi(t,x) \big)dtdx\nu(dz) \right\}, \quad \lambda \in \bR.
\]

For any $\phi \in L^2(\bR^d)$, the process $L_t(\phi)=\{L(1_{[0,t]}\phi)\}_{t\geq 0}$
is a L\'evy process and a zero-mean square-integrable martingale with respect to $(\cF_t)_{t\geq 0}$. This process has a c\`adl\`ag modification. We work with the modification. We let $L_t(B)=L_t(1_{B})$ for any $B \in \cB_b(\bR^d)$.

\medskip

The ``Fourier transform'' of the noise $L$ in space is defined by:
\[
\widehat{L}(\varphi):=\int_0^{\infty}\int_{\bR^d}\cF^{-1} \varphi(t,\cdot)(x)L(dt,dx),
\]
for any $\varphi \in L_{\bC}^2(\bR_{+} \times \bR^d)$. We will use the notation $\widehat{L}(\varphi)=\int_0^{\infty}\int_{\bR^d}\varphi(t,\xi)\widehat{L}(dt,d\xi)$.

For any $t>0$, we let
\begin{equation}
\label{def-L-hat}
\widehat{L}_t(\phi):=L_t(\cF^{-1}\phi) \quad \mbox{for any $\phi \in L_{\bC}^2(\bR^d)$}.
\end{equation}


Next, we consider the case of random integrands. 
We consider the filtration $\cF_t=\cF_t^{0} \vee \cN$, where $\cF_t^0=\sigma (L_s(A);s\leq t,A \in \cB_b(\bR^d))$ and $\cN$ is the null $\sigma$-field.

We say that $\{g(t,x);t\geq 0,x \in \bR^d\}$ is an {\em elementary} process if it is of the form
\begin{equation}
\label{elem}
g(\omega,t,x)=Y(\omega)1_{(a,b]}(t) 1_{A}(x),
\end{equation}
where $0\leq a<b$, $A \in \cB_b(\bR^d)$ and $Y$ is a bounded $\cF_a$-measurable random variable. The stochastic integral with respect to $L$ of a process $g$ of form \eqref{elem} is given by:
\[
(g\cdot L)_t:=\int_0^{t} \int_{\bR^d}g(t,x)L(dt,dx)=Y (L_{t\wedge b}(A)-L_{t\wedge a}(A)).
\]

Let $\cE$ be the set of linear combinations of elementary processes, and $\cP_{\Omega \times \bR_{+}\times \bR^d}=\sigma(\cE)$ be the {\em predictable $\sigma$-field} on $\Omega \times \bR_{+} \times \bR^d$. (We refer the reader to Lemma \ref{lem-predict} for an alternative representation of $\cP_{\Omega \times \bR_{+}\times \bR^d}$.) A process $g=\{g(t,x);t\geq 0,x \in \bR^d\}$ is {\em predictable} if it is measurable with respect to $\cP_{\Omega \times \bR_{+}\times \bR^d}$.

\medskip

For any predictable process $g \in L^2(\Omega \times [0,T] \times \bR^d)$, we define the {\em It\^o integral} of $g$ with respect to $L$, and this integral satisfies the isometry property:
\begin{equation}
\label{isom-L2}
\bE \left|\int_0^T \int_{\bR^d}g(t,x)L(dt,dx)\right|^2=m_2 \bE \int_0^T \int_{\bR^d} |g(t,x)|^2 dxdt.
\end{equation}
In addition,
\begin{equation}
\label{Ito-LN}
\int_0^T \int_{\bR^d} g(t,x)L(dt,dx)=\int_{0}^T \int_{\bR^d} \int_{\bR_0} g(t,x) z \widehat{N}(dt,dx,dz),
\end{equation}
where on the right hand-side we have the It\^o integral with respect to $\widehat{N}$.

\medskip

We recall Rosenthal's inequality: for any c\`adl\`ag martingale $M=(M_t)_{t\geq 0}$ and $p\geq 2$,
\[
\|\sup_{s\leq t}|M_s|\|_p \leq B_p \left(\|\langle M\rangle_{t}^{1/2}\|_{p} +\| \sup_{s\leq t}|(\Delta M)_s|\|_p\right),
\]
where $\langle M\rangle$ is the predictable quadratic variation of $M$, and $(\Delta M)_t=M_{t}-M_{t-}$ is the jump size of $M$ at time $t$. As an application, we obtain the following inequality, which was proved in \cite{BN16}, and plays an important role in the present article. See also Theorem 1 of \cite{marinelli-rockner14}.

\begin{proposition}[Rosenthal Inequality for $L$]
\label{ros-prop}
Let $\Phi=\{\Phi(t,x);t\in [0,T],x \in \bR^d\}$ be a predictable process such that $\Phi \in L^2(\Omega \times [0,T] \times \bR^d)$. If $p\geq 2$ is such that $m_p<\infty$,
then
\begin{equation}
\label{rosenthal}
\bE\left|\int_0^T \int_{\bR^d}\Phi(t,x) L(dt,dx)\right|^p \leq \cC_p \left\{ \bE \left( \int_0^T \int_{\bR^d} |\Phi(t,x)|^2 dx dt \right)^{p/2} + \bE \int_0^T \int_{\bR^d}|\Phi(t,x)|^p dxdt\right\},
\end{equation}
where $\cC_p=2^{p-1}B_p^p(m_2^{p/2} \vee m_p)$.
\end{proposition}

In the case of a deterministic integrand, it is possible to give a necessary and sufficient condition for the existence of the $p$-th moment of the stochastic integral with respect to $L$. To prove this, we will use the following result, inspired by Theorem 25.3 of \cite{sato99} (which is stated for a classical L\'evy process, indexed by $\bR_{+}$).

\begin{theorem}
\label{sato-th25-3}
Let $\Phi \in L^2([0,T] \times \bR^d)$ and $g:\bR \to [0,\infty)$ be a locally integrable function which is ``sub-multiplicative'', i.e.
$g(x+y) \leq a g(x) g(y)$ for all $x,y \in \bR$, for some $a>0$. If
\[
\bE\left[g\left(\int_0^T \int_{\bR^d} \Phi(t,x) L(dt,dx)\right)\right]<\infty,
\]
then
\[
\int_0^T \int_{\bR^d} \int_{\{|z \Phi(t,x)| >1\}} g\big(z \Phi(t,x) \big) \nu(dz)dxdt<\infty.
\]
\end{theorem}

\begin{proof} Since $\Phi$ is It\^o integrable with respect to $L$, it is also $L$-integrable, in the sense of the definition given on page 460 of \cite{RR89}. By Theorem 2.7 of \cite{RR89},
\begin{equation}
\label{th27-RR}
\int_0^T \int_{\bR^d} \big(|z\Phi(t,x)|^2 \wedge 1\big) \nu(dz)dxdt<\infty.
\end{equation}

As mentioned in the introduction, $L(A)$ can be written as:
\[
L(A)=b|A|+\int_{A \times \{|z|\leq 1\}} z \widehat{N}(dt,dx,dz)+\int_{A \times \{|z|> 1\}} z N(dt,dx,dz), 
\]
with $b=-\int_{\{|z|>1\}}z \nu(dz)$.
This means that $L$ has characteristic triplet $(b,0,\nu)$, as defined in Theorem 1.0.9 of \cite{humeau}. By Theorem 2.3.10 of \cite{humeau}, 
\begin{align*}
& \int_0^T \int_{\bR^d}\Phi(t,x) L(dt,dx)=\int_0^T \int_{\bR^d} \left[ 
b \Phi(t,x)+\int_{\bR_0} z \Phi(t,x)\big(1_{\{|z\Phi(t,x)| \leq 1\}}-1_{\{|z|\leq 1\}} \big)\nu(dz) \right] dxdt\\
& \quad + \int_0^T \int_{\bR^d} \int_{\{|z\Phi(t,x)|\leq 1\}} z \Phi(t,x) \widehat{N}(dt,dx,dz)+
\int_0^T \int_{\bR^d}\int_{\{|z\Phi(t,x)|>1\}}z\Phi(t,x) N(dt,dx,dz)\\
& \quad =:B+X+Y.
\end{align*}

We denote by $P_X$ be the law of $X$ and $P_Y$ the law of $Y$. Then
\[
\bE\left[g\left(\int_0^T \int_{\bR^d} \Phi(t,x) L(dt,dx)\right)\right]=\bE\big[g(B+X+Y)\big]=\int_{\bR} \int_{\bR}g(B+x+y)P_{X}(dx)P_{Y}(dy).
\] 
Since this expectation is finite, 
$\int_{\bR}g(B+x+y)P_{Y}(dy)=\bE\big[g(B+x+Y)\big]<\infty$ for $P_{X}$-almost all $x\in \bR$. Hence,
\begin{equation}
\label{Eg-finite}
\bE[g(x'+Y)]<\infty \quad \mbox{for $x'=B+x$ and $P_{X}$-almost all $x \in \bR$}.
\end{equation}

Recall that $\mu(dt,dx,dz)=dtdx\nu(dz)$. Let
\[
\cR=\{(t,x,z)\in [0,T] \times \bR^d \times \bR_0; |z\Phi(t,x)|>1\}.
\]
 From \eqref{th27-RR}, we know that $\mu(\cR)<\infty$. Then, the restriction of $N$ to $\cR$ has the following representation (in law):
\[
N|_{\cR} \stackrel{d}{=} \sum_{i=1}^{K} \delta_{(T_i,X_i,Z_i)},
\]
where $K$ is a Poisson random variable with mean $\mu(\cR)$, $\{(T_i,X_i,Z_i)\}_{i\geq 1}$ are i.i.d. with law $\frac{1}{\mu(\cR)}\mu|_{\cR}$, and are independent of $K$. Here $\stackrel{d}{=}$ denotes equality in distribution. Hence,
\[
Y=\int_{\cR} z \Phi(t,x) N(dt,dx,dz)\stackrel{d}{=}\sum_{i=1}^{K} Z_i \Phi(T_i,X_i)=:Y'.
\]
Since $\{W_i:=Z_i \Phi(T_i,X_i);i\geq 1\}$ are i.i.d., $Y'$ has a compound-Poisson distribution. Let $F$ be the law of $W_1$. For any measurable function $h:\bR \to [0,\infty)$, we have:
\[
\int_{\bR}h(y) F(dy)=\bE[h(W_1)]=\frac{1}{\mu(\cR)}\int_{\cR}h\big(z \Phi(t,x) \big) dtdx \nu(dz),
\]
and
\[
\bE[h(Y)]=\bE[h(Y')]=\sum_{n\geq 0} e^{-\lambda} \frac{\lambda^n}{n!} \int_{\bR^n} h(y)F^{*n}(dy).
\]
where $F^{*n}$ is the law of $X_1+\ldots+X_n$.
Applying this formula to the function $h(x')=g(x'+\cdot)$, and using \eqref{Eg-finite}, we obtain that, for $x'=B+x$ and $P_X$-almost all $x \in \bR$,
\[
\sum_{n\geq 0}e^{-\lambda}\frac{\lambda^n}{n!} \int_{\bR^d}g(x'+y) F^{*n}(dy)<\infty. 
\]
By Lemma 25.5 of \cite{sato99}, there exist $b>0$ and $c>0$ such that
$g(x) \leq b e^{c|x|}$ for all $x \in \bR$.
Hence, $g(y) \leq a g(-x) g(x+y) \leq ab e^{c|x|}g(x+y)$ for all $x,y \in \bR$. It follows that
\[
\sum_{n\geq 0}e^{-\lambda}\frac{\lambda^n}{n!} \int_{\bR^d}g(y) F^{*n}(dy)\leq ab e^{c|x'|} \sum_{n\geq 0}e^{-\lambda}\frac{\lambda^n}{n!} \int_{\bR^d}g(x'+y) F^{*n}(dy)<\infty, 
\]
for $x'=B+x$ and $P_X$-almost all $x \in \bR$. In particular, for $n=1$, 
\[
\int_{\bR}g(y)F(dy)=\frac{1}{\mu(\cR)} \int_{\cR} g\big(z \Phi(t,x) \big)dtdx\nu(dz)<\infty.
\]
\end{proof}

\begin{corollary}
\label{p-mom-int}
Let $\Phi  \in  L^2([0,T] \times \bR^d)$ and $p\geq 2$. Then 
$\bE\left|\int_0^T \int_{\bR^d} \Phi(t,x)L(dt,dx)\right|^p<\infty$
if and only if 
\begin{equation}
\label{nec-suf-p}
\Phi \in L^p([0,T] \times \bR^d) \quad \mbox{and} \quad  m_p<\infty.
\end{equation}
\end{corollary}

\begin{proof}
Let $Z=\int_0^T \int_{\bR^d} \Phi(t,x)L(dt,dx)$.
If \eqref{nec-suf-p} holds, then $\bE|Z|^p<\infty$, by Proposition \ref{ros-prop}. 

For the converse implication, suppose that $\bE|Z|^p<\infty$. Let $g(x)=(1+|x|^2)^{p/2}$ for $x \in \bR$. Then $g$ is sub-multiplicative, and $\bE\big[g(Z)\big]<\infty$. By Theorem \ref{sato-th25-3}, 
\[
\int_0^T \int_{\bR^d} \int_{\{|z\Phi(t,x)|\geq 1\}} g\big( z \Phi(t,x)\big)dtdx\nu(dz)<\infty,
\]
and hence $\int_0^T \int_{\bR^d} \int_{\{|z\Phi(t,x)|\geq 1\}}|z \Phi(t,x)|^p dtdx\nu(dz)<\infty$. On the other hand, by \eqref{th27-RR},
\[
\int_0^T \int_{\bR^d} \int_{\{|z\Phi(t,x)|\leq 1\}}|z \Phi(t,x)|^p dtdx\nu(dz) \leq \int_0^T \int_{\bR^d} \int_{\{|z\Phi(t,x)|\leq 1\}}|z \Phi(t,x)|^2 dtdx\nu(dz)<\infty.
\]
Hence, $\int_0^T \int_{\bR^d} \int_{\bR_0}|z \Phi(t,x)|^p dtdx\nu(dz)=m_p \int_0^T \int_{\bR^d} \int_{\bR_0}|\Phi(t,x)|^p dtdx<\infty$.
\end{proof}

\subsection{L\'evy colored noise}

In this section, we introduce the L\'evy colored noise, paying attention to impose only minimal assumptions on the kernel $\k$. 

We first recall some basic terminology and notations. We say that a $C^{\infty}$-function $\varphi:\bR^d \to \bR$ is {\em rapidly decreasing} if 
\[
s_{\alpha,\beta}(\varphi):=\sup_{x\in \bR^d}|x^{\alpha}D^{\beta} \varphi(x)|<\infty
\]
for all multi-indices $\alpha=(\alpha_1,\ldots,\alpha_d),\beta=(\beta_1,\ldots,\beta_d) \in \bN^d$, where $x^{\alpha}=x_1^{\alpha_1}\ldots x_d^{\alpha_d}$ and $D^{\beta}=\frac{\partial^{|\beta|}}{\partial x_1^{\beta_1}\ldots \partial x_d^{\beta_d}}$. 
We let $\cS(\bR^d)$ be the set of all rapidly decreasing $C^{\infty}$-functions on $\bR^d$. We say that $\varphi_n \to \varphi$ in $\cS(\bR^d)$ if $s_m(\varphi_n-\varphi) \to 0$ for all $m \in \bN$, where
\[
s_m(\varphi)=\sup_{|\alpha|+|\beta|\leq m}s_{\alpha,\beta}(\varphi), \quad \mbox{with} \quad |\alpha|=\sum_{i=1}^{d}\alpha_i.
\]

We let $\cS'(\bR^d)$ be the set of tempered distributions on $\bR^d$, which are continuous linear functionals $S:\cS(\bR^d) \to \bR$. We denote by $(S,\varphi)$ the action of $S$ on $\varphi \in \cS(\bR^d)$. We say that a distribution $S\in \cS'(\bR^d)$ is (identified with) a function $f:\bR^d \to \bR$ if
\[
(S,\varphi)=\int_{\bR^d} f(x) \varphi(x)dx \quad \mbox{for all} \quad \varphi \in \cS(\bR^d).
\]
 
We let $\cS_{\bC}(\bR^d)$ be the set of $\bC$-valued rapidly decreasing $C^{\infty}$-functions on $\bR^d$, and $\cS_{\bC}'(\bR^d)$ the set of complex tempered distributions on $\bR^d$, which are continuous linear functionals $S:\cS_{\bC}(\bR^d)\to \bC$.

The Fourier transform is an isomorphism $\cF :\cS_{\bC}(\bR^d) \to \cS_{\bC}(\bR^d)$ with inverse
$\cF^{-1} \varphi=\frac{1}{(2\pi)^d} \overline{\cF \varphi}$.
The Fourier transform of $S\in \cS_{\bC}'(\bR^d)$ is a distribution $\cF S \in \cS_{\bC}'(\bR^d)$ given by:
\[
(\cF S, \varphi)=(S,\cF \varphi), \quad \mbox{for all} \quad \varphi \in \cS_{\bC}(\bR^d).
\]
 
\begin{definition}
{\rm
We say that a function $f:\bR^d \to \overline{\bR}$ is {\em tempered} if 
\[
\int_{\bR^d} \left(\frac{1}{1+|x|^2}\right)^{\ell} |f(x)|dx<\infty \quad \mbox{for some $\ell>0$}.
\]
}
\end{definition}

A tempered function $f$ is locally integrable and induces a distribution $S_f \in \cS'(\bR^d)$ by:
\[
(S_f,\varphi)=\int_{\bR^d} f(x) \varphi(x) dx \quad \mbox{for all $\varphi \in \cS(\bR^d)$}.
\]
We identify $f$ and $S_f$. We say that $\cF f=\cF S_{f}$ is the Fourier transform of $f$ in $\cS'(\bR^d)$.

\medskip

We introduce the following assumption.

\medskip

\noindent
{\bf Assumption A1.} Let $\k:\bR^d \to \overline{\bR}$ be a tempered function such that:

(a) the Fourier transform $h:=\cF \k$ in $\cS_{\bC}'(\bR^d)$ is a tempered function on $\bR^d$:
\[
\int_{\bR^d} h(\xi)\phi(\xi)d\xi=\int_{\bR^d}\k(x) \cF \phi(x)dx \quad \mbox{for all $\phi \in \cS_{\bC}(\bR^d)$}.
\]

(b) $|h|^2$ is a tempered function;

(c) there exists a tempered function $f:\bR^d \to [0,\infty]$ such that $|h|^2=\cF f$ in $\cS'(\bR^d)$:
\begin{equation}
\label{Four-h2}
\int_{\bR^d} |h(\xi)|^2 \phi(\xi)d\xi=\int_{\bR^d}f(x) \cF \phi(x)dx \quad \mbox{for all $\phi \in \cS_{\bC}(\bR^d)$}.
\end{equation}

\medskip

Part (a) of Assumption A1 can be written also as:
\begin{equation}
\label{Fk-h}
\int_{\bR^d} \varphi(x)\k(x) dx=\frac{1}{(2\pi)^d} \int_{\bR^d}  \overline{\cF \varphi(\xi)} h(\xi)d\xi \quad \mbox{for all $\varphi \in \cS_{\bC}(\bR^d)$.}
\end{equation}
This implies that for any $\varphi \in \cS_{\bC}(\bR^d)$, $\varphi * \k$ is well-defined and is given by:
\begin{equation}
\label{F-phi-h}
(\varphi*\k)(x)=\int_{\bR^d}\varphi(x-y)\k(y)dy=\frac{1}{(2\pi)^d}\int_{\bR^d}e^{i\xi \cdot x} \cF \varphi(\xi)\, h(\xi)d\xi.
\end{equation}
Moreover, $\varphi *\k$ is a tempered function, since if $2\ell>d$,
\[
\int_{\bR^d}\left( \frac{1}{1+|x|^2}\right)^{\ell} |(\varphi*\k)(x)|dx \leq \frac{1}{(2\pi)^d} 
\int_{\bR^d}\left( \frac{1}{1+|x|^2}\right)^{\ell}dx \int_{\bR^d} |\cF \varphi(\xi)h(\xi)|d\xi<\infty.
\]

Part (c) of Assumption A1 says that $\cF(\k*\widetilde{\k})=\cF f$, which is equivalent to
\begin{equation}
\label{f=k*k}
f=\k* \widetilde{\k}, \quad \mbox{with} \quad \widetilde{\k}(x)=\k(-x).
\end{equation}

From relation \eqref{Four-h2}, we deduce that:
\[
\int_{\bR^d} \varphi(x)f(x) dx=\int_{\bR^d} \cF \varphi(\xi) \mu(d\xi), \quad \mbox{for any $\varphi \in \cS(\bR^d)$},
\]
where
\begin{equation}
\label{def-mu}
\mu(d\xi)=\frac{1}{(2\pi)^d}|\cF \k(\xi)|^2 d\xi.
\end{equation}
 
Replacing $\varphi$ by $\varphi * \widetilde{\psi}$, we obtain that for any $\varphi,\psi \in \cS(\bR^d)$,
\begin{equation}
\label{Dalang-cov}
\int_{\bR^d} \int_{\bR^d}\varphi(x) \psi(y) f(x-y)dxdy=\int_{\bR^d}\cF \varphi(\xi)\overline{\cF \psi(\xi)}\mu(d\xi).
\end{equation}

We have the following result.

\begin{lemma}
\label{L2-lem}
Let $\varphi \in \cS(\bR^d)$ be arbitrary. (i) If part (a) of Assumption A1 holds, then
$\cF (\varphi *\k)=\cF \varphi \cF \k$ in $\cS_{\bC}'(\bR^d)$. (ii)
If in addition, part (b) also holds, then $\varphi*\k \in L^2(\bR^d)$ and
$\cF(\varphi* \k)=\cF\varphi \cF \k$ in $L_{\bC}^2(\bR^d)$.
\end{lemma}

\begin{proof}
(i) We show that $(\varphi*k, \cF \phi)=(\cF \varphi\, \cF k,\phi)$ for all $\phi \in \cS_{\bC}(\bR^d)$, or equivalently,
\begin{equation}
\label{cF-phi-k}
\int_{\bR^d}(\varphi*k)(x)\psi(x)dx=\frac{1}{(2\pi)^d} \int_{\bR^d}\cF \varphi(\xi)\cF \k(\xi)\overline{\cF \psi(\xi)}d\xi,
\end{equation}
for all $\psi \in \cS_{\bC}(\bR^d)$. This follows by \eqref{Fk-h}, since
\[
\mbox{LHS of} \ \eqref{cF-phi-k}=\int_{\bR^d}  (\widetilde{\varphi}*\psi)(y)\k(y)dy=\frac{1}{(2\pi)^d}\int_{\bR^d}
\overline{\cF(\widetilde{\varphi}*\psi)(\xi)}h(\xi)d\xi=\mbox{RHS of} \ \eqref{cF-phi-k}.
\]

(ii) We denote $v:=\cF \varphi \cF \k$. Since $|\cF \k|^2$ is tempered, there exists $\ell>0$ such that $\int_{\bR^d}(1+|\xi|^2)^{-\ell} |\cF \k(\xi)|^2d\xi<\infty$. Since $C:=\sup_{\xi}(1+|\xi|^2)^{\ell} |\cF \varphi(\xi)|^2<\infty$, we infer that:
\begin{align*}
\int_{\bR^d} |v(\xi)|^2 d\xi & =\int_{\bR^d} (1+|\xi|^2)^{\ell} |\cF \varphi(\xi)|^2\left( \frac{1}{1+|\xi|^2}\right)^{\ell} |\cF \k(\xi)|^2d\xi\\
&\leq C \int_{\bR^d}\left(\frac{1}{1+|\xi|^2}\right)^{\ell}|\cF \k(\xi)|^2d\xi<\infty. 
\end{align*} 
This shows that $v \in L_{\bC}^2(\bR^d)$. The conclusion follows by Lemma \ref{Fourier-lem}, using part (i).
\end{proof}

In view of Lemma \ref{L2-lem} and definition \eqref{def-L-hat} of the Fourier transform $\hat{L}_t$ of $L_t$, we are now ready to introduce the {\em L\'evy colored noise} $\{X_t(\varphi); t\geq 0,\varphi \in \cS(\bR^d)\}$, given by:
\begin{equation}
\label{def-X}
X_t(\varphi):=L_t(\varphi*\k)=\widehat{L}_t(\cF \varphi \cF \k), \quad \varphi \in \cS(\bR^d).
\end{equation}

Using the isometry property \eqref{isom-L1} of $L$, Plancherel's theorem, Lemma \ref{L2-lem}.(ii), and relation \eqref{Dalang-cov}, we compute the covariance of the L\'evy colored noise:
\begin{align*}
\bE[X_t(\varphi) X_s(\psi)]&=m_2(t \wedge s) \int_{\bR^d} (\varphi * \k)(x) (\psi*\k)(x) dx
\\
&=m_2 (t\wedge s) \frac{1}{(2\pi)^d} \int_{\bR^d} \cF (\varphi*\k)(\xi)\overline{\cF(\psi*\k)(\xi)}d\xi\\
&=m_2 (t\wedge s) \int_{\bR^d}\cF \varphi(\xi) \overline{\cF \psi(\xi)}\mu(d\xi) \\
&=m_2 (t\wedge s) \int_{\bR^d} \int_{\bR^d} \varphi(x)\psi(y)f(x-y)dxdy.
\end{align*}

\begin{remark}
{\rm
The L\'evy colored noise was introduced in \cite{B15} under the sole condition given by part (b) of Assumption A1 (without assuming the existence of the kernel $\k$). This condition is satisfied by the example $h(\xi)=C_H|\xi|^{1/2-H}$ with $H\in (0,1)$ in dimension $d=1$, 
for which the noise $X$ has the same spatial covariance structure as the fractional Brownian motion of index $H$. In this case, part (a) of Assumption A1 holds in the ``regular'' case $H>1/2$ (with $\k(x)=C_H'|x|^{H-3/2}$), but not in the ``rough'' case $H<1/2$. In the regular case, part (c) of Assumption A1 also holds, with $f(x)=C_{H}''|x|^{2H-2}$. Here $C_H,C_H',C_H''$ are positive constants depending on $H$. 
}
\end{remark}

To define the stochastic integral with respect to $X$, we will use the martingale measure theory developed in \cite{dalang99,walsh86}. For this, we will need the following assumption.

\medskip

\noindent
{\bf Assumption A2.} (a) $h$ is non-negative and $\k$ is continuous, non-negative, symmetric with $\k(x)<\infty$ if and only if $x\not=0$.\\
(b)  $f$ is continuous, non-negative, symmetric with $f(x)<\infty$ if and only if $x\not=0$.

\medskip

By Lemma 5.6 of \cite{KX09}, part (a) of Assumption A2 implies that for any {\em non-negative} $\varphi \in L^1(\bR^d)$, we have:
\begin{equation}
\label{Parseval-k}
\int_{\bR^d}\int_{\bR^d}\varphi(x)\varphi(y) \k(x-y)dxdy=\frac{1}{(2\pi)^d} \int_{\bR^d}|\cF \varphi(\xi)|^2 h(\xi)d\xi=:\cE_{\k}(\varphi).
\end{equation}
By Lemma A.1 of \cite{BQS}, for any $\varphi,\psi \in L^1(\bR)$ with $\cE_{\k}(|\varphi|)<\infty$ and $\cE_{\k}(|\psi|)<\infty$, we have:
\[
\int_{\bR^d} \int_{\bR^d}\varphi(x)\psi(y) \k(x-y)dxdy=\frac{1}{(2\pi)^d} \int_{\bR^d}\cF \varphi(\xi)\overline{\cF \psi(\xi)}  h(\xi)d\xi.
\]

Similarly, part (b) of Assumption A2 implies that for any {\em non-negative} $\varphi \in L^1 (\bR^d)$,
\begin{equation}
\label{Parseval-f}
\int_{\bR^d}\int_{\bR^d}\varphi(x)\varphi(y) f(x-y)dxdy=\frac{1}{(2\pi)^d}\int_{\bR^d}|\cF \varphi (\xi)|^2 |h(\xi)|^2 d\xi:=\cE_{f}(\varphi).
\end{equation}

\begin{lemma}
Under Assumptions A1 and A2, for any $A \in \cB_b(\bR^d)$, 
\[
1_{A}*\k\in L^2(\bR^d) \quad \mbox{and} \quad \cF (1_A*\k)=\cF 1_{A}\cF \k \ \mbox{in} \ L^2(\bR^d).
\]
\end{lemma}

\begin{proof}
Let $\varphi \in \cS_{\bC}(\bR^d)$ be arbitrary.  Since $\cE_{\k}(1_{A})<\infty$ and 
$\cE_{\k}(|\varphi|)<\infty$, by \eqref{Parseval-k}, 
\[
\int_{\bR^d} (1_A*\k)(x)\varphi(x)dx=\int_{\bR^d}\int_{\bR^d} 1_{A}(y)\k(x-y)\varphi(x)dxdy=\frac{1}{(2\pi)^d} \int_{\bR^d} \cF 1_{A}(\xi) \overline{\cF \varphi(\xi)} \cF \k(\xi)d\xi.
\]
This means that $\cF (1_A*\k)=\cF 1_{A}\cF \k$ in $\cS_{\bC}'(\bR^d)$; see \eqref{Fk-h}.
The conclusion will follow by Lemma \ref{Fourier-lem}, as long as we show that $v:=\cF 1_{A}\cF \k \in L^2(\bR^d)$. This is clear, since by \eqref{Parseval-f},
\begin{align*}
\int_{\bR^d}|v(\xi)|^2d\xi=\int_{\bR^d}|\cF 1_{A}(\xi)|^2 |h(\xi)|^2 d\xi=(2\pi)^d \int_{\bR^d}\int_{\bR^d}1_{A}(x)1_{A}(y) f(x-y)dxdy<\infty.
\end{align*}
\end{proof}

\underline{For the remaining part of the article, we suppose that Assumptions A1 and A2 hold.}

\medskip

We include below some examples of kernels $\k$ which satisfy these assumptions.

\begin{example} (The heat kernel)
\label{heat-ex}
{\rm Let  $\k=H_{d,\alpha/2}$ for some $\alpha>0$, where
\[
H_{d,\alpha}(x)=\frac{1}{(2\pi \alpha)^{d/2}}\exp\left(-\frac{|x|^2}{2\alpha}\right) \quad \mbox{for $\alpha>0$}.
\]
Assumptions A1 and A2 hold with $f=\k * \k=H_{d,\alpha}$ and $|\cF \k(\xi)|^2=\cF f(\xi)=\exp(-\alpha|\xi|^2/2)$.
}
\end{example}

\begin{example} (The Riesz kernel)
\label{Riesz-ex}
{\rm Let $\k=R_{d,\alpha/2}$ for some $\alpha\in (0,d)$, where
\[
R_{d,\alpha}(x)=C_{d,\alpha}|x|^{-(d-\alpha)} \quad \mbox{with} \quad C_{d,\alpha}=\pi^{-d/2}2^{-\alpha}\frac{\Gamma(\frac{d-\alpha}{2})}{\Gamma(\frac{\alpha}{2})}.
\]
It is known that $\cF R_{d,\alpha}(\xi)=|\xi|^{-\alpha}$ and $R_{d,\alpha}*R_{d,\beta}=R_{d,\alpha+\beta}$ (see \cite{stein70}). Hence, Assumptions A1 and A2 hold with $f=\k * \k=R_{d,\alpha}$ and $|\cF \k (\xi)|^2=\cF f(\xi)=|\xi|^{-\alpha}$.
}
\end{example}

\begin{example} (The Bessel kernel)
\label{Bessel-ex}
{\rm Let $\k=B_{d,\alpha/2}$ for some $\alpha>0$, where
\[
B_{d,\alpha}(x)=\frac{1}{\Gamma(\alpha/2)}\int_0^{\infty} w^{\alpha/2-1}e^{-w} \frac{1}{(4\pi w)^{d/2}}\exp\left(-\frac{|x|^2}{4w}\right)dw.
\]
It is known that $\cF B_{d,\alpha}(\xi)=(1+|\xi|^{2})^{-\alpha/2}$ and $B_{d,\alpha}*B_{d,\beta}=B_{d,\alpha+\beta}$ (see \cite{stein70}).  Hence, Assumptions A1 and A2 hold with $f=\k * \k=B_{d,\alpha}$ and $|\cF \k (\xi)|^2=\cF f(\xi)=(1+|\xi|^2)^{-\alpha/2}$.
}
\end{example}

\begin{example} (The Poisson kernel)
\label{Poisson-ex}
{\rm Let $\k=P_{d,\alpha/2}$ for some $\alpha>0$, where
\[
P_{d,\alpha}(x)=\frac{\Gamma(\frac{d+1}{2})}{\pi^{\frac{d+1}{2}}}\alpha \left( \frac{1}{|x|^2+\alpha^2}\right)^{\frac{d+1}{2}}
\]
is the density of a $d$-dimensional Cauchy distribution (see Example 2.12 of \cite{sato99}). Since $\cF P_{d,\alpha}(\xi)=\exp(-\alpha|\xi|)$ and $P_{d,\alpha}*P_{d,\beta}=P_{d,\alpha+\beta}$, Assumptions A1 and A2 hold with $f=\k * \k=P_{d,\alpha}$ and $|h(\xi)|^2=\cF f(\xi)=\exp(-\alpha|\xi|)$.
}
\end{example}

\begin{example} (Kernels of product-type)
\label{product-ex}
{\rm Let $\k(x)=\prod_{j=1}^d \k_j(x_j)$ for $x=(x_1,\ldots,x_d)$, where $\k_1,\ldots,\k_d$ are kernels on $\bR$ which satisfy Assumptions A1 and A2, with $f_j=\k_j*\widetilde{\k}_j$ for $j=1,\ldots,d$. Then $\k$ satisfies Assumptions A1 and A2 with $f(x)=\prod_{j=1}^{d}f_j(x_j)$.
}
\end{example}

Let $\cD(\bR^d)$ be the set of $C^{\infty}$-functions on $\bR^d$ with compact support, and $\cP_{0,d}(\bR^d)$ be the completion of  $\cD(\bR^d)$ with respect to the inner product
\[
\langle \varphi,\psi \rangle_{0}=m_2 \int_{\bR^d} \int_{\bR^d}\varphi(x)\psi(y) f(x-y)dxdy.
\]
(The lower index $d$ stands for ``deterministic''.) The map $\varphi \to X_t(\varphi) \in L^2(\Omega)$ is an isometry which can be extended from $\cD(\bR^d)$ to $\cP_{0,d}(\bR^d)$. 

The space $\cP_{0,d}(\bR^d)$ contains distributions: by Theorem 3.5.(1) of \cite{BGP12}, we know that $\overline{\cP}_{d}(\bR^d) \subset \cP_{0,d}(\bR^d)$, where
\[
\overline{\cP}_{d}=\{S \in \cS'(\bR^d); \cF S \ \mbox{is a function and} \ \int_{\bR^d}|\cF S(\xi)|^2 \mu(d\xi)<\infty \}.
\]

\medskip

On the other hand, for any set $A \in \cB_b(\bR^d)$, by approximating the indicator $1_A$ with functions in $\cD(\bR^d)$, we infer that $1_{A} \in \cP_{0,d}(\bR^d)$ and 
\[
X_t(A):=X_t(1_{A})=L_t(1_A*\k).
\]
Note that for any $t,s>0$ and $A,B \in \cB_b(\bR^d)$
\[
\bE[X_t(A) X_s(B)]=m_2(t \wedge s)\int_A \int_{B}f(x-y)dxdy. 
\]

We now proceed to extend the definition of the L\'evy colored noise to incorporate integrands which depend also on time.

\begin{lemma}
\label{lem-phi-L2}
Under Assumption A1, 
the map $(t,x) \mapsto \big( \varphi(t,\cdot)*\k\big)(x)$ is in $L^2(\bR_{+} \times \bR^d)$, for any $\varphi \in \cD(\bR_{+} \times \bR^d)$. 
\end{lemma}

\begin{proof}
By Lemma \ref{L2-lem}, $\varphi(t,\cdot)*\k \in L^2(\bR^d)$ and $\cF (\varphi(t,\cdot)*\k)=\cF \varphi(t,\cdot) \cF \k$ in $L^2(\bR^d)$. We denote $\phi_{\xi}(t)=\cF \varphi(t,\cdot)(\xi)$. If $t<0$, then $\varphi(t,\cdot)=0$ and $\phi_{\xi}(t)=0$. Note that $\phi_{\xi} \in L^2(\bR)$: if $\varphi(t,\cdot)$ has compact support $K$, then 
\[
|\phi_{\xi} (t)| \leq \int_{K}|\varphi(t,x)|dx \quad \mbox{and} \quad
|\phi_{\xi} (t)|^2 \leq |K|\int_{K}|\varphi(t,x)|^2dx.
\]
We apply Plancherel theorem twice, first on $\bR^d$, and then on $\bR$. We obtain:
\begin{align*}
&\int_{0}^{\infty} \int_{\bR^d}|\big( \varphi(t,\cdot)*\k\big)(x)|^2 dxdt=\frac{1}{(2\pi)^d}
\int_{0}^{\infty} \int_{\bR^d} |\cF \varphi(t,\cdot)(\xi)|^2 |\cF \k(\xi)|^2 d\xi dt\\
& \quad =\frac{1}{(2\pi)^d} \int_{\bR^d} \left(\int_{\bR}|\phi_{\xi}(t)|^2 dt\right)  |\cF \k(\xi)|^2  d\xi=\frac{1}{(2\pi)^{d+1}} \int_{\bR^{d+1}}|\cF \varphi(\tau,\xi)|^2 |\cF \k(\xi)|^2 d\xi d\tau<\infty,
\end{align*}
where for the second equality we used the fact that $\cF \phi_{\xi}(\tau)=\cF \varphi(\tau,\xi)$.
\end{proof}

In view of Lemma \ref{lem-phi-L2}, for any $\varphi \in \cD(\bR_{+} \times \bR^d)$, we define
\begin{align}
\label{def-XS}
X(\varphi)&:=\int_0^{\infty} \int_{\bR^d}\varphi(t,x) X(dt,dx)=\int_0^{\infty}\int_{\bR^d}\big( \varphi(t,\cdot)*\k\big)(x) L(dt,dx)\\
\nonumber
&=\int_0^{\infty} \int_{\bR^d}\cF \varphi(t,\cdot)(\xi) \cF \k(\xi)\widehat{L}(dt,d\xi).
\end{align}
The covariance of this process is given by:
\begin{align*}
\bE[X(\varphi)X(\psi)]&=m_2 \int_{0}^{\infty}\int_{\bR^d}
\big( \varphi(t,\cdot)*\k\big)(x)  \big( \psi(t,\cdot)*\k\big)(x) dxdt\\
&=\frac{m_2}{(2\pi)^d} \int_{0}^{\infty}\int_{\bR^d} \cF \varphi(t,\cdot)(\xi)\overline{\cF \psi(t,\cdot)(\xi)}|h(\xi)|^2 d\xi dt=:\langle \varphi,\psi\rangle_{0},
\end{align*}
 for any $\varphi,\psi\in \cD(\bR_{+} \times \bR^d)$, and coincides the covariance of the spatially homogeneous Gaussian noise considered by Dalang in \cite{dalang99}. 
 
 \medskip
 
We fix $T>0$. We let $\cP_{0,d}$ be the completion of $\cD([0,T]\times \bR^d)$ with respect to 
 $\langle \cdot,\cdot \rangle_{0}$. The map $\varphi \to X(\varphi) \in L^2(\Omega)$ is an isometry which is extended from $\cD([0,T]\times \bR^d)$ to $\cP_{0,d}$.
This notation is consistent with the discussion above. More precisely,
by approximation with functions in $\cD([0,T]\times \bR^d)$, it can be proved that for any $\varphi \in \cS(\bR^d)$, 
\[
1_{[0,t]}\varphi \in \cP_{0,d} \quad \mbox{and} \quad X(1_{[0,t]} \varphi)=X_t(\varphi),
\]
and for any $A \in \cB_b(\bR^d)$, 
\[
1_{[0,t] \times A} \in \cP_{0,d} \quad \mbox{and} \quad X(1_{[0,t]\times A}) =X_t(A).
\]

The space $\cP_{0,d}$ contains distributions in the space variable, namely $\overline{\cP}_{d} \subset \cP_{0,d}$, where $\overline{\cP}_{d}$ is the set of maps $S:[0,T] \to \cS'(\bR^d)$ such that $\cF S(t,\cdot)$ is a function for all $t \in [0,T]$, $(t,\xi)\mapsto \cF S(t,\cdot)(\xi)$ is measurable on $[0,T] \times \bR^d$, and 
\[ 
\|S\|_0^2:=m_2 \int_0^T \int_{\bR^d}|\cF S(t,\cdot)(\xi)|^2 \mu(d\xi)<\infty.
\]

For general $S \in \cP_{0,d}$, if the convolution $S(t\cdot) *\k$ is not well-defined, we cannot express $X(S)$ in terms of $L$ (as in \eqref{def-XS}). But if $S \in \overline{\cP}_d$, we still have 
\[
X(S)=\int_0^T \int_{\bR^d}\cF S (t,\cdot)(\xi) \cF \k(\xi) \widehat{L}(dt,d\xi).
\]

\section{Integration with respect to L\'evy colored noise}

The stochastic integral with respect to the noise $X$ is defined exactly as in the case of the spatially homogeneous Gaussian noise considered in \cite{dalang99}. For the sake of clarity, we review this theory briefly in this section.

\medskip

The process $\{X_t(A);t\geq 0,A \in \cB_b(\bR^d)\}$ is a worthy martingale measure, as defined in \cite{walsh86}. 
If $g$ is an elementary process of form \eqref{elem}, we define its stochastic integral with respect to $X$ by:
\[
(g \cdot X)_t:=\int_0^t \int_{\bR^d} g(s,x)X(ds,dx)=Y \Big(X_{t\wedge b}(A)-X_{t\wedge a}(A) \Big).
\]

We fix $T>0$. The map $\cE \ni g \to (g\cdot X)_T \in L^2(\Omega)$ is an isometry:
\[
\bE\big[|(g\cdot X)_T|^2\big]=m_2 \bE \left[\int_0^T \int_{\bR^d} \int_{\bR^d} g(t,x)g(t,y)f(x-y)dxdydt\right]=:\| g\|_0^2,
\]
which is extended to $\cP_0$, the completion of $\cE$ with respect to $\langle \cdot,\cdot\rangle_0$.

Let $\cP_{+}$ be the set of predictable processes $g$ such that
\[
\|g\|_{+}^{2}:=m_2  \bE \left[\int_0^T \int_{\bR^d} \int_{\bR^d} |g(t,x)g(t,y)|f(x-y)dxdydt\right]<\infty.
\]

Then $\cP_{+} \subset \cP_0$ and $\|g\|_0 \leq \|g\|_{+}$ for any $g \in \cP_{+}$.
Note that $\cP_0$ is also the completion of $\cE_0$ with respect to $\|\cdot\|_0$, where
$\cE_0$ is the set of $g \in \cP_{+}$ such that $g(\omega,t,\cdot) \in \cS(\bR^d)$ for all $(\omega,t)\in [0,T] \times \bR^d$.

The following observation will play an important role. For any $g \in \cP_{+}$ and $t \in [0,T]$,
\begin{equation}
\label{g*k1}
\big| \big(g(t,\cdot)*\k\big)(x)\big|^2=\int_{\bR^d}\int_{\bR^d}g(t,y)g(t,z)\k(x-y)\k(x-z)dydz,
\end{equation}
and therefore, by \eqref{f=k*k}, we have:
\begin{equation}
\label{g*k2}
\int_{\bR^d}\big| \big(g(t,\cdot)*\k\big)(x)\big|^2dx=\int_{\bR^d}\int_{\bR^d}g(t,y)g(t,z)f(y-z)dydz.
\end{equation}

The following result gives the relation between the stochastic integrals with respect to $X$ and $L$.
\begin{lemma}
\label{G-L-lem}
If $g \in \cP_{+}$, then $\bE\int_0^T \int_{\bR^d}\big|\big(g(t,\cdot)*\k\big)(x)\big|^2dxdt<\infty$ and
\begin{equation}
\label{g-XL}
\int_0^T \int_{\bR^d}g(t,x) X(dt,dx)=\int_0^T \int_{\bR^d} \big(g(t,\cdot)*\k\big)(x)L(dt,dx).
\end{equation}
\end{lemma}

\begin{proof}
By \eqref{g*k2},
\[
0\leq \bE\int_0^T \int_{\bR^d}\big|\big(g(t,\cdot)*\k\big)(x)\big|^2dxdt \leq \bE \int_0^T \int_{\bR^d}\int_{\bR^d} |g(t,y)g(t,z)|f(y-z)dydzdt<\infty.
\]
We know that relation \eqref{g-XL} holds if $g \in \cD([0,T] \times \bR^d)$; see \eqref{def-XS}. The general case follows by approximation, since $\cD([0,T] \times \bR^d)$ is dense in $\cP_{+}$.
\end{proof}

\begin{remark}
{\rm
Although a complete characterization of elements of $\cP_0$ may not be possible, a useful subspace (much larger than $\cP_{+}$) has been identified. By Theorem 4.7 of \cite{B15}, $\overline{\cP} \subset \cP_0$, where $\overline{\cP}$ is the set of predictable functions
$S:\Omega \times [0,T] \to \cS'(\bR^d)$  such that $\cF S(\omega,t,\cdot)$ is a function for all $(\omega,t)\in \Omega \times [0,T]$, and 
\[
\|S\|_0^2 :=m_2 \bE \left[\int_0^T \int_{\bR^d}|\cF S(t,\cdot)(\xi)|^2 \mu(d\xi)dt\right]<\infty.
\]
If $\mu$ has density $g$ and $g^{-1}1_{\{g>0\}}$ is tempered, then $\cP_0=\overline{\cP}$, by
Theorem 3.5 of \cite{BGP12}.
}
\end{remark}



We introduce now the stochastic integral with respect to the martingale measure $X^Z$, following very closely the discussion contained in \cite{dalang99}.

Let $Z=\{Z(t,x);t\in [0,T],x\in \bR^d\}$ be a predictable process such that 
\begin{equation}
\label{sup-Z}
\sup_{(t,x) \in [0,T] \times \bR^d}\bE|Z(t,x)|^2<\infty.
\end{equation}

\begin{remark}
{\rm If $g \in \cP_{+}$ and $Z$ satisfies \eqref{sup-Z}, then $gZ \in \cP_{+}$, and hence, relations \eqref{g*k1} and \eqref{g*k2} hold with $g$ replaced by $gZ$.
}
\end{remark}

We let $X^Z$ be the following martingale measure: for any $t\geq 0$ and $A \in \cB_b(\bR^d)$, 
\[
X_t^Z(A):=\int_0^t \int_{A}Z(s,x) X(ds,dx).
\]

We consider the norms:
\begin{align*}
\|g\|_{0,Z}^2 &=m_2 \bE \int_0^T \int_{\bR^d} \int_{\bR^d}g(t,x)Z(t,x)g(t,y)Z(t,y) f(x-y) dxdydt,\\
\|g\|_{+,Z}^2 &=m_2 \bE \int_0^T \int_{\bR^d} \int_{\bR^d}|g(t,x)Z(t,x)g(t,y)Z(t,y)| f(x-y) dxdydt.
\end{align*}

We let $\cP_{0,Z}$ be the completion of $\cE$ with respect to $\|\cdot\|_{0,Z}$, and $\cP_{+,Z}$ be the set of predictable processes $g$ such that $\|g\|_{+,Z}<\infty$. Then $\cP_{+,Z} \subset \cP_{0,Z}$ and $\|g\|_{0,Z} \leq \|g\|_{+,Z}$ for any $g \in \cP_{+,Z}$.

Following the same procedure as above (for the case $Z=1$), for any $g \in \cP_{0,Z}$, we define the stochastic integral of $g$ with respect to $X^Z$, starting with elementary processes and then using the isometry property. We denote:
\[
(g\cdot X^Z)_t:=\int_0^t \int_{\bR^d}g(s,x)X^Z(ds,dx)=\int_0^s \int_{\bR^d}g(s,x)Z(s,x)X(ds,dx).
\]

In addition, we assume that $Z$ satisfies the following hypothesis.

\medskip

\noindent {\bf Hypothesis A.}  For any $t>0$,
$\bE[Z(t,x)Z(t,y)]=:g_t(x-y)$ depends only on $x-y$.

\medskip

Since the function $f_t(x)=f(x)g_t(x)$ is non-negative definite, there exists a tempered measure $\mu_t$ on $\bR^d$ such that $f_t=\cF \mu_t$ in $\cS'(\bR^d)$. For any $\varphi \in \cE_{0,d}$,
\begin{align}
\label{norm-0Z}
\|\varphi\|_{0,Z}^2 &=\int_0^T \int_{\bR^d} \int_{\bR^d}\varphi(t,x)\varphi(t,y) f_t(x-y)dxdydt=\int_0^T \int_{\bR^d}| \cF \varphi(t,\cdot)(\xi)|^2 \mu_{t}^Z(d\xi)dt.
\end{align}


\begin{remark}
{\rm
The L\'evy noise has no influence on the definition of the spaces $\cP_{0},\cP_{+},\cP_{0,Z}$ and $\cP_{+,Z}$, which are {\em the same} as in \cite{dalang99}. 
}
\end{remark}

The following result is the analogue of Theorem 2 of \cite{dalang99} for the L\'evy colored noise, and contains some additional information about the integral $S \cdot X^Z$.

\begin{theorem}
\label{dalang-th2}
Let $Z=\{Z(t,x);t \in [0,T],x\in \bR^d\}$ be a predictable process for which \eqref{sup-Z} and Hypothesis A are satisfied. Let $t \mapsto S(t)$ be a deterministic function with values in the space of non-negative distributions with rapid decrease, such that
\begin{equation}
\label{Dalang-26}
\|S\|_0:=m_2 \int_0^T \int_{\bR^d}|\cF S(t,\cdot)(\xi)|^2 \mu(d\xi)dt<\infty.
\end{equation}
Then:
a) the map $(t,x) \mapsto \big(S(t,\cdot)*\k\big)(x)$ lies in $L^2([0,T]\times \bR^d)$,

b)  $S$ belongs to $\cP_{0,Z}$ and
\begin{align}
\label{isometry}
\bE|(S \cdot X^Z)_t|^2&=m_2 \int_0^t \int_{\bR^d}|\cF S(s,\cdot)(\xi)|^2 \mu_s^Z(d\xi)ds\\
&\leq \int_0^t \sup_{x\in \bR^d}\bE|Z(s,x)|^2 \int_{\bR^d}|\cF S(t,\cdot)(\xi)|^2 \mu(d\xi)ds,
\end{align}

c) there exists a predictable process $\Phi=\{ \Phi(t,x)\} \in L^2(\Omega \times [0,T] \times \bR^d)$ such that
\[
(S \cdot X^Z)_t=\int_0^t \int_{\bR^d}\Phi(s,x) L(ds,dx) \quad \mbox{for all $t \in [0,T]$}.
\]
If $S(t,\cdot) \in \cS(\bR^d)$ for all $t \in [0,T]$, then $\Phi(t,x)=\big((SZ)(t,\cdot)*k\big)(x)$.
\end{theorem}

\begin{proof}
a) By \eqref{Dalang-26}, for almost all $t\in [0,T]$, $\cF S(t,\cdot) \cF k \in L_{\bC}^2(\bR^d)$. By Lemma \ref{Fourier-lem},
$S(t,\cdot) *\k \in L^2(\bR^d)$ and $\cF(S(t,\cdot) *\k)=\cF S(t,\cdot) \cF \k$ in $L_{\bC}^2(\bR^d)$. By Plancherel's theorem,
\[
\int_0^T \int_{\bR^d}\big|\big(S(t,\cdot)*\k\big)(x)\big|^2 dxdt=\frac{1}{(2\pi)^d}\int_0^T \int_{\bR^d}|\cF S(t,\cdot)(\xi)|^2 |\cF \k(\xi)|^2 d\xi dt<\infty.
\]

b) This was proved in Theorem 2 of \cite{dalang99}, by approximating $S(t,\cdot)$ with smooth functions.

c) We use the same approximation technique as in \cite{dalang99}. We consider two cases:

\medskip

{\em Case 1. $S(t,\cdot) \in \cS(\bR^d)$ for all $t \in [0,T]$.} By  \eqref{Dalang-26}, and the fact that $S\geq 0$, we infer that $S\in \cP_{+}$, and hence $SZ \in \cP_{+}$.
We apply \eqref{g-XL} with $g$ replaced by $SZ$: 
\begin{equation}
\label{SZ-dX}
(S \cdot X^Z)_t=\int_0^t \int_{\bR^d}S(s,x) Z(s,z) X(ds,dx)=\int_0^t \int_{\bR^d} \big((SZ)(s,\cdot)*\k\big)(x)L(ds,dx).
\end{equation}

{\em Case 2. $S$ is general.} Let $S_n(t,\cdot)=S(t,\cdot)* \psi_n$, where $\psi_n(x)=n^d \psi(nx)$ and $\psi\in \cD(\bR^d)$ is such that $\psi \geq 0$, the support of $\psi$ is contained in the unit ball of $\bR^d$, and $\int_{\bR^d} \psi(x)dx=1$. Then $S_n(t,\cdot) \in \cS(\bR^d)$ and $\cF S_n(t,\cdot)=\cF S(t,\cdot) \cF \psi_n$. 

By Case 1, relation \eqref{SZ-dX} holds for $S_n$. We let $n \to \infty$. On the left hand-side, from the proof of Theorem 2 of \cite{dalang99}, we know that $\|S_n-S\|_{0,Z} \to 0$, and hence $(S_n \cdot X^Z)_t \to (S \cdot X^Z)_t$ in $L^2(\Omega)$. On the right hand-side, letting $\Psi_n(t,x)=\big((S_n Z)(t,\cdot)*k\big)(x)$, we see that $(\Phi_n)_{n\geq 1}$ is a Cauchy sequence in $L^2(\Omega \times [0,T] \times \bR^d)$, since by \eqref{g-XL},
\begin{align*}
& \|\Phi_n-\Phi_m\|_{L^2(\Omega \times [0,T] \times \bR^d)}^2 =\bE \int_0^T \int_{\bR^d}\big|
\big((S_nZ-S_mZ)(t,\cdot)*\k\big)(x) \big|^2 dxdt\\
&\quad =\bE \int_0^T \int_{\bR^d} \int_{\bR^d}(S_n-S_m)(t,y)(S_n-S_m)(t,z)Z(t,y)Z(t,z)f(y-z)dydzdt\\
&\quad = \frac{1}{m_2}\|S_n-S_m\|_{0,Z} \to 0 \quad \mbox{as $n,m \to \infty$}.
\end{align*}
If $\Phi$ is the limit of $(\Phi_n)_{n\geq 1}$ in $L^2(\Omega \times [0,T] \times \bR^d)$, then $\int \Phi_n dL \to \int \Phi dL$ in $L^2(\Omega)$.
\end{proof}

The following result gives an estimate for the moments of $(S\cdot X^Z)_t$, and is the analogue of Theorem 5 of \cite{dalang99} for the L\'evy colored noise.

\begin{theorem}
\label{dalang-th5}
Let $p\geq 2$ be such that $m_p<\infty$. Under the assumptions of Theorem \ref{dalang-th2}, suppose in addition that: 
\begin{equation}
\label{mom-Z}
\sup_{(t,x)\in [0,T] \times \bR^d}\bE|Z(t,x)|^p<\infty,
\end{equation}
\begin{equation}
\label{S-in-Lp}
\int_0^T \int_{\bR^d} \big|\big(S(t,\cdot)*\k\big)(x)\big|^p dx dt<\infty.
\end{equation}
Then, for any $t \in [0,T]$,
\begin{align*}
\bE|(S \cdot X^Z)_t|^p & \leq \cC_{p}(t) \left\{\int_0^t \sup_{x\in \bR^d} \bE|Z(s,x)|^p \left(\int_{\bR^d}|\cF S(s,\cdot)(\xi)|^2 \mu(d\xi) +\|S(s\cdot,)*\k\|_{L^p(\bR^d)}^p\right) ds\right\},
\end{align*}
where $\cC_{p}(t)=\cC_p\max(\nu_t^{p/2-1},1)$, $\cC_p$ is the constant from Proposition \ref{ros-prop}, and 
\[
\nu_t:=\int_0^t \int_{\bR^d}|\cF S(s,\cdot)(\xi)|^2 \mu(d\xi)ds.
\]
\end{theorem}

\begin{proof}
{\em Case 1.} {\em $S(t,\cdot) \in \cS(\bR^d)$ for all $t\in [0,T]$.}
We apply Proposition \ref{ros-prop} to $\Phi(t,x)=\big((SZ)(t,\cdot)*\k\big)(x)$. We obtain:
\begin{align}
\nonumber
\bE|(S \cdot X^Z)_t|^p  & \leq \cC_p \left\{ \bE \left( \int_0^t \int_{\bR^d} \big|\big((SZ)(s,\cdot)*\k\big)(x)\big|^2
dxds\right)^{p/2}+ \right.\\
\label{mom-SXZ}
& \quad \quad \quad \left. \bE \int_0^t \int_{\bR^d}
\big|\big((SZ)(s,\cdot)*\k\big)(x)\big|^p dxds \right\}=:\cC_p(T_1+T_2).
\end{align}

We treat separately the two terms. To treat $T_1$, we use relation \eqref{g*k2} for $g=SZ \in \cP_{+}$.
Then we apply the following version of H\"older's inequality: if $(X,\cX,\mu)$ is a finite measure space, then for any measurable function $h:X \to \bR$ and for any $p\geq 1$,
\begin{equation}
\label{Holder}
\left|\int_{X}h(x)\mu(dx)\right|^p \leq \big[\mu(X)\big]^{p-1} \int_{X}|h(x)|^p \mu(dx).
\end{equation}
In our case, we apply \eqref{Holder}
to the measure $\mu_t(ds,dy,dz)=S(s,y)S(s,z)f(y-z)dsdydz$ on $[0,t] \times \bR^d \times \bR^d$, whose total mass is $\nu_t$.
We obtain:
\begin{align}
\label{def-T1}
T_1 & =\bE\left| \int_0^t \int_{\bR^d} \int_{\bR^d} Z(s,y)Z(s,z) S(s,y)S(s,z) f(y-z) dydz ds\right|^{p/2} \\
\nonumber
& \leq \nu_t^{\frac{p}{2}-1}  \int_0^t \int_{\bR^d} \int_{\bR^d} E|Z(s,y)Z(s,z)|^{p/2} S(s,y)S(s,z) f(y-z) dydz ds \\
\nonumber
& \leq \nu_t^{\frac{p}{2}-1}  \int_0^t  \sup_{x\in \bR^d}E|Z(s,x)|^{p}  \int_{\bR^d} \int_{\bR^d} S(s,y)S(s,z) f(y-z) dydz ds \\
\nonumber
&=\nu_t^{\frac{p}{2}-1}  \int_0^t  \sup_{x\in \bR^d}E|Z(s,x)|^{p}  \int_{\bR^d} |\cF S(s,\cdot)(\xi)|^2 \mu(d\xi) ds.
\end{align}

To treat $T_2$, we estimate $\bE \big|\big((SZ)(s,\cdot)* \k\big)(x)\big|^p$ using Minkowski inequality:
\begin{align*}
& \Big(\bE \big|\big((SZ)(s,\cdot)* \k\big)(x)\big|^p\Big)^{1/p}=\left( \bE \left|
\int_{\bR^d}S(s,y)Z(s,y)\k(x-y)dy \right|^p \right)^{1/p} \\
&\quad =\left\| \int_{\bR^d}S(s,y)Z(s,y)\k(x-y)dy  \right\|_p \leq  \int_{\bR^d}S(s,y)\|Z(s,y)\|_p \k(x-y)dy \\
& \quad \leq 
\sup_{x\in \bR^d}\|Z(s,x)\|_p \big(S(s,\cdot)*\k\big)(x). 
\end{align*}
Hence,
\[
T_2 \leq \int_0^t \sup_{x\in \bR^d}\bE|Z(s,x)|^p \int_{\bR^d}\big|\big(S(s,\cdot)*\k\big)(x)\big|^pdx ds.
\]

\medskip

{\em Case 2. $S$ is general.} Let $S_n$ be as in the proof of Theorem \ref{dalang-th2} (Case 2).
 By Case 1,
\begin{align*}
& \bE|(S_n \cdot X^Z)_t|^p   \\
& \quad \leq \cC_{p}^{(n)}(t) \left\{\int_0^t \sup_{x\in \bR^d} \bE|Z(s,x)|^p \left(\int_{\bR^d}|\cF S_n(s,\cdot)(\xi)|^2 \mu(d\xi)+\|S_n(s,\cdot)*\k\|_{L^p(\bR^d)}^p\right) ds  \right\},
\end{align*}
where $\cC_p^{(n)}(t)=\cC_p \max(\nu_t^{(n)},1)$ and $\nu_t^{(n)}=\int_0^t \int_{\bR^d}|\cF S_n(s,\cdot)(\xi)|^2 \mu(d\xi)ds$. For any $n\geq 1$, $|\cF S_n(s,\cdot)|\leq |\cF S(s,\cdot)|$ and hence, $\nu_t^{(n)}\leq \nu_t$
and $\cC_p^{(n)}(t)\leq \cC_p(t)$. 
By Young's inequality,
\[
\|S_n(s,\cdot)*\k\|_{L^p(\bR^d)}=\|S(s,\cdot)*\psi_n*\k\|_{L^p(\bR^d)}\leq \|S(s,\cdot)*\k\|_{L^p(\bR^d)}.
\]
We know that $S_n \cdot X^Z \to S \cdot X^Z$ in $L^2(\Omega)$. Along a subsequence, this convergence happens in the almost sure sense. Therefore, by invoking Fatou's lemma, we obtain:
\begin{align*}
& \bE|(S\cdot X^Z)_t|^p \leq \liminf_{n}  \bE|(S \cdot X^Z)_t|^p \\
& \quad \leq \cC_{p}(t) \left\{\int_0^t \sup_{x\in \bR^d} \bE|Z(s,x)|^p \left(\int_{\bR^d}|\cF S(s,\cdot)(\xi)|^2 \mu(d\xi)+\|S(s,\cdot)*\k\|_{L^p(\bR^d)}^p\right) ds  \right\}.
\end{align*}
\end{proof}

The next result is a variant of Theorem \ref{dalang-th5} in which the constant does not depend on $t$. This will be used in Section \ref{section-no-drift} to obtain exponential bounds for the $p$-moments of the solution of the non-linear equation \eqref{nonlin-eq} with $b=0$.

\begin{theorem}
\label{dalang-th5-new}
Under the assumptions of Theorem \ref{dalang-th5}, for any $t \in [0,T]$,
\begin{align*}
\bE|(S \cdot X^Z)_t|^p & \leq \cC_p \left\{\left(\int_0^t \sup_{x\in \bR^d} \|Z(s,x)\|_p^2 \int_{\bR^d}|\cF S(s,\cdot)(\xi)|^2 \mu(d\xi) ds \right)^{p/2}+ \right. \\
& \quad \quad \quad \ \left. \int_0^t \sup_{x \in \bR^d}\bE|Z(s,x)|^p  \, \|S(s\cdot,)*\k\|_{L^p(\bR^d)}^p ds\right\},
\end{align*}
where $\cC_p$ is the constant from Proposition \ref{ros-prop}.
\end{theorem}

\begin{proof}
{\em Case 1.} {\em $S(t,\cdot) \in \cS(\bR^d)$ for all $t\in [0,T]$.}
We use again \eqref{mom-SXZ}, but now we employ a different method for estimating $T_1$. Using \eqref{def-T1}, Minkowski's inequality and Cauchy-Schwarz inequality, we see that
\begin{align*}
T_1^{2/p}& =\left\| \int_0^t \int_{\bR^d} \int_{\bR^d} S(s,y)S(s,z)Z(s,y)Z(s,z)f(y-z)dydzds \right\|_{p/2}\\
& \leq \int_0^t \int_{\bR^d} \int_{\bR^d} S(s,y)S(s,z)\|Z(s,y)Z(s,z)\|_{p/2} f(y-z)dydzds \\
& \leq  \int_0^t \int_{\bR^d} \int_{\bR^d} S(s,y)S(s,z)\|Z(s,y)\|_{p}\|Z(s,z)\|_{p} f(y-z)dydzds\\
& \leq \int_0^t \sup_{x\in \bR^d}\|Z(s,x)\|_p^2 \int_{\bR^d}\int_{\bR^d}S(s,y)S(s,z) f(y-z)dydzds\\
&=\int_0^t \sup_{x\in \bR^d}\|Z(s,x)\|_p^2 \int_{\bR^d} |\cF S(t,\cdot)(\xi)|^2 \mu(d\xi)ds.
\end{align*}

\medskip

{\em Case 2.} {\em $S$ is general.} We proceed as in the proof of Theorem \ref{dalang-th5} (Case 2).
\end{proof}

\begin{remark}
\label{remark-lin}
{\rm The estimates given by Theorems \ref{dalang-th5} and \ref{dalang-th5-new} require the additional condition \eqref{S-in-Lp}, compared to Theorem 5 of \cite{dalang99}. This is to be expected: if  $Z=1$ and $S \in \cP_{+}$, then by \eqref{g-XL},
$$(S\cdot X)_t=\int_{0}^t \int_{\bR^d} \big(S(s,\cdot)*\k\big)(x) L(ds,dx).$$
By Corollary \ref{p-mom-int}, $\bE|(S \cdot X)_t|^p<\infty$ if and only if $(s,x)\mapsto (S(s,\cdot)*\k)(x)$ is in $L^p([0,t]\times \bR^d)$ and $m_p<\infty$.
}
\end{remark}

\section{The linear equation: existence of moments}
\label{section-lin}

In this section, we consider the linear equation:
\begin{equation}
\label{linear-eq}
\cL v(t,x) =\dot{X}(t,x) \quad t\geq 0, x\in \bR^d,
\end{equation}
with zero initial conditions, where $\cL$ is a second order partial differential operator with coefficients that do not depend on the space variable. 
The goal of this section is to give some sufficient conditions for the solution $v$ to have finite $p$-th moments, for $p\geq 2$. In Sections \ref{subsection-heat} and \ref{subsection-wave}, we will give the range of values $p$ for which this happens, in the case when $\cL$ is the heat operator, respectively the wave operator, for three examples of kernels $\k$: the heat kernel, the Riesz kernel, and the Bessel kernel. 

\medskip

We assume that the fundamental solution $G_t$ of $\cL u=0$ is a {\em distribution with rapid decrease} on $\bR^d$.
This implies that the Fourier transform $\cF G_t$ is a $C^{\infty}$-function, and for any $S \in \cS'(\bR^d)$, $G_t * S $ is a well-defined distribution in $\cS'(\bR^d)$ given by 
$(G_t * S, \varphi)=(S, \widetilde{\varphi}*G_t)$, and $\cF (G_t*S)=\cF G_t \cF S$ in $\cS_{\bC}'(\bR^d)$. 
In particular, $\cF (G_t*\k)=\cF G_t \cF \k $ in $\cS_{\bC}'(\bR^d)$.

\medskip

We assume that the distribution $S(t,\cdot)=G_t$ satisfies \eqref{Dalang-26}, i.e. 
\begin{equation}
\label{Dalang-26-1}
\int_0^{T} \int_{\bR^d} |\cF G_{t}(\xi)|^2 \mu(d\xi) dt<\infty \quad \mbox{for any $T>0$}.
\end{equation} 
This implies that for any $t \geq 0$,
\begin{equation}
\label{Dalang-26-2}
\int_0^{t} \int_{\bR^d} |\cF G_{t-s}(x-\cdot)(\xi)|^2 \mu(d\xi) ds<\infty,
\end{equation}
which means that $G_{t-\cdot}(x-\cdot)1_{[0,t]}$ lies in $\overline{\cP}_d$, and hence in $\cP_{0,d}$. 
By analogy with the classical D'Alembert formula, we introduce the following definition. See also Definition \ref{def-11}.

\begin{definition}
{\rm The process $\{v(t,x);t \geq 0,x\in \bR^d\}$ given by
\[
v(t,x):=\int_0^t \int_{\bR^d}G_{t-s}(x-y)X(ds,dy)
\]
is called the {\em (mild) solution} of equation \eqref{linear-eq}. }
\end{definition}

By Theorem \ref{dalang-th2}, the map $(s,y)\mapsto \big(G_{t-s}(x-\cdot)*\k\big)(y)$ lies in $L^2([0,t] \times \bR^d)$.
By approximating $G_{t-s}$ with smooth functions and using \eqref{def-XS}, we deduce that
\begin{equation}
\label{def-v-L}
v(t,x)=\int_0^t \int_{\bR^d} \big(G_{t-s}(x-\cdot)*\k\big)(y)L(ds,dy).
\end{equation}

The process $\{v(t,x);x \in \bR^d\}$ is strictly stationary (by Lemma \ref{property-S} below), and hence $\bE|v(t,x)|^p$ does not depend on $x$.
By the isometry property,
\[
\bE|v(t,x)|^2=m_2 \int_0^t \int_{\bR^d}|\cF G_{t-s}(\xi)|^2 \mu(d\xi)ds=m_2 \int_0^t \int_{\bR^d}|(G_{t-s}*\k)(y)|^2 dyds.
\]

Let $p\geq 2$ be such that $m_p<\infty$.
By Corollary \ref{p-mom-int}, $\bE|v(t,x)|^p<\infty$ if and only if
\[
\cM_{p}(t):=\int_0^t \int_{\bR^d} |(G_{t-s}*\k)(y)|^{p}dyds<\infty.
\]

Using \eqref{def-v-L}, and applying Proposition \ref{ros-prop}, we get:
\[
\bE|v(t,x)|^p  \leq \cC_p \left\{ \left( \int_0^t \int_{\bR^d} \big|(G_{t-s} * \k)(y)\big|^2 dy ds\right)^{p/2}+ \int_0^t \int_{\bR^d} \big|(G_{t-s} * \k)(y)\big|^p dy ds \right\}.
\]
We denote
\[
J_p(t):=\|G_t*\k\|_{L^p(\bR^d)}^2.
\]
Then 
the previous estimate can be written as:
\begin{align}
\nonumber
\bE|v(t,x)|^p & \leq
\cC_p\left\{ \left(\int_0^t J_2(s)ds \right)^{p/2} +\int_0^t (J_p(s))^{p/2} ds\right\}\\
\label{UB-mom-p}
& =\cC_p \left\{\big(\cM_2(t)\big)^{p/2}+\cM_p(t)\right\}.
\end{align}

\begin{remark}
{\rm
Recall that $\mu(d\xi)=\frac{1}{(2\pi)^d}|\cF \k(\xi)|^2 d\xi$. Then, by Plancherel's theorem,
\[
J_2(t)=\int_{\bR^d}\big|(G_t*\k)(x)\big|^2 dx=\int_{\bR^d} |\cF G_t(\xi)|^2 \mu(d\xi).
\] 
}
\end{remark}

To estimate $J_p(t)$, we consider several examples of kernels $\k$. The first class of examples are those that satisfy the following assumption.

\medskip

\noindent
{\bf Assumption A3.}
There exists a function $K_p:\bR^d \to [0,\infty]$ such that
\[
\k^{p/2} * \widetilde{\k}^{p/2}=K_p^{p/2} \quad \mbox{and} \quad \cF K_p=\mu_p \ \mbox{in $\cS_{\bC}'(\bR^d)$},
\]
where $\mu_p$ is a non-negative tempered measure on $\bR^d$.

\medskip

Note that Assumption A3 automatically holds for $p=2$, with $K_2=f$ and $\mu_2=\mu$. 

\begin{example} (The heat kernel)
\label{ex-heat-2}
{\rm Let $\k=H_{d,\alpha/2}$ be the heat kernel from Example \ref{heat-ex}. Then $\k^{p/2}=H_{d,\alpha/2}^{p/2}=C H_{d,\alpha/p}$, and $\k^{p/2}*\k^{p/2}=C H_{d,2\alpha/p}=C H_{d,\alpha}^{p/2}$. Hence, Assumption A3 holds for all $p>0$, with $K_p=C H_{d,\alpha}$ and $\mu_p(d\xi)=C\exp(-\alpha|\xi|^2/2)$. Here $C>0$ is a constant depending on $(d,\alpha,p)$ which may be different in each of its appearances.
}
\end{example}

\begin{theorem}
\label{Jpt-th}
a) If $\k$ satisfies Assumption A3 for some $p\geq 2$, then 
\[
J_p(t) \leq J_p'(t):=\int_{\bR^d}|\cF G_t(\xi)|^2 \mu_p(d\xi),
\]
provided that $J_p'(t)<\infty$.

b) If $\k=R_{d,\alpha/2}$ is the Riesz kernel of order $\alpha/2$ (given by Example \ref{Riesz-ex}) for some $\alpha \in (0,d)$, and $G_t \in L^q(\bR^d)$ for some $q>1$, then for all $t>0$,
\[
J_p(t) \leq C_{d,\alpha,p}\|G_t\|_{L^q(\bR^d)}^2, \quad \mbox{with $\frac{1}{p}=\frac{1}{q}-\frac{\alpha}{2d}$},
\]
where
$C_{d,\alpha,p}>0$ is a constant depending on $(d,\alpha,p)$.

c) If $\k=B_{d,\alpha/2}$ is the Bessel kernel of order $\alpha/2$ (given by Example \ref{Bessel-ex}) for some $\alpha>0$, and $G_t \in L^p(\bR^d)$ for some $p\geq 1$, then for all $t>0$,
\[
J_p(t) \leq C_{d,\alpha,p}\|G_t\|_{L^p(\bR^d)}^2,
\]
where $C_{d,\alpha,p}>0$ is a constant depending on $(d,\alpha,p)$.
\end{theorem}

\begin{proof}
a) {\em Case 1.} Assume that $G_t \in \cS(\bR^d)$. We write
\[
|(G_t*\k)(x)|^2=\int_{\bR^d}\int_{\bR^d}G_t(y)G_t(z)\k(x-y)\k(x-z)dydz.
\]
By Minkowski's inequality, for all $p\geq 2$,
\begin{align*}
J_p(t)&= \left(\int_{\bR^d}|(G_t*\k)(x)|^pdx\right)^{2/p}
=\left\| \int_{\bR^d}\int_{\bR^d}G_t(y)G_t(z)\k(\cdot-y)\k(\cdot-z)dydz  \right\|_{L^{p/2}(\bR^d)}\\
&\leq \int_{\bR^d}\int_{\bR^d}G_t(y)G_t(z)\left\|\k(\cdot-y)\k(\cdot-z) \right\|_{L^{p/2}(\bR^d)} dydz \\
&=\int_{\bR^d}\int_{\bR^d}G_t(y)G_t(z) K_p(y-z)dydz=J_p'(t).
\end{align*}

{\em Case 2.} In the general case, we consider the approximation $G_t^{(n)}=G_t*\psi_n$, where $\psi_n(x)=n^d\psi(nx)$ for some non-negative function $\psi \in \cD(\bR^d)$ with support contained in the unit ball, and $\int_{\bR^d}\psi(x)dx=1$. Then $G_t^{(n)}\in \cS(\bR^d)$ for all $n\geq 1$. Hence, by Case 1,
\begin{equation}
\label{case1-A3}
\|G_t^{(n)}*\k\|_{L^p(\bR^d)}^2 \leq \int_{\bR^d}|\cF G_t^{(n)}(\xi)|^2 \mu_p(d\xi) \quad \mbox{for all $n\geq 1$}.
\end{equation}
Note that $G_t^{(n)}*\k=G_t *\k * \psi_n \to G_t *\k$ pointwise. By Fatou's lemma and \eqref{case1-A3}, 
\begin{align*}
\|G_t*\k\|_{L^p(\bR^d)}^2 & \leq \liminf_{n\to \infty} \|G_t^{(n)}*\k\|_{L^p(\bR^d)}^2 \\
& \leq 
\liminf_{n\to \infty}  \int_{\bR^d}|\cF G_t^{(n)}(\xi)|^2 \mu_p(d\xi)= \int_{\bR^d}|\cF G_t(\xi)|^2 \mu_p(d\xi),
\end{align*}
where the last equality is due to the dominated convergence theorem. 

b) This follows by Littlewood-Hardy-Sobolev Theorem (Theorem 1, page 119 of \cite{stein70}).

c) This estimate can be found on page 134 of \cite{stein70}.
\end{proof}

\begin{remark}
{\rm
Theorem \ref{Jpt-th} allows us to find {\em sufficient} conditions which ensure that the solution $v$ of equation \eqref{linear-eq} has a finite $p$-th moment, for $p>2$. Because we do not have a lower bound for $J_p(t)$, we do not know if these conditions are necessary.
}
\end{remark}

\subsection{The linear heat equation}
\label{subsection-heat}

Consider the heat equation with additive noise:
\begin{equation}
\label{heat-add}
\frac{\partial v}{\partial t}(t,x)=\frac{1}{2}\Delta v(t,x)+\dot{X}(t,x) \quad t>0,x\in \bR^d,
\end{equation}
and zero initial conditions. In this case,
\[
G_t(x)=\frac{1}{(2\pi t)^{d/2}}\exp\left(-\frac{|x|^2}{2t} \right) \quad \mbox{and} \quad \cF G_t(\xi)=\exp\left(-\frac{t|\xi|^2}{2}\right).
\]

The solution of equation \eqref{heat-add} exists provided that $G_t$ satisfies \eqref{Dalang-26-1}, which is equivalent to Dalang's condition \eqref{Dalang-cond}, as shown in Section 3 of \cite{dalang99}.

\begin{remark}
{\rm Condition \eqref{Dalang-cond} always holds when $d=1$, and for $d\geq 2$ can be expressed in terms of the function $f$. See Remark 10.(b) of \cite{dalang99}.
}
\end{remark}

We now examine the existence of the higher moments of this solution for the three examples of kernels $\k$ mentioned above. 

\begin{example} (The heat kernel)
\label{heat-heat}
{\rm Let $k=H_{d,\alpha/2}$ be the heat kernel given by Example \ref{heat-ex} for some $\alpha>0$.
Then $\mu(d\xi)=(2\pi)^{-d}\exp(-\alpha|\xi|^2/2)d\xi$ and Dalang's condition holds for all $\alpha>0$. By Example \ref{ex-heat-2}, Assumption A3 holds for all $p>0$. By Theorem \ref{Jpt-th}.(a), for all $p\geq 2$ and $t>0$,
\[
J_p(t) \leq J_p'(t)=C\int_{\bR^d}e^{-t|\xi|^2}e^{-\frac{\alpha|\xi|^2}{2}}d\xi=C\left(t+\frac{\alpha}{2}
\right)^{-d/2} \leq C \alpha^{-d/2},
\]
and hence $\cM_p(t)=\int_0^t (J_p(s))^{p/2}ds\leq C$. By 
\eqref{UB-mom-p}, for any $d\geq 1$, $t>0$ and $x \in \bR^d$,
\[
\bE|v(t,x)|^p\leq C \quad \mbox{for any $p\geq 2$ such that $m_p<\infty$}.
\] 
}
\end{example}

\begin{example} (The Riesz kernel)
\label{heat-Riesz}
{\rm Let $k=R_{d,\alpha/2}$ be the Riesz kernel given by Example \ref{Riesz-ex} for some $\alpha\in (0,d)$. Then $\mu(d\xi)=(2\pi)^{-d}|\xi|^{-\alpha}d\xi$ and Dalang's condition holds for all $\alpha \in ((d-2)\vee 0,d)$. By Theorem \ref{Jpt-th}.(b), for any $p\geq 2$, letting $\frac{1}{q}=\frac{1}{p}+\frac{\alpha}{2d}$, we have:
\[
J_p(t) \leq C \|G_t\|_{L^q(\bR^d)}^2=Ct^{d(1-q)/q}=C t^{d(\frac{1}{p}+\frac{\alpha}{2d}-1)}.
\]
(The condition $q>1$ is equivalent to $p>\frac{2d}{2d-\alpha}$, which holds for any $p\geq 2$.)
Therefore,
\[
\cM_p(t)=\int_0^t (J_p(s))^{p/2}ds \leq C\int_0^t s^{\frac{dp}{2}(\frac{1}{p}+\frac{\alpha}{2d}-1)} ds=C t^{\frac{d(1-p)}{2}+\frac{\alpha p}{4}+1},
\]
provided that $\frac{d(1-p)}{2}+\frac{\alpha p}{4}+1>0$. By \eqref{UB-mom-p},
for any $d\geq 1$, $t>0$ and $x\in \bR^d$,
\[
\bE|v(t,x)|^p \leq C \Big(  t^{\frac{p}{2}(1-\frac{d-\alpha}{2})}  + t^{\frac{d(1-p)}{2}+\frac{\alpha p}{4}+1}\Big) \quad \mbox{for any $2\leq p<\frac{2d+4}{2d-\alpha}$ such that $m_p<\infty$}.
\]
}
\end{example}

\begin{example} (The Bessel kernel)
\label{heat-Bessel}
{\rm Let $k=B_{d,\alpha/2}$ be the Bessel kernel given by Example \ref{Bessel-ex} for some $\alpha>0$. Then $\mu(d\xi)=(2\pi)^{-d}(1+|\xi|^2)^{-\alpha/2}d\xi$ and Dalang's condition holds for all $\alpha >d-2$. By Theorem \ref{Jpt-th}.(c), for any $p\geq 2$ 
\[
J_p(t) \leq C \|G_t\|_{L^p(\bR^d)}^2=Ct^{d(1-p)/p}.
\]
Therefore,
\[
\cM_p(t)=\int_0^t (J_p(s))^{p/2}ds \leq C\int_0^t s^{\frac{d(1-p)}{2}} ds=C t^{\frac{d(1-p)}{2}+1},
\]
provided that $p<1+\frac{2}{d}$. Since $p\geq 2$, this forces $d=1$. If $d=1$, 
for any $t>0$ and $x\in \bR$, 
\[
\bE|v(t,x)|^p \leq C\Big(t^{\frac{p}{4}} + t^{\frac{3-p}{2}}\Big) \quad \mbox{for any $2\leq p<3$ such that $m_p<\infty$}. 
\]

}

\end{example}

\subsection{The linear wave equation}
\label{subsection-wave}

Consider the wave equation with additive noise:
\begin{equation}
\label{wave-add}
\frac{\partial^2 v}{\partial t^2}(t,x)=\Delta v(t,x)+\dot{X}(t,x) \quad t>0,x\in \bR^d,
\end{equation}
and zero initial conditions. 

We recall that the fundamental solution of the wave equation on $\bR_{+} \times \bR^d$ is given by:
\begin{eqnarray*}
G_t(x)&=& \frac{1}{2}1_{\{|x| <t\}} \quad \mbox{if} \ d=1,\\
G_t(x)&=& \frac{1}{2\pi}\frac{1}{\sqrt{t^2-|x|^2}}1_{\{|x|<t\}} \quad \mbox{if} \ d=2,\\
G_t&=&\frac{1}{4\pi t}\sigma_t, \quad \mbox{if} \ d=3,\\\
\end{eqnarray*}
where $\sigma_t$ is the surface measure on the sphere $\{x \in \bR^3; |x|=t\}$.
If $d=1$ or $d=2$, $G_t$ is a non-negative function in $L^1(\bR^d)$, and if $d=3$, $G_t$ is a finite measure in $\bR^3$.

If $d \geq 4$ is even, $G_t$ is a distribution with compact support in $\bR^d$ given by:
$$G_t= \frac{1}{1 \cdot 3 \cdot \ldots \cdot (d-1)}\left(\frac{1}{t}\frac{\partial}{\partial t}
\right)^{(d-2)/2}(t^{d-1}\Upsilon_t), \quad \Upsilon_t(\varphi)=\frac{1}{\omega_{d+1}}\int_{B(0,1)}
\frac{\varphi(ty)}{\sqrt{1-|x|^2}} dx,$$
and if
$d \geq 5$ is odd, $G_t$ is a distribution with compact support in $\bR^d$ given by:
$$G_t= \frac{1}{1 \cdot 3 \cdot \ldots \cdot (d-2)}\left(\frac{1}{t}\frac{\partial}{\partial t}
 \right)^{(d-3)/2}(t^{d-2}\Sigma_t), \quad \Sigma_t(\varphi)=\frac{1}{\omega_d}\int_{\partial B(0,1)}\varphi(tz)d\sigma(z),$$
where $\omega_d$ is the surface area of the unit sphere $\partial B(0,1)$ in $\bR^d$, and $\sigma$ is the surface measure on $\partial B(0,1)$
(see e.g. Theorem (5.28), page 176 of \cite{folland95}).

It is known that for any $d \geq $1, the Fourier transform of $G_t$ is given by:
\begin{equation}
\label{Fourier-G}
\cF G_t(\xi)=\frac{\sin(t|\xi|)}{|\xi|}, \quad \xi \in \bR^d.
\end{equation}

The solution of equation \eqref{wave-add} exists provided that $G_t$ satisfies \eqref{Dalang-26-1}, which is again equivalent to Dalang's condition \eqref{Dalang-cond}.
We examine the existence of the higher moments of this solution for three examples of kernels $\k$.

\begin{example} (The heat kernel)
\label{wave-heat}
{\rm Let $k=H_{d,\alpha/2}$ be the heat kernel given by Example \ref{heat-ex} for some $\alpha>0$.
Dalang's condition holds for all $\alpha>0$. 
By Theorem \ref{Jpt-th}.(a) and the inequality $|\sin(x)|\leq |x|$, for all $p\geq 2$ and $t>0$,
\[
J_p(t) \leq J_p'(t)=C\int_{\bR^d}\frac{\sin^2(t|\xi|)}{|\xi|^2}e^{-\frac{\alpha|\xi|^2}{2}}d\xi\leq Ct^2,
\]
and $\cM_p(t)=\int_0^t (J_p(s))^{p/2}ds \leq C \int_0^t s^p ds=Ct^{p+1}$. By \eqref{UB-mom-p}, for any $d\geq 1$, $t>0$ and $x \in \bR^d$,
\[
\bE|v(t,x)|^p\leq C \Big( t^{\frac{3p}{2}}+ t^{p+1} \Big) \quad \mbox{for any $p \geq 2$ such that $m_p<\infty$}. 
\]
}
\end{example}

\begin{example} (The Riesz kernel)
\label{wave-Riesz}
{\rm Let $k=R_{d,\alpha/2}$ be the Riesz kernel given by Example \ref{Riesz-ex} for some $\alpha\in (0,d)$. Recall that Dalang's condition holds for all $\alpha \in ((d-2)\vee 0,d)$. But to apply Theorem \ref{Jpt-th}.(b), $G_t$ has to be a function, so we can only consider $d=1$ or $d=2$. 

If $d=1$, for any $t>0$ and $p \geq 2$, letting $\frac{1}{q}=\frac{1}{p}+\frac{\alpha}{2}$ (and noticing that $q>1$), 
\[
J_p(t) \leq C \|G_t\|_{L^q(\bR)}^2=C t^{2/q}=C t^{\frac{2}{p}+\alpha},
\]
and $\cM_p(t)=\int_0^t (J_p(s))^{p/2}ds \leq C\int_0^t s^{\frac{\alpha p}{2}+1} ds=C t^{\frac{\alpha p}{2}+2}$.
By \eqref{UB-mom-p}, for any $t>0$ and $x \in \bR^d$, 
\[
\bE|v(t,x)|^p \leq \quad C \Big( t^{\frac{p}{2}(\alpha+2)}+ t^{\frac{\alpha p}{2}+2} \Big) \quad \mbox{for any $p\geq 2$ such that $m_p<\infty$}.
\]

If $d=2$, $\|G_t\|_{L^q(\bR^2)}=C t^{(2-q)/q}$ if $q\in (0,2)$ and $\|G_t\|_{L^q(\bR^2)}=\infty$ if $q\geq 2$. For any $p\geq 2$, letting $\frac{1}{q}=\frac{1}{p}+\frac{\alpha}{4}$ (and noticing that $q>1$), we have:
\[
J_p(t) \leq C \|G_t\|_{L^q(\bR^2)}^2=C t^{2(2-q)/q}=Ct^{\frac{4}{p}+\alpha-2},
\]
provided that $q<2$, which is equivalent to $p<\frac{4}{2-\alpha}$. In this case,
\[
\cM_p(t)=\int_0^t (J_p(s))^{p/2}ds \leq C\int_0^t s^{2+\frac{\alpha p}{2}-p} ds=C t^{3+\frac{\alpha p}{2}-p},
\]
since $3+\frac{\alpha p}{2}-p>0$ (which is equivalent to $p<\frac{6}{2-\alpha}$). By \eqref{UB-mom-p}, for any $t>0$ and $x \in \bR^2$,
\[
\bE|v(t,x)|^p \leq C \Big( t^{\frac{p}{2}(\alpha+1)} +t^{3+\frac{\alpha p}{2}-p}\Big)  \quad \mbox{for any $2\leq p<\frac{4}{2-\alpha}$ such that $m_p<\infty$}.
\]
}
\end{example}

\begin{example} (The Bessel kernel)
\label{wave-Bessel}
{\rm Let $k=B_{d,\alpha/2}$ be the Bessel kernel given by Example \ref{Bessel-ex} for some $\alpha>0$. Dalang's condition holds for all $\alpha >d-2$. But to apply
Theorem \ref{Jpt-th}.(c), $G_t$ must be a function, so we can only consider $d=1$ and $d=2$.

If $d=1$, for any $p\geq 2$ and $t>0$,
\[
J_p(t) \leq C \|G_t\|_{L^p(\bR^d)}^2=Ct^{2/p},
\]
and $\cM_p(t)=\int_0^t (J_p(s))^{p/2}ds \leq C\int_0^t s ds=C t^2$,
By \eqref{UB-mom-p}, for any $t>0$ and $x \in \bR^d$,
\[
\bE|v(t,x)|^p \leq C \Big( t^{p}+t^2 \Big) \quad \mbox{for all $p\geq 2$}.
\]

If $d=2$, $\|G_t\|_{L^p(\bR^d)}=\infty$ for all $p>2$, and we cannot apply Theorem \ref{Jpt-th}.(c). In this case, we do not know if $\bE|v(t,x)|^p <\infty$ for some $p>2$. 
}
\end{example}

\section{The non-linear equation}

In this section, we study the non-linear equation \eqref{nonlin-eq} with constant initial condition $\eta \in \bR$.

\subsection{Existence of solution and moments}
\label{section-with-drift}

The goal of this section is to prove the existence of the solution under very general assumptions on $G_t$ (similar to those given in \cite{dalang99} for the Gaussian case), and to present some sufficient conditions under which the solution has a finite $p$-th moment, for $p\geq 2$. We then analyze these conditions for the heat equation and the wave equation, in the case of three examples of kernels $\k$: the heat kernel, the Riesz kernel, and the Bessel kernel.

\medskip


As in Section \ref{section-lin}, we assume that the fundamental solution $G_t$ of $\cL u=0$ is a distribution with rapid decrease on $\bR^d$.
In addition, we impose the following hypothesis, which is listed as Hypothesis C in \cite{dalang99-err}. 

\medskip

\noindent
{\bf Hypothesis B.} a) For any $t>0$, $G_t$ is a non-negative finite measure on $\bR^d$ satisfying
\[
\sup_{t \in [0,T]}G_t(\bR^d)<\infty \quad \mbox{for all $T>0$},
\]
and the distribution induced by this measure (denoted also by $G_t$) has rapid decrease and satisfies \eqref{Dalang-26-1};\\
b) $t \mapsto \cF G_t(\xi)$ is continuous on $[0,\infty)$, for any $\xi \in \bR^d$;\\
c) there exists a non-negative distribution $k_t$ with rapid decrease and $\varepsilon>0$ such that 
\[
|\cF G_{t+h}(\xi)-\cF G_t(\xi)| \leq |\cF k_t(\xi)| \quad \mbox{for all $h \in [0,\varepsilon],t>0,\xi \in \bR^d$},
\]
and $\int_0^T \int_{\bR^d} |\cF k_t(\xi)|^2 \mu(d\xi)dt<\infty$ for any $T>0$.

\medskip

\begin{remark}
\label{rem-HypB}
{\rm
Hypothesis B holds is satisfied by the fundamental solution of the heat equation in any dimension $d\geq 1$, and by  fundamental solution of the wave equation in dimension $d\in \{1,2,3\}$, provided that $\mu$ satisfies Dalang's condition \eqref{Dalang-cond}.
}
\end{remark}

\begin{theorem}
\label{th13-p2}
If Hypothesis B holds, then equation \eqref{nonlin-eq} with constant initial condition $\eta \in \bR$ and globally Lipschitz functions $b$ and $\sigma$, has a unique solution $u$. Moreover, $u$ is $L^2(\Omega)$-continuous and satisfies
\begin{equation}
\label{u-2nd-mom}
\sup_{t \in [0,T]} \sup_{x\in \bR^d}\bE|u(t,x)|^2<\infty.
\end{equation}
In particular, this holds for the stochastic heat equation in any dimension $d\geq 1$, and the stochastic wave equation in dimension $d\in \{1,2,3\}$, provided that $\mu$ satisfies Dalang's condition \eqref{Dalang-cond}; if $d=1$, no extra condition is required, since \eqref{Dalang-cond} always holds.
\end{theorem}

\begin{proof}
The proof follows using the same argument as in the proof of Theorem 13 of \cite{dalang99}, using the Picard iteration sequence: $u_0(t,x)=\eta$,
\[
u_{n+1}(t,x)=\eta+\int_0^t \int_{\bR^d}G_{t-s}(x-y)\sigma\big(u_n(s,y)\big) X(ds,dy)+\int_0^t \int_{\bR^d}b\big(u_n(t-s,x-y)\big)G_{s}(dy)ds.
\]
The only part which requires some new justification is the fact that $Z_n(t,x)=\sigma(u(t,x))$ satisfies Hypothesis A, i.e.
$\bE[Z_n(t,x)Z_n(t,y)]$ depends only on $x-y$. This is a consequence of the fact that $u_n$ has property (S), which is shown in the proof of Lemma \ref{property-S} below.
\end{proof}

We recall the following definition from \cite{dalang99}.

\begin{definition}
{\rm
Let $Z=\{Z(t,x);t\geq 0,x\in \bR^d\}$ be a process defined on the same probability space as $X$.
We say that $Z$ {\em has property (S)} (with respect to $X$) if the finite-dimensional distributions of
\[
\left\{\{Z^{(z)}(t,x);t\geq 0,x\in \bR^d\},\{X_t^{(z)}(B);t \geq 0,B \in \cB_b(\bR^d) \}\right\}
\]
do not depend on $z\in \bR$, where $Z^{(z)}(t,x)=Z(t,x+z)$ and $X_t^{(z)}(B)=X_t(B+z)$.
}
\end{definition}

\begin{lemma}
\label{spat-homog}
The process $X$ is spatially homogeneous, in the sense that the finite-dimensional distributions of $\{X_t^{(z)}(A);t \geq 0,A \in \cB_b(\bR^d)\}$ do not depend on $z$.
\end{lemma}
 
\begin{proof}
By Lemma 5.2 of \cite{BZ25},  the finite-dimensional distributions of $\{L_t^{(z)}(A);t \geq 0,A \in \cB_b(\bR^d)\}$ do not depend on $z$. Then, for any $t>0$ and $A \in \cB_b(\bR^d)$,
\begin{align*}
X_t^{(z)}(A)&=X_t(A+z) =L_t(1_{A+z}*\k)=\int_0^t \int_{\bR^d} (1_{A}*\k)(x-z)L(ds,dx)\\
&=\int_0^t \int_{\bR^d} (1_A*\k)(x) L^{(z)}(ds,dx)\stackrel{d}{=}\int_0^t \int_{\bR^d}(1_A *\k)(x) L(ds,dx)=X_t(A),
\end{align*}
where $\stackrel{d}{=}$ denotes equality in distribution. The same argument shows that $\big(X_{t_1}^{(z)}(A_1),\ldots, \linebreak X_{t_k}^{(z)}(A_k)\big)
\stackrel{d}{=}\big(X_{t_1}(A_1),\ldots,X_{t_k}(A_k)\big)$ for any $t_1,\ldots,t_k \in \bR_{+}$ and $A_1,\ldots,A_k \in \cB_b(\bR^d)$.
\end{proof}

\begin{lemma}
\label{property-S}
The solution $u=\{u(t,x);t\geq 0,x \in \bR^d\}$ of equation \eqref{nonlin-eq} has property (S). In particular, for any $(t_1,x_1),\ldots,(t_k,x_k)\in \bR_{+} \times \bR$ and for any $z \in \bR$,
\begin{equation}
\label{station}
\big(u(t_1,x_1+z),\ldots,u(t_k,x_k+z)\big) \stackrel{{\rm d}}{=}\big(u(t_1,x_1),\ldots,u(t_k,x_k)\big),
\end{equation}
where $\stackrel{d}{=}$ denotes equality in distribution. In particular, the process $\{u(t,x)\}_{x\in \bR^d}$ is strictly stationary, for any $t\geq 0$.
\end{lemma}

\begin{proof}
Since $u_n(t,x) \to u(t,x)$ in $L^2(\Omega)$ for all $(t,x)\in \bR_{+}\times \bR^d$, it suffices to show that $u_n$ has property (S), for any $n \geq 0$.
We will prove this by induction on $n\geq 0$. 

For $n=0$, the result is clear, since $u_0$ does not depend on $x$, and $X$ is spatially homogeneous, by Lemma \ref{spat-homog}. The induction step follows as in Lemma 18 of \cite{dalang99}.
\end{proof}

We now examine the question of existence of higher order moments of the solution, using Theorem \ref{dalang-th5}. The following result is the analogue of Theorem 13 of \cite{dalang99}.

\begin{theorem}
Suppose that Hypothesis B holds. Let $p\geq 2$ be such that $m_p<\infty$. If $\cM_p(t)<\infty$ for all $t>0$, then the solution $u$ of equation \eqref{nonlin-eq} satisfies:
\begin{equation}
\label{u-mom-p}
\sup_{(t,x)\in [0,T] \times \bR^d} \bE|u(t,x)|^p<\infty \quad \mbox{for all} \quad T>0.
\end{equation}
\end{theorem}

\begin{proof}
Let $(u_n)_{n\geq 0}$ be the sequence of Picard iterations defined in the proof of Theorem \ref{th13-p2}. Note that
\begin{align*}
u_{n+1}(t,x)-u_n(t,x)&= \int_0^t \int_{\bR^d} G_{t-s}(x-y) \Big(\sigma\big(u_n(s,y)\big)-\sigma\big(u_{n-1}(s,y)\big)\Big) X(ds,dy)\\
& \quad +\int_0^t \int_{\bR^d} \Big(b\big(u_n(t-s,x-y)\big)-b\big(u_{n-1}(t-s,x-y)\big)\Big)G_s(dy)ds
\\
&=:U_{n+1}(t,x)+V_{n+1}(t,x).
\end{align*}

We treat separately the two terms. For the first term, by Theorem \ref{dalang-th5}, we have:
\begin{align*}
& \bE|U_{n+1}(t,x)|^p  \leq \cC_{p}(t) \int_0^t \sup_{y \in \bR^d}\bE\big|\s\big(u_n(s,y)\big)-\s\big(u_{n-1}(s,y)\big)\big|^p \\
& \qquad \qquad \qquad \quad \quad \quad \quad \left( \int_{\bR^d}|\cF G_{t-s}(\xi)|^2 \mu(d\xi)+ \|G_{t-s}*\k\|_{L^p(\bR^d)}^p\right)ds\\
& \quad \leq \cC_{p}(t) {\rm Lip}^p(\s)\int_0^t \sup_{y \in \bR^d}\bE\big|u_n(s,y)-u_{n-1}(s,y)\big|^p \Big(J_2(t-s)+\big(J_p(t-s)\big)^{p/2} \Big)ds.
\end{align*}

For the second term, using H\"older's inequality, letting $\nu_t=\int_0^t G_s(\bR^d)ds$, we have:
\begin{align*}
\bE|V_{n+1}(t,x)|^p & \leq \nu_t^{p-1} \bE \int_0^t \int_{\bR^d}
\Big|b\big(u_n(t-s,x-y)\big)-b\big(u_{n-1}(t-s,x-y)\big)\Big|^p G_s(dy)ds\\
& \leq \nu_{t}^{p-1}{\rm Lip}^p(b) \sup_{s \in [0,t]} G_s(\bR^d) \int_0^t \sup_{y \in \bR^d}\bE|u_{n}(s,y)-u_{n-1}(s,y)|^p ds.
\end{align*}
Letting $H_{n}(t)=\sup_{x \in \bR^d}\bE|u_{n}(t,x)-u_{n-1}(t,x)|^p$ for any $n \geq 0$ (with $u_{-1}=0$), we obtain:
\[
H_{n+1}(t) \leq \cC_p'(t) \int_0^t H_n(s)\Big(J_2(t-s)+\big(J_p(t-s)\big)^{p/2} +1\Big)ds,
\]
where $\cC_p'(t)=\max\{\cC_{p}(t) {\rm Lip}^p(\s),\nu_{t}^{p-1}{\rm Lip}^p(b)\sup_{s \in [0,t]} G_s(\bR^d)\}$.

By Lemma 15 of \cite{dalang99}, $\sum_{n\geq 0}\sup_{t \leq T}\big(H_n(t)\big)^{p/2}<\infty$.
Hence $\{u_n(t,x)\}_{n\geq 0}$ is a Cauchy sequence in $L^p(\Omega)$, uniformly in $(t,x)\in [0,T]\times \bR^d$. Its limit is the solution $u$. 
\end{proof}

\begin{example}[Heat Equation] 
{\rm Let $\cL=\frac{\partial}{\partial t}-\frac{1}{2}\Delta$ be the heat operator in dimension $d\geq 1$ and $u$ the solution of equation \eqref{nonlin-eq}. 
We recall the estimates for $\cM_p(t)$ given in Examples \ref{heat-heat}, \ref{heat-Riesz}, \ref{heat-Bessel} for each of the three kernels below.

a) If $\k=H_{d,\alpha/2}$ for some $\alpha>0$, then \eqref{u-mom-p} holds for any $p\geq 2$ such that  $m_p<\infty$.

b) If $\k=R_{d,\alpha/2}$ for some $\alpha\in ((d-2)\vee 0,d)$, then \eqref{u-mom-p} holds for any $2\leq p< \frac{2d+4}{2d-\alpha}$ such that  $m_p<\infty$.

c) If $d=1$ and $\k=B_{1,\alpha/2}$ for some $\alpha>0$, then \eqref{u-mom-p} holds for any $2\leq p<3$ such that  $m_p<\infty$.
}
\end{example}

\begin{example}[Wave Equation] 
{\rm Let $\cL=\frac{\partial^2}{\partial t^2}-\Delta$ be the wave operator in dimension $d\in \{1,2,3\}$ and $u$ the solution of equation \ref{nonlin-eq}. 
We recall the estimates for $\cM_p(t)$ given in Examples \ref{wave-heat}, \ref{wave-Riesz}, \ref{wave-Bessel} for the kernels below.

a) If $\k=H_{d,\alpha/2}$ for some $\alpha>0$, then \eqref{u-mom-p} holds for any $p\geq 2$ such that  $m_p<\infty$.

b) If $d=1$ and $\k=R_{1,\alpha/2}$ for some $\alpha \in (0,1)$, then \eqref{u-mom-p} holds for any $p\geq 2$ such that  $m_p<\infty$. If $d=2$ and $\k=R_{2,\alpha/2}$ for some $\alpha \in (0,2)$, then \eqref{u-mom-p} holds for any $2\leq p<\frac{4}{2-\alpha}$ such that  $m_p<\infty$.

c) If $d=1$ and $\k=B_{1,\alpha/2}$ for some $\alpha>0$, then \eqref{u-mom-p} holds for any $p\geq 2$ such that  $m_p<\infty$.
}
\end{example}

\subsection{Exponential bounds for moments}
\label{section-no-drift}

In this section, we consider the non-linear equation without drift:
\begin{equation}
\label{nonlin-no-drift}
\cL u(t,x) =\sigma(u(t,x)) \dot{X}(t,x), \quad t\geq 0,x\in \bR^d
\end{equation}
with constant initial condition $\eta \in \bR$ and  globally Lipschitz function $\sigma$.
The goal of this section is to provide an exponential upper bound for the $p$-th moments of the solution to equation \eqref{nonlin-no-drift}, for values $p\geq 2$ for which these moments exist.
For this, we use the method introduced in \cite{FK09, K14} for the Gaussian case.

\bigskip

For any $\beta>0$ and $p \geq 2$, we let $\cL^{\beta,p}$ be the set of predictable processes $\Phi=\{\Phi(t,x); t \geq 0,x \in \bR^d\}$ such that
$\|\Phi\|_{\beta,p}<\infty$, where
$$\|\Phi\|_{\beta,p}:=\sup_{t \geq 0} \sup_{x \in \bR^d}\
\Big(e^{-\beta t}\|\Phi(t,x)\|_p\Big).$$
It can be proved that $\cL^{\beta,p}$ equipped with $\|\cdot\|_{\beta,p}$ is a Banach space. (We identify processes $\Phi_1$ and $\Phi_2$, if $\Phi_1$ is a modification of $\Phi_2$.)

\medskip

For any predictable random field $Z=\{Z(t,x);t\geq 0,x\in \bR^d\}$, we let
\[
(G*Z)(t,x)=\int_0^t \int_{\bR^d}G_{t-s}(x-y)Z(s,y)X(ds,dy)=\int_0^t \int_{\bR^d} G_{t-s}(x-y)X^Z(ds,dy),
\]
provided that $G_{t-\cdot}(x-\cdot)1_{[0,t]} \in \cP_{0,Z}$.

We begin with some preliminary results. Recall that $J_p(t)=\|G_t*\k\|_{L^p(\bR^d)}^{2}$.

\begin{proposition}[Stochastic Young Inequality]
\label{young}
Let $p\geq 2$ be such that $m_p<\infty$. Suppose that $G_t$ satisfies the following conditions: \\
a) $G_t$ is a non-negative distribution with rapid decrease such that \eqref{Dalang-26-1} holds; \\
b) $(t,x)\mapsto (G_t *\k)(x) \in L^p([0,T] \times \bR^d)$ for any $T>0$;\\
c) there exists $\beta>0$ such that $A_{\beta,p}<\infty$, where
\begin{equation}
\label{def-A}
A_{\beta,p}:=\left(\int_0^{\infty} e^{-2\beta t} J_2(t) dt\right)^{p/2} +\int_0^{\infty}e^{-p\beta t} \big(J_p(t)\big)^{p/2}dt.
\end{equation}

Then, for any predictable process $Z=\{Z(t,x);t\geq 0,x \in \bR^d\}$ which satisfies Hypothesis A and \eqref{mom-Z} for all $T>0$, we have:
 $$\|G * Z\|_{\beta,p}^p \leq \cC_p A_{\beta,p}\|Z\|_{\beta,p}^p,$$
where $\cC_p$ is the constant from Proposition \ref{ros-prop}.
\end{proposition}

\begin{proof}
We proceed as in the proof of Proposition 3.2 of \cite{BN16} (for the L\'evy white noise $L$), which was in turn inspired by the 
proof of Proposition 5.2 of \cite{K14} (for the space-time Gaussian white noise). Condition a) implies that $G_{t-\cdot}(x-\cdot)1_{[0,t]} \in  \cP_{0,Z}$, by Theorem \ref{dalang-th2}, and hence $G*Z$ is well-defined. Condition b) allows us to apply Theorem \ref{dalang-th5-new} to the distribution $S(s,\cdot)=G_{t-s}(x-\cdot)1_{[0,t]}(s)$ to deduce the following estimate:
\begin{align*}
& \bE|(G*Z)(t,x)|^p  \\
& \quad \leq \cC_p \left\{ \left( \int_0^t \sup_{y \in \bR^d} \|Z(s,y)\|_p^2  J_2(t-s)ds \right)^{p/2} + \int_0^t \sup_{y\in \bR^d} \bE|Z(s,y)|^p \big(J_p(t-s)\big)^{p/2} ds\right\} \\
& \quad \leq  \cC_p  \left\{ \left( \|Z\|_{\beta,p}^2 \int_0^t  e^{2\beta s} J_2(t-s)ds \right)^{p/2} + \|Z\|_{\beta,p}^p \int_0^t e^{p\beta s} \big(J_p(t-s)\big)^{p/2} ds\right\}\\
& \quad = \cC_p \|Z\|_{\beta,p}^p \left\{ \left(  \int_0^t e^{2\beta (t-s)} J_2(s)ds \right)^{p/2} + \int_0^t e^{p\beta (t-s)} \big(J_p(s)\big)^{p/2} ds\right\}\\
& \quad =\cC_p \|Z\|_{\beta,p}^p e^{p\beta t} \left\{ 
\left(  \int_0^t e^{-2\beta s} J_2(s)ds \right)^{p/2} + \int_0^t e^{-p\beta s} \big(J_p(s)\big)^{p/2} ds\right\} \\
& \quad \leq \cC_p \|Z\|_{\beta,p}^p e^{p\beta t} A_{\beta,p}.
\end{align*}
Hence $e^{-p \beta t}\bE|(G*Z)(t,x)|^p \leq \cC_p A_{\beta,p} \|Z\|_{\beta,p}^p $ for all $t \geq 0$ and $x \in \bR^d$. 
\end{proof}

\begin{theorem}
\label{Lyap-th}
Suppose that Hypothesis B holds. Let $p\geq 2$ be such that $m_p<\infty$. Suppose that $(t,x) \mapsto (G_t *\k)(t,x) \in L^p([0,T] \times \bR^d)$ for all $T>0$, and there exists $\beta>0$ such that $A_{\beta,p}<\infty$,
\begin{equation}
\label{A-1}
{\rm Lip}^p(\sigma) \, \cC_p A_{\beta,p}<1 \quad \mbox{and} \quad
\|v\|_{\beta,p}<\infty,
\end{equation}
where $A_{\beta,p}$ is given by \eqref{def-A}, $\cC_p$ is the constant from Proposition \ref{ros-prop}, and $v$ is the solution of the linear equation \eqref{linear-eq}. If $u$ is the solution of equation \eqref{nonlin-no-drift}, then,
\begin{equation}
\label{Lyap}
\frac{1}{p} \limsup_{t\to \infty}\frac{1}{t}\log \bE|u(t,x)|^p \leq \beta^*:=\inf\{\beta>0; {\rm Lip}^p(\sigma)
\cC_p A_{\beta,p}<1\}.
\end{equation}
\end{theorem}

\begin{proof}
Let $(u_n)_{n\geq 0}$ be the sequence of Picard iterations defined in the proof of Theorem \ref{th13-p2} with $b=0$. Then $u_{n+1}-u_n=G * \big(\s(u_n)-\s(u_{n-1})\big)$ for any $n\geq 0$, where  $u_{-1}=0$. By Proposition \ref{young} and the Lipschitz property of $\s$, for any $n\geq 0$,
\[
\|u_{n+1}-u_n\|_{\beta,p} \leq \big(\cC_p A_{\beta,p}\big)^{1/p}\|\s(u_n)-\s(u_{n-1})\|_{\beta,p} \leq \gamma \|u_n-u_{n-1}\|_{\beta,p},
\]
where $\gamma:={\rm Lip}(\s)\big(\cC_p A_{\beta,p}\big)^{1/p} $. Hence, $\|u_{n+1}-u_{n}\|_{\beta,p}\leq \gamma^n \|u_1-u_0\|_{\beta,p}$ for any $n\geq 0$. Note that
$u_1-u_0=\s(\eta)v$ and $\|u_1-u_0\|_{\beta,p}=\s(\eta)\|v\|_{\beta,p}<\infty$.

Since $\gamma<1$, $\sum_{n\geq 1}\|u_{n+1}-u_n\|_{\beta,p}<\infty$ and hence, $(u_n)_{n\geq 1}$ is a Cauchy sequence in $\cL_{\beta,p}$. The limit of $(u_n)_{n\geq 1}$ in $\cL_{\beta,p}$ is the solution to equation \eqref{nonlin-no-drift}. Hence,
\[
K_{\beta,p}:=\|u\|_{\beta,p}^p=\sup_{t\geq 0}\sup_{x\in \bR^d} \big( e^{-p\beta t}\bE|u(t,x)|^p \big)<\infty,
\]
and $\frac{1}{t}\log \bE|u(t,x)|^p \leq \frac{1}{t}\log K_{\beta,p}+\beta p$ for any $t>0$ and $x \in \bR^d$. Letting $t\to \infty$, we get:
\[
\limsup_{t\to \infty}\frac{1}{t}\log \bE|u(t,x)|^p \leq \beta p.
\]
\end{proof}

In the next examples, 
we apply Theorem \ref{Lyap-th} to the case of the stochastic heat and wave equations, for three kernels: a) the heat kernel; b) the Riesz kernel; c) the Bessel kernel. We check that the hypotheses of this theorem are satisfied:
\begin{description}
\item[(i)] Hypothesis B holds by Remark \ref{rem-HypB}. 
\item[(ii)] The fact that $(t,x) \mapsto (G_t *\k)(t,x) \in L^p([0,T] \times \bR^d)$ follows from the estimates for $\cM_p(t)$ given in Examples \ref{heat-heat}, \ref{heat-Riesz}, \ref{heat-Bessel} (for the heat equation), respectively Examples \ref{wave-heat}, \ref{wave-Riesz}, \ref{wave-Bessel} (for the wave equation),  for values $p\geq 2$ in a certain range. 
\item[(iii)] The estimates obtained for $\bE|v(t,x)|^p$  in these examples show that $\|v\|_{\beta,p}<\infty$ for all $\beta>0$, for values $p\geq 2$ in the same respective ranges, for which $m_p<\infty$. 
\item[(iv)] The bounds that we present below for $A_{\beta,p}$ guarantee that the first condition in \eqref{A-1} holds, if $\beta$ is small enough. (For these bounds, we use the estimates for $J_p(t)$ obtained in the above-mentioned examples.)
\end{description}

\begin{example}[Heat Equation] 
\label{heat-ex-p}
{\rm Let $\cL=\frac{\partial}{\partial t}-\frac{1}{2}\Delta$ be the heat operator 
in dimension $d\geq 1$ 
and $u$ the solution of equation \eqref{nonlin-eq} with $b=0$. 

\medskip

a) If $\k=H_{d,\alpha/2}$ for some $\alpha>0$, then for any $\beta>0$ and $p\geq 2$,
\[
A_{\beta,p}\leq C \left\{ \left(\int_0^{\infty}e^{-2\beta t}dt \right)^{p/2}+\int_{0}^{\infty}e^{-p\beta t}dt\right\}\leq C\left\{\frac{1}{\beta^{p/2}} +\frac{1}{\beta}\right\}.
\]
In this case, \eqref{Lyap} holds for any $d\geq 1$, and for any $p\geq 2$ such that $m_p<\infty$.

\medskip

b) If $\k=R_{d,\alpha/2}$ for some $\alpha\in ((d-2)\vee 0,d)$, then for any $\beta>0$ and $p\geq 2$,
\begin{align*}
A_{\beta,p} & \leq C \left\{ \left( \int_0^t e^{-2\beta t} t^{-\frac{d-\alpha}{2}} dt\right)^{p/2}+\int_0^{\infty}e^{-p\beta t} t^{\frac{dp}{2}(\frac{1}{p}+\frac{\alpha}{2d}-1)}dt\right\}\\
& =C\left(
\frac{1}{\beta^{\frac{p}{2}(1-\frac{d-\alpha}{2})}}+\frac{1}{\beta^{\frac{d(1-p)}{2}+\frac{\alpha p}{4}+1}} \right).
\end{align*}
In this case, \eqref{Lyap} holds for any $d\geq 1$, and for any $2\leq p< \frac{2d+4}{2d-\alpha}$ such that $m_p<\infty$.

\medskip

c) If $\k=B_{1,\alpha/2}$ for some $\alpha>0$, then for any $\beta>0$ and $p\geq 2$, 
\begin{align*}
A_{\beta,p} & \leq C \left\{ \left( \int_0^t e^{-2\beta t} t^{-\frac{1}{2}} dt\right)^{p/2}+\int_0^{\infty}e^{-p\beta t} t^{\frac{1-p}{2}}dt\right\}=C\left(
\frac{1}{\beta^{p/4}}+\frac{1}{\beta^{\frac{1-p}{2}+1}} \right).
\end{align*}
In this case, \eqref{Lyap} holds in dimension $d=1$, for any $2\leq p< 3$ such that $m_p<\infty$.

}
\end{example}

\begin{example}[Wave Equation] 
\label{wave-ex-p}
{\rm Let $\cL=\frac{\partial^2}{\partial t^2}-\Delta$ be the wave operator in dimension $d\in \{1,2,3\}$
 and $u$ the solution of equation \eqref{nonlin-eq} with $b=0$. 

\medskip

a) If $\k=H_{d,\alpha/2}$ for some $\alpha>0$, then for any $\beta>0$ and $p\geq 2$, 
\[
A_{\beta,p}\leq C \left\{ \left(\int_0^{\infty}e^{-2\beta t} t^{2}dt \right)^{p/2}+\int_{0}^{\infty}e^{-p\beta t}t^pdt\right\}\leq C\left\{\frac{1}{\beta^{3p/2}} +\frac{1}{\beta^{p+1}}\right\}.
\]
In this case, \eqref{Lyap} holds for any $d \in \{1,2,3\}$, and for any $p\geq 2$ such that $m_p<\infty$.

\medskip

b) If $d=1$ and $\k=R_{1,\alpha/2}$ for some $\alpha \in (0,1)$, then for any $\beta>0$ and $p\geq 2$, 
\[
A_{\beta,p}\leq C \left\{ \left(\int_0^{\infty}e^{-2\beta t} t^{\alpha+1}dt \right)^{p/2}+\int_{0}^{\infty}e^{-p\beta t}t^{\frac{\alpha p}{2}+1}dt\right\}\leq C\left\{\frac{1}{\beta^{p(\alpha+2)/2}} +\frac{1}{\beta^{\frac{\alpha p}{2}+2}}  \right\}.
\]
In this case, \eqref{Lyap} holds for any $p\geq 2$ such that $m_p<\infty$.

If $d=2$ and $\k=R_{2,\alpha/2}$ for some $\alpha \in (0,2)$, then for any $\beta>0$ and $2\leq p<\frac{4}{2-\alpha}$, 
\[
A_{\beta,p}\leq C \left\{ \left(\int_0^{\infty}e^{-2\beta t} t^{\alpha}dt \right)^{p/2}+\int_{0}^{\infty}e^{-p\beta t}t^{\frac{4}{p}+\alpha-2} dt\right\}\leq C\left\{\frac{1}{\beta^{p(\alpha+1)/2}} +\frac{1}{\beta^{\frac{4}{p}+\alpha-1}}  \right\}.
\]
In this case, \eqref{Lyap} holds for any $2 \leq p<\frac{4}{2-\alpha}$ such that $m_p<\infty$.

\medskip

c) If $d=1$ and $\k=B_{1,\alpha/2}$ for some $\alpha>0$, then for any $\beta>0$ and $p\geq 2$,
\[
A_{\beta,p}\leq C \left\{ \left(\int_0^{\infty}e^{-2\beta t} t dt \right)^{p/2}+\int_{0}^{\infty}e^{-p\beta t}t dt\right\}\leq C\left\{\frac{1}{\beta^{p}} +\frac{1}{\beta^{2}} \right\}.
\]
In this case, \eqref{Lyap} holds for any $p\geq 2$ such that $m_p<\infty$.}
\end{example}

\section{Intermittency}
\label{section-interm}

In this section, we provide a lower bound for the second moment of the solution of the non-linear equation \eqref{nonlin-no-drift} (without drift). We will assume that $\cL$ is the heat operator in any dimension $d\geq 1$, or the wave operator in dimension $d\leq 2$.
We show that, under some conditions, 
\[
\overline{\gamma}(2)>0,
\]
where $\overline{\gamma}(2)$ is the upper Lyapounov exponent of order 2, defined by \eqref{def-gamma}.
In the next section, in the case when $\sigma(u)=\lambda u$, we will provide an explicit expression for $\overline{\gamma}(2)$, and we will show that in some cases, this is actually a limit.

Our result is similar to Theorem 1.8 of \cite{FK13}, which treats a more general parabolic equation:
\begin{equation}
\label{par-eq}
\partial_t u=\cL^* u+\s(u) \dot{W},
\end{equation}
driven by a spatially-homogeneous Gaussian noise $\dot{W}$, which is white in time and has the spatial covariance function $f$. In \eqref{par-eq}, $\cL^*$ is the generator of a $d$-dimensional L\'evy process $(X_t)_{t\geq 0}$, and the fundamental solution $G_t^*$ of the equation $\partial_t u=\cL^* u$ is the density of $X_t$ at time $t$. In particular, if $\cL^*=\frac{1}{2}\Delta$, $(X_t)_{t\geq 0}$ is $d$-dimensional standard Brownian motion, and \eqref{par-eq} becomes the heat equation.

Our method covers simultaneous the heat and wave equations, based on the observation that the Laplace transform:
\begin{equation}
\label{def-Ibeta}
\cI_{\beta}(\xi)=\int_{0}^{\infty}e^{-\beta t} |\cF G_t(\xi)|^2 dt, \quad \xi \in \bR^d,\beta>0,
\end{equation}
has a similar form for both equations: (see Lemma 2.2 of \cite{BS19})
\begin{align}
\label{cI-heat}
\cI_{\beta}(\xi)&=\frac{1}{\beta+|\xi|^2} \quad \mbox{for the heat equation}, \\
\label{cI-wave}
\cI_{\beta}(\xi)&=\frac{1}{2\beta} \cdot \frac{1}{\frac{\beta^2}{4}+|\xi|^2} \quad \mbox{for the wave equation}.
\end{align}

For any $\pmb{a}=(a_1,\ldots,a_d) \in \bR_{+}^d$, we let $[\pmb{a},2\pmb{a}]=[a_1,2a_1] \times \ldots \times [a_d,2a_d]$, and we define
\begin{equation}
\label{def-upsilon}
\Upsilon_{\pmb a}(\beta)=\int_{[\pmb{a},2\pmb{a}]}\frac{1}{\beta+|\xi|^2}\mu(d\xi) \quad \mbox{for any $\beta \geq 0$}.
\end{equation}

We are now ready to state the main result of this section. We emphasize that for this result, $\sigma$ has to be non-negative.

\begin{theorem}
\label{th-Lyap2}
Suppose that the measure $\mu$ satisfies Dalang's condition \eqref{Dalang-cond}. Let $u$ be the solution of equation \eqref{nonlin-no-drift} with constant initial condition $\eta \in \bR$, where $\cL$ is the heat operator in any dimension $d\geq 1$, or the wave operator in dimension $d\leq 2$. 

Suppose that there exists a constant $L_{\sigma}>0$ such that 
\begin{equation}
\label{sigma-cond}
\s(u) \geq L_{\s}|u| \quad \mbox{for any $u \in \bR$}.
\end{equation}

In addition, assume that the following conditions are satisfied:\\
a) $\cF \k(\xi)$ depends on $\xi=(\xi^{(1)},\ldots,\xi^{(d)})$ only through $|\xi^{(1)}|,\ldots,
|\xi^{(d)}|$;\\
b) if $|\xi^{(k)}| \leq |\eta^{(k)}|$ for any $k=1,\ldots,d$, then $\cF \k(\xi) \geq \cF \k(\eta)$;\\
c) there exists $\pmb{a} \in \bR_{+}^d$ such that 
\begin{align}
\label{upsilon-cond}
\Upsilon_{\pmb a}(0) & > \frac{1}{m_2 L_{\sigma}^2} \quad \mbox{for the heat equation}\\
\nonumber
\Upsilon_{\pmb a}(0) & >0 \quad \quad \quad \mbox{for the wave equation}.
\end{align}
Then $\overline{\gamma}(2)>0$.
\end{theorem}

\begin{proof} We use the same argument as in the proof of Theorem 1.8 of \cite{FK13}. See also  correction given in \cite{FK13-err}.
For any $t>0$ and $x,y\in \bR^d$,
\begin{align*}
&\bE\big[\big|u(t,x)u(t,x)\big| \big] \geq \bE[u(t,x)u(t,y)]\\
&=\eta^2+m_2 \int_0^t \int_{(\bR^d)^2} G_{t-s}(x-z)G_{t-s}(y-z') f(z-z')\bE\big[\s\big(u(s,z)\big) \s\big(u(s,z')\big)\big]dzdz'ds\\
& \geq \eta^2+m_2 L_{\s}^2 \int_0^t \int_{(\bR^d)^2} G_{t-s}(x-z)G_{t-s}(y-z') f(z-z') \bE\big[\big|u(s,z)u(s,z')\big|\big]  dzdz'ds,
\end{align*}
where the equality is due to the isometry property \eqref{isometry}, and relation \eqref{norm-0Z}.
We multiply this inequality by $e^{-\beta t}$ (for fixed $\beta>0$), and we integrate $dt$ on $\bR_{+}$. Using the notation:
\begin{align}
\label{def-H-beta}
H_{\beta}(x,y)&=\int_{0}^{\infty}e^{-\beta t} \bE\big[ \big|u(t,x) u(t,y)\big| \big] dt, \\
\nonumber
G_{\beta}(x,y)&=\int_{0}^{\infty}e^{-\beta t} G_t(x) G_t(y)dt,\\
\nonumber
(\cA_{\beta}h)(x,y)&=(Fh *G_{\beta})(x,y) \quad \mbox{with} \quad F(x,y)=f(x-y),
\end{align}
we obtain:
\[
H_{\beta}(x,y) \geq \frac{\eta^2}{\beta} + m_2 L_{\s}^2 (\cA_{\beta}H_{\beta})(x,y).
\]
Iterating this inequality, and letting ${\pmb 1}(x,y)=$ for all $x,y \in \bR^d$, we obtain:
\begin{equation}
\label{LB-H}
H_{\beta} \geq \frac{\eta^2}{\beta} +\frac{\eta^2}{\beta}\sum_{n\geq 1} m_2^n L_{\s}^{2n} \cA_{\beta}^n \pmb{1}.
\end{equation}
A direct calculation shows that
\[
\cA_{\beta}^n \pmb{1} (x,y)=\int_{(\bR^d)^n} e^{-i (\sum_{j=1}^n \xi_j)\cdot (x-y)} \prod_{j=1}^{n}\cI_{\beta}(\xi_1+\ldots+\xi_j) \mu(d\xi_1)\ldots \mu(d\xi_n),
\]
where $\cI_{\beta}(\xi)$ is given by \eqref{def-Ibeta}. Applying \eqref{LB-H} for $x=y$, we obtain:
\begin{align}
\label{H-beta-xx}
H_{\beta}(x,x) \geq \frac{\eta}{\beta^2} +\frac{\eta}{\beta^2}\sum_{n\geq 0} m_2^{n} L_{\s}^{2n} \int_{(\bR^d)^n} \prod_{j=1}^{n}\cI_{\beta}(\xi_1+\ldots+\xi_j) \mu(d\xi_1)\ldots \mu(d\xi_n).
\end{align}

To find a lower bound for the integral above, we proceed as in the correction \cite{FK13}.
Recall that $\mu$ has density $g=|\cF \k|^2$. Using the change of variables $\eta_j=\xi_1+\ldots+\xi_j$ for $j=1,\ldots,n$, and the convention $\eta_0=0$, we have:
\begin{align*}
Q_n(\beta)& :=\int_{(\bR^d)^n} \prod_{j=1}^{n}\cI_{\beta}(\xi_1+\ldots+\xi_j) \mu(d\xi_1)\ldots \mu(d\xi_n)\\
&=\int_{(\bR^d)^n} \prod_{j=1}^{n}\cI_{\beta}(\xi_1+\ldots+\xi_j) \prod_{j=1}^{n}g(\xi_j) d\xi_1 \ldots d\xi_n \\
& =\int_{(\bR^d)^n} \prod_{j=1}^{n}\cI_{\beta}(\eta_j) \prod_{j=1}^{n}g(\eta_j-\eta_{j-1}) d\eta_1 \ldots d\eta_n \\
& \geq \int_{[\pmb{a},2\pmb{a}]^n} \prod_{j=1}^{n}\cI_{\beta}(\eta_j) \prod_{j=1}^{n}g(\eta_j-\eta_{j-1}) d\eta_1 \ldots d\eta_n.
\end{align*}
If $\eta_j \in [\pmb{a},2\pmb{a}]$, then $a_k \leq \eta_j^{(k)} \leq 2a_k$ for all $k=1,\ldots,d$. Hence, $|\eta_j^{(k)}-\eta_{j-1}^{(k)}| \leq |\eta_j^{(k)}|$  for all $k=1,\ldots,d$  and $j=1,\ldots,n$. By hypothesis b), $g(\eta_j-\eta_{j-1}) \geq g(\eta_j)$ for all $j\leq n$, and so
\[
Q_n(\beta) \geq \left(\int_{[\pmb{a},2\pmb{a}]} \cI_{\beta}(\xi)\mu(d\xi)\right)^n. 
\]

Recalling definition \eqref{def-upsilon} of $\Upsilon_a(\beta)$ and expressions \eqref{cI-heat} and \eqref{cI-wave} for $\cI_{\beta}(\xi)$ in the case of heat equation, respectively wave equation, we infer that
\begin{align*}
H_{\beta}(x,x) & \geq \frac{\eta^2}{\beta}+\frac{\eta^2}{\beta}\sum_{n\geq 1} m_2^n L_{\s}^n \big[\Upsilon_{\pmb a}(\beta)\big]^n \quad \mbox{for the heat equation},\\
H_{\beta}(x,x) & \geq \frac{\eta^2}{\beta}+\frac{\eta^2}{\beta}\sum_{n\geq 1} m_2^n L_{\s}^n \left(\frac{1}{2\beta}\right)^n \Big[\Upsilon_{\pmb a}\Big(\frac{\beta^2}{4}\Big)\Big]^n \quad \mbox{for the wave equation}.
\end{align*}

Note that $\beta \mapsto \Upsilon_{\pmb a}(\beta)$ is a non-increasing function on $(0,\infty)$, $\lim_{\beta \to \infty}\Upsilon_{\pmb a}(\beta)=0$ and
\[
\lim_{\beta \to 0+}\Upsilon_{\pmb a}(\beta)=\Upsilon_{\pmb a}(0).
\]

We choose $\beta>0$ small enough, such that 
\begin{align*}
\Upsilon_{\pmb a}(\beta) &\geq \frac{1}{m_2 L_{\s}^2} \quad \mbox{for the heat equation}, \\
\frac{1}{2\beta} \Upsilon_{\pmb a}(\beta) & \geq \frac{1}{m_2 L_{\s}^2}  \quad \mbox{for the wave equation}.
\end{align*}

With this choice of $\beta$, $H_{\beta}(x,x)=\infty$. The argument starting two lines after (5.64) of \cite{FK13} shows that $\overline{\gamma}(2) \geq \beta$.
\end{proof}

We are now ready to state the result about the weak intermittency of the solution.

\begin{corollary}
\label{interm1}
If $m_p<\infty$ for all $p\geq 2$, and $\sigma$ satisfies condition \eqref{sigma-cond}, then the solution of \eqref{nonlin-no-drift} with constant initial condition $\eta \in \bR$ is weakly intermittent, in the following cases:\\
(i) $\cL$ is the heat operator, $d=1$ and $\k=H_{1,\alpha/2}$ for some $\alpha>0$;\\
(ii) $\cL$ is the wave operator and either one of the following conditions hold:\\
a) $d \in \{1,2\}$ and $\k=H_{d,\alpha/2}$ for some $\alpha>0$;\\
b) $d=1$ and $\k=R_{1,\alpha/2}$ for some $\alpha \in (0,1)$; \\
c) $d=1$ and $\k=B_{1,\alpha}$ for some $\alpha >0$.
\end{corollary}

\begin{proof}
The cases listed in this corollary are those mentioned in Examples \ref{heat-ex-p} and \ref{wave-ex-p} for which \eqref{Lyap} holds for all $p\geq 2$. It is clear that in all these cases, conditions a) and b) of Theorem \ref{th-Lyap2} are satisfied. 
It remains to show that condition c) also holds. We recall definition \eqref{def-mu} of $\mu$, and the formulas for $|\cF \k|^2$ given in Examples \ref{heat-ex}, \ref{Riesz-ex} and \ref{Bessel-ex}. 
\[
(2\pi)^d \Upsilon_{\pmb a}(0)=
\left\{
\begin{array}{ll}\int_{[{\pmb a},2\pmb{a}]}|\xi|^{-2} \exp(-\frac{\alpha |\xi|^2}{2}) d\xi & \mbox{if $\k=H_{d,\alpha/2}$ for some $\alpha>0$}, \\
\int_{[\pmb{a},2\pmb{a}]}|\xi|^{-2-\alpha} d\xi & \mbox{if $\k=R_{d,\alpha/2}$ for some $\alpha\in(0,d)$}, \\
\int_{[\pmb{a},2\pmb{a}]}|\xi|^{-2} (1+|\xi|^2)^{-\alpha/2}  & \mbox{if $\k=B_{d,\alpha/2}$ for some $\alpha>0$}
\end{array} \right.
\]
In all these cases, $\Upsilon_{\pmb a}(0)>0$ for all $\pmb{a} \in \bR_{+}^d$. 

In the case of the heat equation, we have to show that \eqref{upsilon-cond} holds. If $\k=H_{d,\alpha/2}$, then
\[
(2\pi)^d \Upsilon_{\pmb a}(0) =\int_{[{\pmb a},2\pmb{a}]} \frac{1}{|\xi|^2} e^{-\frac{\alpha |\xi|^2}{2}} d\xi \geq  \frac{(2\pi)^d}{4|{\pmb a}|^2} e^{-2\alpha |{\pmb a}|^2} \prod_{k=1}^d a_k.
\]

If $a_1=\ldots=a_d=a$, the lower bound above is equal to $Q(a):=\frac{(2\pi)^d}{4d} a^{d-2} e^{-2\alpha da^2}$, which converges to $\infty$ when $a \to 0$, 
{\em provided that} $d=1$. Therefore, we can choose $a$ small enough such that $Q(a)>
(2\pi)^{d}(m_2 L_{\sigma}^2)^{-1}$, and therefore \eqref{upsilon-cond} holds.

\end{proof}

\section{The parabolic and hyperbolic Anderson models}
\label{section-pham}

In this section, we study the equation 
\begin{equation}
\label{pham}
\cL u(t,x)=\lambda u(t,x) \dot{X}(t,x), \quad t\geq 0,x \in \bR^d,
\end{equation}
with constant initial condition $\eta \in \bR$, where $\cL$ is the heat operator in dimension $d\geq 1$ or the wave operator in dimension $d\leq 2$. We assume that $\lambda >0$. When  $\cL$ is the heat operator,
\eqref{pham} is called the {\em parabolic Anderson model}.  When  $\cL$ is the wave operator,
\eqref{pham} is called the {\em hyperbolic Anderson model}. 
Using tools from Malliavin calculus, we will provide the chaos expansion of the solution, and under some conditions, an explicit expression for the upper Lyapunov exponent $\overline{\gamma}(2)$.

\medskip

We recall that the solution of \eqref{pham} satisfies the integral equation:
\begin{equation}
\label{pham-eq}
u(t,x)=\eta+\lambda \int_0^t \int_{\bR^d} G_{t-s}(x-y) u(s,y) X(ds,dy).
\end{equation}

\medskip

We will first show that $u(t,x)$ has an explicit series representation.  
For this, we need to recall some basic elements of Malliavin calculus on the Poisson space. We refer the reader to \cite{last-penrose18} for more details. We consider the Hilbert space
\[
\fH=L^2({\bf Z},\cZ,\fm),
\]
 where ${\bf Z}=\bR_{+} \times \bR^d \times \bR_0$, $\cZ=\cB(Z)$, $\fm={\rm Leb} \times \nu$, and ${\rm Leb}$ is the Lebesgue measure on $\bR_{+} \times \bR^d$.

It is known that any random variable $F \in L^2(\Omega)$ which is measurable with respect to $N$ has the (unique) chaos expansion
\begin{equation}
\label{chaos-F}
F=\bE(F)+\sum_{n\geq 1}I_n(f_n) \quad \mbox{for some} \quad f_n \in \fH^{\odot n},
\end{equation}
where $I_n$ is the multiple integral of order $n$ with respect to $\widehat{N}$, and $\fH^{\odot n}$ is the set of symmetric functions in $\fH^{\otimes n}=L^2(Z^n,\cZ^n, \fm^n)$. The terms in this representation are orthogonal in $L^2(\Omega)$. Moreover, 
\[
\bE|I_n(f)|^2=n! \|f\|_{\fH^{\otimes n}}^2 \quad \mbox{for all $f \in \fH^{\odot n}$}.  
\]

To identify the chaos expansion of the solution $u(t,x)$ of equation \eqref{pham}, we proceed first formally, then we provide a rigorous proof.

For the formal argument, on the right hand-side of \eqref{pham-eq}, we express $u(s,y)$ using the integral form:
\[
u(s,y)=\eta+\lambda \int_0^s \int_{\bR^d}G_{s-r}(y-z)u(r,z) X(dr,dz)
\]
then we express $u(r,z)$ using the same integral form, and we iterate this procedure. We obtain:
\begin{align*}
u(t,x)& =\eta+\eta \lambda \int_0^t \int_{\bR^d}G_{t-s}(x-y)X(ds,dy)+\\
&  \quad \eta \lambda^2 \int_{0<t_1<t_2<t}\int_{(\bR^d)^2} G_{t-t_2}(x-x_2)G_{t_2-t_1}(x_2-x_1)X(dt_1,dx_1)X(dt_2,dx_2)+\ldots
\end{align*}

Using \eqref{g-XL}, we can pass to an integral with respect to $L$, and then to an integral with respect to $\widehat{N}$ (using \eqref{Poisson-rep}):
\[
\int_0^t \int_{\bR^d}G_{t-s}(x-y)X(ds,dy)=\int_0^t \int_{\bR^d}\int_{\bR_0}\big(G_{t-s}(x-\cdot)* \k\big)(y)z \widehat{N}(ds,dy,dz)
\]
Formally, we can do the same for the second integral:
\begin{align*}
& \int_{0<t_1<t_2<t}\int_{(\bR^d)^2} G_{t-t_2}(x-x_2)G_{t_2-t_1}(x_2-x_1)X(dt_1,dx_1)X(dt_2,dx_2)=\\
& \int_{0<t_1<t_2<t} \int_{(\bR^d)^2} \int_{\bR_0^2} \big(G_{t-t_2}(x-\cdot)G_{t_2-t_1}(x_2-\cdot)* \k^{\otimes 2}\big)(x_1,x_2) z_1 z_2 \widehat{N}(dt_1,dx_1,dz_1,dt_2,dx_2,dz_2).
\end{align*}

This leads us to the following general notation: for any $n\geq 1$, we let
\begin{align}
\label{def-fn-star}
f_n^*(t_1,x_1,z_1,\ldots,t_n,x_n,z_n,t,x)&:=\big(f_n(t_1,\cdot,\ldots,t_n,\cdot,t,x)* \k^{\otimes n}\big)(x_1,\ldots,x_n)z_1 \ldots z_n,
\end{align}
where $\k^{\otimes n}(x_1,\ldots,x_n)=\prod_{i=1}^n\k(x_i)$, and
\[
f_n(t_1,x_1,\ldots,t_n,x_n,t,x)=\eta \lambda^n  G_{t-t_n}(x-x_n)G_{t_n-t_{n-1}}(x_n-x_{n-1})\ldots G_{t_2-t_1}(x_2-x_1).
\]
Recalling convention \eqref{G-convention}, we see that $f_n(\cdot,t,x)$ contains the indicator of the simplex 
\[
T_n(t)=\{(t_1,\ldots,t_n)\in [0,t]^n; 0<t_1<\ldots<t_n<t\}.
\]

To give the rigorous proof for the chaos expansion of the solution, we need to introduce few more facts from Malliavin calculus.

Let ${\rm dom}(D)$ denote the set of random variables $F$ as in \eqref{chaos-F} with symmetric kernels $(f_n)_{n\geq 1}$ satisfying
\[
\sum_{n\geq 1}n! n \|f_n\|_{\fH^{\otimes n}}^2<\infty.
\]

For any $F \in {\rm dom}(D)$, the {\em Malliavin derivative}
$DF$ is an $\fH$-valued random variable defined as follows: for fixed $\xi=(t,x,z) \in {\bf Z}$,
\begin{align*}
D_{\xi} F = \sum_{n=1}^\infty n I_{n-1}(f_n(\bullet,\xi)).
\end{align*}
The {\em Skorohod integral} is the adjoint of $D$, characterized by the duality relation: 
\[
\bE[F \delta(V)]=\bE[\langle DF, V \rangle_{\fH}] \quad \mbox{for all $F \in {\rm dom}(D)$}.
\]

The following results are taken from \cite{BZ25}. See also Propositions 1.3.7 and 1.3.11 of \cite{nualart06} for the Gaussian white noise case.

\begin{lemma}[Lemma 2.12 of \cite{BZ25}] 
\label{lem212}
Let $V \in L^2(\Omega;\fH)$ with $
V(\xi)=h_0(\xi)+\sum_{n\geq 0}I_n\big(h_n(\bullet,\xi)\big)$, for all $\xi \in {\bf Z}$, where
$h_0(\xi)=\bE[V(\xi)]$ and $h_n(\bullet,\xi)\in \fH^{\odot n}$.
Then $V \in {\rm dom}(\delta)$ if and only if 
\[
\sum_{n\geq 0}(n+1)! \|\widetilde{h}_n\|_{\fH^{\otimes(n+1)}}^2<\infty,
\]
where $\widetilde{h}_n$ is the symmetrization of $h_n$ in all its $n+1$ variables:
\[
\widetilde{h}_{n+1}(\xi_1,\ldots,\xi_n,\xi_{n+1})=\frac{1}{n+1} \sum_{i=1}^{n}h_n(\xi_1,\ldots,\xi_{i-1},\xi_{n+1}, \xi_{i+1},\ldots, \xi_n).
\]
In this case, $\delta(V)=\sum_{n\geq 0}I_{n+1}(\widetilde{h}_n)$. 
\end{lemma}

A process $V=\{V(t,x,z);t\geq 0, x \in \bR^d, z \in \bR_0\}$ is {\em predictable} it is measurable with respect to the predictable $\sigma$-field on $\Omega \times \bR_{+} \times \bR^d \times \bR_0$, which is the $\sigma$-field generated by elementary processes of the form
\[
g(\omega,t,x,z)=Y(\omega)1_{(a,b]}(t)1_{A}(x) 1_{B}(z),
\]
where $0\leq a<b$, $A \in \cB_b(\bR^d)$, $B \in \cB_b(\bR_0)$ and $Y$ is a bounded and $\cF_a$-measurable random variable.

\begin{lemma}[Lemma 2.9 of \cite{BZ25}]
\label{lem29}
If $V \in L^2(\Omega;\fH)$ is predictable, then $V \in {\rm dom}(\delta)$ and $\delta(V)$ coincides with the It\^o integral with respect to $\widehat{N}$.
\end{lemma} 

After these preliminaries, we are now ready to prove the following result.

\begin{theorem}
\label{chaos-u-th}
Suppose that the measure $\mu$ satisfies Dalang's condition \eqref{Dalang-cond}.
Let $u$ be the solution of equation \eqref{pham} with constant initial condition $\eta \in \bR$ and $\lambda \in \bR \verb2\2 \{0\}$, where $\cL$ is the heat operator in dimension $d\geq 1$, or the wave operator in dimension $d\leq 2$. Assume that for any $t>0$ and $x \in \bR^d$, $u(t,x)$ has the Poisson chaos expansion:
\begin{equation}
\label{chaos-u}
u(t,x)=\eta+\sum_{n\geq 1}I_n\big( k_n(\bullet,t,x)\big),
\end{equation}
for some kernels $k_n(\bullet,t,x)\in \fH^{\odot n}$. Then 
$k_n(\bullet,t,x)=f_n^{*}(\bullet,t,x)$ for all $t>0$, $x \in \bR^d$, and $n \geq 1$,
where $f_n^*(\bullet,t,x)$ is given by \eqref{def-fn-star}.
\end{theorem}

\begin{proof}
We proceed as in the Gaussian case (see e.g. pages 302-303 of \cite{hu-nualart09}). 

We fix $t>0$ and $x \in \bR^d$. Since $u$ satisfies \eqref{u-2nd-mom}, the process $(s,y) \mapsto v^{(t,x)}(s,y)=G_{t-s}(x-y)u(s,y)$ with $s \in [0,t]$ and $y \in \bR^d$, lies in $ \in \cP_{+}$. Hence, by \eqref{g-XL} and \eqref{Ito-LN},
\[
u(t,x)=1+\int_0^t \int_{\bR^d} \int_{\bR_0} \big(v^{(t,x)}(s,\cdot)*\k\big)(y) z \widehat{N}(ds,dy,dz).
\]
In particular, this implies that the process $(s,y,z) \mapsto V^{(t,x)}(s,y,z)= \big(v^{(t,x)}(s,\cdot)*\k\big)(y) z $ is It\^o integrable with respect to $\widehat{N}$. Since $V^{(t,x)}$ is predictable, by Lemma \ref{lem29}, $V^{(t,x)} \in {\rm dom}(\delta)$ and its It\^o integral coincides with $\delta(V^{(t,x)})$. Therefore,
\begin{equation}
\label{delta-V}
u(t,x)=1+\delta(V^{(t,x)}).
\end{equation}

It is not difficult to find the chaos expansion of $V^{(t,x)}$.
For this, we write \eqref{chaos-u} for $u(s,y)$, and we multiply this relation by $G_{t-s}(x-y)$ to obtain the chaos expansion:
\[
v^{(t,x)}(s,y)=G_{t-s}(x-y)+\sum_{n\geq 1} I_n\big(G_{t-s}(x-y)k_n(\bullet,s,y)\big).
\]
Taking the convolution in the space variable with $\k$, we see that
\[
V^{(t,x)}(s,y,z)=\sum_{n\geq 0}I_n(g_n^{(t,x)}(\bullet,s,y,z)),
\]
where $g_0^{(t,x)}(s,y,z)=\big(G_{t-s}(x-\cdot) *\k\big)(y) z$, and
\begin{align*}
g_{n}^{(t,x)}(\bullet,s,y,z) &= z \int_{\bR^d} G_{t-s}(x-y')k_n(\bullet,s,y') \k(y-y')dy' \quad \mbox{for all $n\geq 1$}.
\end{align*}

By Lemma \ref{lem212}, $\delta(V^{(t,x)})=\sum_{n\geq 0} I_{n+1}(\widetilde{g_{n}^{(t,x)}} )$, where
\[
\widetilde{g_{n}^{(t,x)}}(\xi_1,\ldots,\xi_n,\xi_{n+1})=\frac{1}{n+1}\sum_{i=1}^{n+1} g_n^{(t,x)}(\xi_1,\ldots,\xi_{i-1},\xi_{n+1},\xi_{i+1},\ldots,\xi_n).
\]
In summary, we obtain that
\[
1+\sum_{n\geq 1}I_n\big(k_n(\bullet,t,x)\big)=u(t,x)=1+\sum_{n\geq 0} I_{n+1}(\widetilde{g_{n}^{(t,x)}} ).
\]
By uniqueness, since both kernels $k_n(\bullet,t,x)$ and $\widetilde{g_{n}^{(t,x)}}$ are {\em symmetric}, we must have:
\[
k_{n+1}(\bullet,t,x)=\widetilde{g_{n}^{(t,x)}} \quad \mbox{for all $n\geq 0$}.
\]
This allows us to find recursively the kernels $k_{n+1}(\bullet,t,x)$, starting with $n=0$.
\end{proof}

The next result gives the explicit expression for the second moment of the solution of the parabolic/hyperbolic Anderson model \eqref{pham}, driven by the L\'evy colored noise $X$.

\begin{corollary}
\label{2-mom-u}

Under the conditions of Theorem \ref{chaos-u-th},
\begin{align*}
\bE|u(t,x)|^2&= \eta^2 \left(1+\sum_{n\geq 1}m_2^n \lambda^{2n} \int_{T_n(t)} \int_{(\bR^d)^n}
\prod_{j=1}^n |\cF G_{t_{j+1}-t_j}(\xi_1+\ldots+\xi_j)|^2\mu(d\xi_1)\ldots\mu(d\xi_n) d\pmb{t_n} \right),
\end{align*}
where $\pmb{t_n}=(t_1,\ldots,t_n)$.
\end{corollary}

\begin{proof}
From the chaos expansion $u(t,x)=\eta+\sum_{n\geq 1}I_n\big( f_n^*(\bullet,t,x)\big)$ 
and the orthogonality of the Poisson chaos spaces, we infer that
\begin{align*}
\bE|u(t,x)|^2&=\eta^2 +\sum_{n\geq 1} n! \|\widetilde{f}_n^*(\bullet,t,x)\|_{\fH^{\otimes n}}^2,
\end{align*}
where $\widetilde{f}_n^*(\bullet,t,x)$ is the symmetrization of $f_n^*(\bullet,t,x)$ in the first $n$ variables. Since the function $f_n(\bullet,t,x)$ contains the indicator of the simplex $T_n(t)$, we have:
\[
 n! \|\widetilde{f}_n^*(\bullet,t,x)\|_{\fH^{\otimes n}}^2 =m_2^n \int_{T_n(t)}\int_{(\bR^d)^n}|\cF \big(f_n(t_1,\cdot,\ldots,t_n,\cdot)* \k^{\otimes n}\big)(\pmb{x_n})|^2 d\pmb{x_n} d\pmb{t_n},
\]
where $\pmb{x_n}=(x_1,\ldots,x_n)$. The conclusion now follows passing to Fourier transforms. For this, we use the same argument as in the Gaussian case (see e.g. Lemma 2.1 of \cite{BS19}), and we recall definition \eqref{def-mu} of the measure $\mu$ in our case.
\end{proof}

The previous corollary allows us to deduce the following result.

\begin{corollary}
Suppose that the measure $\mu$ satisfies Dalang's condition \eqref{Dalang-cond}. Let $u$ be the solution of equation \eqref{pham} with constant initial condition $\eta \in \bR$ and $\lambda \in \bR \verb2\2\{0\}$. Assume that conditions a), b), c) in Theorem \ref{th-Lyap2} hold, with $L_{\sigma}^2$ replaced by $\lambda^2$ in condition c). Then 
\[
\overline{\gamma}(2)>0.
\]
\end{corollary}

\begin{proof}
Recall the definition \eqref{def-H-beta} of $H_{\beta}(x,y)$. In the case $\sigma(u)=\lambda u$, due to the explicit formula for $\bE|u(t,x)|^2$ given by Corollary \ref{2-mom-u}, we have: (see Lemma 2.3 of \cite{BS19})
\begin{align*}
H_{\beta}(x,x)&=\frac{\eta^2}{\beta}+\frac{\eta^2}{\beta}\sum_{n \geq 1} m_2^{n} \lambda^{2n} \int_{(\bR^d)^n} \prod_{j=1}^{n}\cI_{\beta}(\xi_1+\ldots+\xi_j) \mu(d\xi_1)\ldots \mu(d\xi_n).
\end{align*}
This is precisely relation \eqref{H-beta-xx} in which $L_{\sigma}^2$ is replaced by $\lambda^2$, and the $\geq$ sign is replaced by $=$. The conclusion follows using the same argument as in the proof of Theorem \ref{th-Lyap2}. 
\end{proof}

\begin{corollary}
\label{interm2}
If $m_p<\infty$ for all $p\geq 2$, then the solution of \eqref{pham} with constant initial condition $\eta \in \bR$ and $\lambda >0$ is weakly intermittent in all the cases mentioned in Corollary \ref{interm1}.
\end{corollary}

Corollary \ref{2-mom-u} also allows us to perform a direct comparison with the Gaussian case, which in particular will lead to the explicit expression of $\overline{\gamma}(2)$.

More precisely, let $W$ be the spatially-homogeneous Gaussian noise from Dalang's article \cite{dalang99}, i.e. $W=\{W(\varphi);\varphi \in \cD(\bR_{+} \times \bR^d)\}$ is a zero-mean Gaussian process with covariance:
\begin{equation}
\label{Gauss-cov}
\bE[W(\varphi)W(\psi)]=\int_{0}^{\infty}\int_{\bR^d} \cF \varphi(t,\cdot)(\xi) \overline{\cF \psi(t,\cdot)(\xi)}\mu(d\xi) dt,
\end{equation}
where $\mu$ is given by \eqref{def-mu}.
For any $\theta >0$, let $u_{\theta}$ be the solution of the equation:
\begin{equation}
\label{pham-Gauss}
\cL u =\theta u \dot{W}, \quad t>0, x \in \bR^d,
\end{equation}
with constant initial condition $\eta \in \bR$, where $\cL$ is the heat operator in dimension $d\geq 1$, or the wave operator in dimension $d\leq 2$. By Corollary \ref{2-mom-u},
\begin{equation}
\label{Levy-Gauss}
\bE|u(t,x)|^2=\bE|u_{\theta}(t,x)|^2 \quad \mbox{with $\theta=\sqrt{m_2}\lambda$}.
\end{equation}
Therefore, it suffices to derive the asymptotic behavior of $\bE|u_{\theta}(t,x)|^2$ for fixed $\theta>0$.

\medskip

The following result was proved in \cite{HLN17} for the heat equation, respectively \cite{BS19} for the wave equation, in the case $\theta=1$. We provide the extension to general $\theta$.

For this result, we need to introduce the following variational constants:
\begin{align*}
\cE(f)&=\sup_{g \in \cF_d} \left\{ \int_{\bR^d} f(x) g^2(x)dx-\frac{1}{2}\int_{\bR^d} |\nabla f(x)|^2 dx\right\},\\
\cE_2(f)&=\sup_{g \in \cF_{2d}} \left\{ \int_{(\bR^d)^2} f(x) g^2(x_1,x_2)dx_1dx_2-\frac{1}{2}\int_{(\bR^d)^2} |\nabla f(x_1,x_2)|^2 dx_1 dx_2\right\},
\end{align*}
where $\cF_d=\{g \in H^1(\bR^d);\int_{\bR^d}g^2(x)dx=1\}$ and $H^1(\bR^d)$ is the Sobolev space of order 1.

\begin{theorem}
\label{Gauss-th}
Suppose that $\mu$ satisfies Dalang's condition \eqref{Dalang-cond} and $f$ satisfies the scaling property:
\begin{equation}
\label{scaling-f}
f(cx)=c^{-\alpha}f(x) \quad \mbox{for any $c>0$ and $x \in \bR^d$},
\end{equation}
for some $\alpha>0$. For any $\theta>0$, let $u_{\theta}$ be the solution of equation \eqref{pham-Gauss}. Then, 
\begin{align*}
\lim_{t\to \infty}\frac{1}{t}\log \bE|u_{\theta}(t,x)|^2 & =\theta^{\frac{2}{2-\alpha}} 2^{-\frac{\alpha}{2-\alpha}}  \cE(f) \quad \mbox{for the heat equation},\\
\limsup_{t\to \infty}\frac{1}{t}\log \bE|u_{\theta}(t,x)|^2 & =\theta^{\frac{1}{3-\alpha}}  2^{\frac{2-3\alpha}{6-2\alpha}} \big(\cE(f)\big)^{\frac{2-\alpha}{6-2\alpha}} \quad \mbox{for the wave equation}.
\end{align*}
\end{theorem}

\begin{proof}
The limit or limsup are not affected by the choice of $\eta$, so we can take $\eta=1$. Let $v_{\theta}$ be the solution of equation
$\cL v =v W_{\theta}$ with $t>0$, $x\in \bR^d$, and
initial condition 1, where $W_{\theta}$ is a spatially-homogeneous Gaussian noise with covariance \eqref{Gauss-cov}, in which the function $f$ is replaced by $\theta f$. Let $\cH=L^2(\bR_{+};\cP_{0,d})$. Then,
\[
\bE|u_{\theta}(t,x)|^2 =1+\sum_{n\geq 1} \theta^{2n}n! \|\widetilde{f}_n(\cdot,t,x)\|_{\cH^{\otimes n}}^2=\bE|v_{\theta}(t,x)|^2.
\]
For the heat equation, by Theorem 1.3 of \cite{HLN17} and  Lemma B.2 of \cite{BS19} ,
\[
\lim_{t \to \infty} \frac{1}{t} \log \bE|v_{\theta}(t,x)|^2 =\cE_2(\theta f) =\theta^{\frac{2}{2-\alpha}} 2^{-\frac{\alpha}{2-\alpha}}\cE(f),
\]
For the wave equation, by Theorem 1.1 of \cite{BS19} and Lemma B.2 ibid,
\[
\limsup_{t \to \infty} \frac{1}{t} \log \bE|v_{\theta}(t,x)|^2 =2^{\frac{2-3\alpha}{6-2\alpha}} \big(\cE(\theta f)\big)^{\frac{2-\alpha}{6-2\alpha}}=2^{\frac{2-3\alpha}{6-2\alpha}} \theta^{\frac{1}{3-\alpha}} \big(\cE(f)\big)^{\frac{2-\alpha}{6-2\alpha}}.
\]
\end{proof}

As a consequence of relation \eqref{Levy-Gauss} and Theorem \ref{Gauss-th}, we obtain the following result.

\begin{corollary}
Let $u$ be the solution of equation \eqref{pham} with constant initial condition $\eta \in 
\bR$ and $\lambda>0$. Assume that $\mu$ satisfies Dalang's condition \eqref{Dalang-cond} and $\k$ satisfies: 
\[
\k(cx)=c^{-(d-\frac{\alpha}{2})}\k(x) \quad \mbox{for all $c>0$ and $x \in \bR^d$},
\]
for some $\alpha \in (0,d)$. Let $\alpha'=d-\alpha$, we have:
\begin{align*}
\lim_{t \to \infty}\frac{1}{t}\log \bE|u(t,x)|^2&=(\sqrt{m_2}\lambda)^{\frac{2}{2-\alpha'}} 2^{-\frac{\alpha'}{2-\alpha'}}\cE(f) \quad \mbox{for the heat equation},\\
\limsup_{t \to \infty}\frac{1}{t}\log \bE|u(t,x)|^2&=(\sqrt{m_2}\lambda)^{\frac{1}{3-\alpha'}} 2^{\frac{2-3\alpha'}{6-2\alpha'}}\big(\cE(f)\big)^{\frac{2-\alpha'}{6-2\alpha'}} \quad \mbox{for the wave equation}.
\end{align*}
\end{corollary}

\begin{proof}
Note that $h=\cF \k$ satisfies the scaling property $h(cx)=c^{-\alpha/2}h(x)$. Hence, $g:=\cF f=|\cF \k|^2$ satisfies $g(cx)=c^{-\alpha}g(x)$, which means that $f(cx)=c^{-(d-\alpha)}f(x)$. The conclusion follows by applying Theorem \ref{Gauss-th} with $\alpha'=d-\alpha$.
\end{proof}

In the case of the Riesz kernel, we have an alternative representation of the second-order Lyapounov exponent of the solution $u_{\theta}$.

\begin{theorem}
Let $f(x)=R_{d,\alpha}(x)=C_{d,\alpha}|x|^{-(d-\alpha)}$ be the Riesz kernel of order $\alpha \in (0,d)$, as defined in Example \ref{Riesz-ex}. For any $\theta >0$, let $u_{\theta}$ be the solution of equation \eqref{pham-Gauss}. Then, 
\[
\lim_{t \to \infty}\frac{1}{t}\log \bE|u_{\theta}(t,x)|^2 =
\left\{
\begin{array}{ll} 
\big(\theta \rho \big)^{\frac{2}{2-\alpha'}} & \mbox{for the heat equation}, \\
\big(2^{1-\alpha'} \theta \rho\big)^{\frac{1}{3-\alpha'}} & \mbox{for the wave equation},
\end{array} \right.
\]
where $\alpha'=d-\alpha$ and
\begin{equation}
\label{def-rho}
\rho=\frac{1}{(2\pi)^d}\sup_{\|g\|_2=1} \int_{(\bR^d)^2} \frac{g(\xi)}{\sqrt{1+|\xi|^2}}
\frac{g(\eta)}{\sqrt{1+|\eta|^2}}|\xi-\eta|^{-\alpha}d\xi d\eta.
\end{equation}
\end{theorem}

\begin{proof}
In the case $\theta=1$, this follows by Theorem 1.3 of \cite{BS19}. The general case follows by considering the relation with the solution $v_{\theta}$, as in the proof of Theorem \ref{Gauss-th}.
\end{proof}

Consequently, we have the following result.

\begin{corollary}
Let $\k=R_{d,\alpha/2}$ be the Riesz kernel given in Example \ref{Riesz-ex}, for some $\alpha \in (0,2)$. Let $u$ is the solution of equation \eqref{pham} with constant initial condition $\eta \in 
\bR$ and $\lambda>0$, where $\cL$ is the heat operator in dimension $d\geq 1$ or the wave operator in dimension $d\leq 2$. Then,
\[
\lim_{t \to \infty}\frac{1}{t}\log \bE|u(t,x)|^2 =
\left\{
\begin{array}{ll} 
\big(\sqrt{m_2}\lambda \rho \big)^{\frac{2}{2-\alpha'}} & \mbox{for the heat equation}, \\
\big(2^{1-\alpha'} \sqrt{m_2}\lambda \rho\big)^{\frac{1}{3-\alpha'}} & \mbox{for the wave equation},
\end{array} \right.
\]
where $\alpha'=d-\alpha$, and $\rho$ is defined by \eqref{def-rho}.
\end{corollary}

\appendix
\section{Some auxiliary results}

\begin{lemma}
\label{Fourier-lem}
Let $f\in \cS'(\bR^d)$ be such that $\cF f=h$ in $\cS_{\bC}'(\bR^d)$ for some (tempered $\bC$-valued) function $h$. If $h \in L_{\bC}^2(\bR^d)$, then $f \in L^2(\bR^d)$ and $\cF f=h$ in $L_{\bC}^2(\bR^d)$.
\end{lemma}

\begin{proof}
We know that $(f,\cF \phi)=(\cF f,\phi)=(h,\phi)$ for all $\phi \in \cS_{\bC}(\bR^d)$, or equivalently,
\begin{equation}
\label{f-phi}
(f,\varphi)=\frac{1}{(2\pi)^d}(h,\overline{\cF \varphi}) \quad \mbox{for all}  \quad \varphi \in \cS_{\bC}(\bR^d).
\end{equation}
Since $h \in L_{\bC}^2(\bR^d)$ and $\cF : L_{\bC}^2(\bR^d) \to L_{\bC}^2(\bR^d)$ is a one-to-one map, there exists a function $g \in L_{\bC}^2(\bR^d)$ such that $\cF g=h$. This means that there exists a sequence $(g_n)_n$ in $\cS_{\bC}(\bR^d)$ such that $g_n \to g$ in $L_{\bC}^2(\bR^d)$ and $\cF g_n \to h$ in $L_{\bC}^2(\bR^d)$. By Plancherel theorem and \eqref{f-phi}, 
\[
(g,\varphi)=\frac{1}{(2\pi)^d}(\cF g,\overline{\cF \varphi})=
\frac{1}{(2\pi)^d}(h,\overline{\cF \varphi})=(f,\varphi),
\quad \mbox{for any $\varphi \in \cS_{\bC}(\bR^d)$}.
\] 
It follows that $f=g$ a.e., $f \in L^2(\bR^d)$ and $h=\cF f$ in $L_{\bC}^2(\bR^d)$.
\end{proof} 

The following result gives a useful representation of the predictable $\sigma$-field
$\cP_{\Omega \times \bR_{+} \times \bR^d}$. We include its proof since we could not find it in the literature.

\begin{lemma}
\label{lem-predict}
Let $\cP_{\Omega \times \bR_{+}}$ and $\cP_{\Omega \times \bR_{+} \times \bR^d}$ be the predictable $\sigma$-fields on $\Omega \times \bR_{+}$, respectively $\Omega \times \bR_{+}\times \bR^d$. Then 
\[
\cP_{\Omega \times \bR_{+} \times \bR^d}=\cP_{\Omega \times \bR_{+}} \otimes \cB(\bR^d).
\]
\end{lemma}

\begin{proof}
We denote $\cP=\cP_{\Omega \times \bR_{+}}$ and $\widetilde{\cP}=\cP_{\Omega \times \bR_{+}\times \bR^d}$.

(a) We first prove that any process $X \in \cE$ is $\cP \otimes \cB(\bR^d)$-measurable. This will imply that $\widetilde{\cP} \subset \cP \otimes \cB(\bR^d)$, since $\widetilde{\cP}$ is the minimal $\sigma$-field with respect to which all processes in $\cE$ are measurable.

Let $X \in \cE$ be arbitrary. Without loss of generality, we may assume that $X$ is of the form \eqref{elem}. Since $(\omega,t) \mapsto Y(\omega)1_{(a,b]}(t)$ is $\cP$-measurable, and $x \mapsto 1_{A}(x) $ is $\cB(\bR^d)$-measurable, it follows that $(\omega,t,x) \mapsto Y(\omega)1_{(a,b]}(t) 1_{A}(x)$ is $\cP \otimes \cB(\bR^d)$-measurable.

(b) For the reverse inclusion, it is enough to prove that $\cR \subset \widetilde{\cP}$, where $\cR$ is the set of rectangles of the form $F \times B$, with $F \in \cP$ and $B \in \cB(\bR^d)$.
For this, we will use a $\lambda-\pi$-class argument. More precisely,
let $\cL$ be the class of all sets $F \in \cP$ such that $F \times B \in \widetilde{\cP}$ for all $B \in \cB(\bR^d)$. Note that $\cL$ is a $\sigma$-field. 

We need to prove that $\cL=\cP$. For this, we observe that $\cP=\sigma(\cA)$ where $\cA$ is the $\pi$-system consisting of sets of the form 
\begin{equation}
\label{set-F}
F=\bigcap_{i=1}^k F_i, \quad \mbox{with} \quad F_i=\{(\omega,t) \in \Omega \times \bR_{+};X_i(\omega,t) \in A_i\},
\end{equation}
where $k\geq 1$, $X_1,\ldots,X_k$ are simple processes, and $A_1,\ldots, A_k \in \cB(\bR)$. It is enough to prove that 
\begin{equation}
\label{A-L}
\cA \subset \cL.
\end{equation}
This will imply that $\cP=\sigma(A) \subset \cL$, and hence $\cP=\cL$. To prove \eqref{A-L}, let $F$ be a set of the form \eqref{set-F}. We have to prove that $F \times B \in \widetilde{\cP}$ for all $B \in \cB(\bR^d)$. Let $B \in \cB(\bR^d)$ be arbitrary. Then
\[
F \times B =\bigcap_{i=1}^{k}(F_i \times B)=\bigcap_{i=1}^{k}\{(\omega,t,x); X_i(\omega,t)\in A_i,x \in B\}.
\]
Note that
\[
\{(\omega,t,x); X_i(\omega,t)1_{B}(x)\in A_i\}=\{(\omega,t,x); X_i(\omega,t)\in A_i,x\in B\} \cup \{(\omega,t,x); 0\in A_i,x\in B^c\} .
\]
The last set is $\Omega \times \bR_{+} \times B^c$ if $0\in A_i$, or $\emptyset$ if $0 \not\in A_i$. Hence
\[
\{(\omega,t,x); X_i(\omega,t)\in A_i,x \in B\}=
\left\{
\begin{array}{ll} 
\{(\omega,t,x); X_i(\omega,t)1_{B}(x)\in A_i\} \verb2\2 (\Omega \times \bR_{+} \times B^c) & \mbox{if $0\in A_i$} \\
\{(\omega,t,x); X_i(\omega,t)1_{B}(x)\in A_i\} & \mbox{if $0\not \in A_i$}
\end{array} \right.
\]
In both cases, $\{(\omega,t,x); X_i(\omega,t)\in A_i,x \in B\}\in \widetilde{\cP}$. Hence $F \times B \in \widetilde{\cP}$. 
\end{proof}

{\bf Acknowledgement.} The authors are grateful to Mahmoud Zarepour for drawing their attention to reference \cite{MCC98} about the variance gamma process, which lead to Example \ref{ex-VG}. The authors would like to thank Markus Riedle for asking us to clarify the relation between the predictable $\sigma$-fields on $\Omega \times \bR_{+}$ and $\Omega \times \bR_{+} \times \bR^d$, as stated in Lemma \ref{lem-predict}.

\end{document}